\documentclass{article}
\usepackage{amsfonts}
\usepackage{amsmath}
\usepackage{amssymb}

\newtheorem{theorem}{Theorem}

\newtheorem{corollary}[theorem]{Corollary}

\newtheorem{lemma}[theorem]{Lemma}

\newtheorem{proposition}[theorem]{Proposition}
\newtheorem{remark}[theorem]{Remark}

\newenvironment{proof}[1][Proof]{\noindent\textbf{#1.} }{\ \rule{0.5em}{0.5em}}

\let\Im\relax
\DeclareMathOperator{\Im}{Im}
\let\Re\relax
\DeclareMathOperator{\Re}{Re}

\begin{document}
\title{On the distribution of zeros of the derivative of Selberg's zeta function
associated to finite volume Riemann surfaces}
\author{Jay Jorgenson and Lejla Smajlovi\'{c}
\footnote{The first named
author acknowledges support from NSF and PSC-CUNY grants.
}}
\date{10 November 2012}

\maketitle \abstract In 1934, A.~Speiser proved that the Riemann
hypothesis for the Riemann zeta function $\zeta_{\mathbb Q}$ is
equivalent to the statement that its derivative $\zeta'_{\mathbb
Q}(s)$ is non-vanishing in the half-plane $\textrm{Re}(s) < 1/2$, see \cite{Sp34}. Further results in the study of the zeros of
$\zeta'_{\mathbb Q}$ were established by Spira \cite{Sp73} and
Levinson-Montgomery \cite{LM74}.  The theorems from \cite{LM74}
played an important role in Levinson's proof that at least one-third of
all non-trivial zeros of the Riemann zeta function $\zeta_{\mathbb Q}(s)$ are located
on the critical line $\text{\rm Re}(s) = 1/2$;  see \cite{Le74} for Levinson's paper as well
as \cite{Se42} for Selberg's classical article in this direction.
With these results, it then becomes important to
study the zeros of derivatives of
zeta functions.

\vskip .05in
In \cite{Luo05}, W.~Luo investigated the
distribution of zeros of the derivative of the Selberg zeta
function associated to compact hyperbolic Riemann surfaces.  In
essence, the main results in \cite{Luo05}, which were extended in
\cite{Gar08}, \cite{Gar09}, \cite{Min08} and \cite{Min10}, involve
the following three points: Finiteness for the number of zeros in
the half plane $\Re (s)<1/2$; an asymptotic expansion for the counting function
measuring the vertical distribution of zeros; and an asymptotic expansion
for the counting function measuring
the horizontal distance of zeros from the critical line.  In the
present article, we study the more complicated setting of
distribution of zeros of the derivative of the Selberg zeta
function associated to a non-compact, finite volume hyperbolic
Riemann surface $M$. There are numerous difficulties which exist
in the non-compact case that are not present in the compact
setting, beginning with fact that in the non-compact case the
Selberg zeta function does not satisfy the analogue of the Riemann
hypothesis.  To be more specific, we actually study the zeros of
$(Z_{M}H_{M})'$, where $Z_{M}$ is the Selberg zeta function and
$H_{M}$ is the Dirichlet series component of the scattering
matrix, both associated to an arbitrary finite-volume hyperbolic
Riemann surface $M$.  As in the above mentioned articles, our
main results address finiteness of zeros in the half plane $\Re (s)<1/2$, an
asymptotic count for the vertical distribution of zeros, and an
asymptotic count for the horizontal distance of zeros.

\vskip .05in
Generally speaking, the philosophy behind the Phillips-Sarnak conjecture suggests that the spectral
analysis of the Laplacian acting on smooth functions on $M$
should depend on the arithmetic nature of the underlying Fuchsian group $\Gamma$.
One realization of the spectral analysis of the Laplacian is the location of
zeros of $Z_{M}$.  Our analysis yields an invariant $A_{M}$ which appears in the vertical
and horizontal distribution of zeros of $(Z_{M}H_{M})'$, and we show that $A_{M}$ has
different values for surfaces associated to two topologically equivalent yet
different arithmetically defined Fuchsian groups. We view this aspect of our main theorem as
both supporting the Phillips-Sarnak philosophy and indicating the existence of further
spectral phenomena which provides a additional refinement within the
set of arithmetically defined Fuchsian groups.
\endabstract

\section{Introduction}

\subsection{Classical results for the Riemann zeta function}

In \cite{Sp34}, A.~Speiser proved that the Riemann hypothesis for
the Riemann zeta function $\zeta_{\mathbb Q}$ is equivalent to
proving that its derivative $\zeta'_{\mathbb Q}(s)$ is
non-vanishing for $\textrm{Re}(s) < 1/2$.  As a consequence of
Speiser's theorem, and recognizing the Riemann hypothesis as a
mathematical question of unparalleled historical and mathematical
significance, one now recognizes the need to study the zeros of
$\zeta'_{\mathbb Q}$, the derivative of the Riemann zeta
function.  The most ambitious goal would be to obtain a precise
description of the location of the zeros of $\zeta'_{\mathbb
Q}$, analogous to the Riemann hypothesis itself for
$\zeta_{\mathbb Q}$.  At this time, a more realistic goal would
be to investigate distributional questions associated to the zeros
of $\zeta'_{\mathbb Q}$, with or without employing underlying
assumptions or hypotheses associated to the location of zeros of
$\zeta_{\mathbb Q}$, such as the classical Riemann hypothesis.

\vskip .10in Let $N_{\textrm{ver}}(T;\zeta_{\mathbb Q})$ denote
the number of zeros $s=\sigma + it$ of the Riemann zeta function
$\zeta_{\mathbb Q}(s)$ such that $\sigma \in (0,1)$ and $0 < t < T$.  Classically, it is known that
\begin{equation}\label{verRiemann}
N_{\textrm{ver}}(T;\zeta_{\mathbb Q}) = \frac{T}{2\pi} \log
(T/2\pi) - \frac{T}{2\pi} + O(\log T), \,\,\,\,\,\textrm{as $T
\rightarrow \infty$.}
\end{equation}
In \cite{Ber70}, vertical counting function $N_{\textrm{ver}}(T;\zeta'_{\mathbb Q})$ of
non-trivial zeros of $\zeta'_{\mathbb Q}$ is studied; it is shown that
\begin{equation}\label{verRiemannderiv}
N_{\textrm{ver}}(T;\zeta'_{\mathbb Q}) = \frac{T}{2\pi} \log
(T/2\pi) - (1 + \log 2))\frac{T}{2\pi} + O(\log T),
\,\,\,\,\,\textrm{as $T \rightarrow \infty$.}
\end{equation}
As noted in \cite{LM74}, the above results can be combined to
yield
$$
N_{\textrm{ver}}(T;\zeta_{\mathbb Q}) =
N_{\textrm{ver}}(T;\zeta'_{\mathbb Q}) + T\cdot \frac{\log
2}{2\pi} + O(\log T), \,\,\,\,\,\textrm{as $T \rightarrow \infty$.}
$$
Similar results for higher derivatives $\zeta^{(k)}_{\mathbb
Q}$ of the Riemann zeta function are also given in \cite{Ber70}.

\vskip .10in The counting function
$$
N_{\textrm{hor}}(T;\zeta'_{\mathbb Q}) =
\sum\limits_{\zeta'_{\mathbb Q}(\sigma + it) = 0 \atop 0 < t < T ,
1/2 < \sigma < 1} (\sigma - 1/2)
$$
is defined to study the horizonal distribution of the zeros of
$\zeta'_{\mathbb Q}$.  From \cite{LM74}, Theorem 5. it is easy to deduce that
\begin{equation}\label{horRiemannderiv}
N_{\textrm{hor}}(T;\zeta'_{\mathbb Q}) \simeq \frac{T}{2\pi}\log
\log (T/2\pi), \,\,\,\,\,\textrm{as $T \rightarrow \infty$.}
\end{equation}
Obviously, the Riemann hypothesis addresses the behavior of the
horizontal counting function $N_{\textrm{hor}}(T;\zeta_{\mathbb Q})$, asserting that
the function is identically zero.

\vskip .10in  In
\cite{Le74}, N. Levinson used (\ref{verRiemannderiv}) to prove
that at least one-third of the non-trivial zeros of $\zeta_{\mathbb Q}(s)$ lie on the
critical line $\text{\rm Re}(s) = 1/2$.  Refined studies involving the
horizontal distribution (\ref{horRiemannderiv}) are given in
\cite{So98}, with additional results in the following articles:
\cite{CG90}, \cite{Fe05}, \cite{GY07}, \cite{Ki08}, \cite{Ng08}
and \cite{Z_M H_M01}, to name a few.

\vskip .10in Within the field of random matrix theory, there are
many points where the zeros of the Riemann zeta function arise:
see the survey article \cite{Di03} and references therein.  Beyond
this, the article \cite{Me03} introduces a connection between
zeros of $\zeta'_{\mathbb Q}$ and random matrix theory, thus
further highlighting the results (\ref{verRiemannderiv}) and
(\ref{horRiemannderiv}).  Additional investigations into the zeros
of the derivative of the Riemann zeta function continue, with new
results appearing frequently:  see \cite{BPS10}, \cite{DFFHMP} and
\cite{FK10}.

\vskip .10in
\subsection{Selberg zeta functions for compact Riemann surfaces}

\vskip .10in Naturally, a generalization of (\ref{verRiemann}),
(\ref{verRiemannderiv}) and (\ref{horRiemannderiv}) can be considered
for any zeta and $L$-function arising from number theory or
elsewhere.  In \cite{Luo05}, W.~Luo initiated the study of the
zeros of the derivative $Z_{M}'$ of the Selberg zeta function $Z_{M}$ associated to a
compact, hyperbolic Riemann surface $M$, proving analogues of
(\ref{verRiemannderiv}) and (\ref{horRiemannderiv}).  Further
refinements of the horizontal and vertical counting functions
$N_{\textrm{hor}}(T;Z'_{M})$ and $N_{\textrm{ver}}(T;Z'_{M})$ were
established in \cite{Gar08} and \cite{Gar09}.  Let us summarize
the three main results, which are location, vertical distribution,
and horizontal distribution of the zeros of $Z'_{M}$.

\vskip .10in  In \cite{Luo05} it is shown that $Z'_{M}(s)$ has at
most a finite number of non-trivial zeros in the half-plane
$\textrm{Re}(s) < 1/2$.  This result was strengthened in
\cite{Min08} and \cite{Min10} where it is  proved that $Z'_{M}(s)$
has no non-trivial zeros in the half-plane $\textrm{Re}(s) < 1/2$.

\vskip .10in  Let $\textrm{vol}(M)$ denote the hyperbolic volume
of $M$.  Let $\ell_{M,0}$ denote the length of the shortest closed
geodesic on $M$; a shortest closed geodesic on $M$ sometimes is called a
\textit{systole} of $M$. Let
$m_{M,0}$ denote the number of inconjugate geodesics whose length
is $\ell_{M,0}$. Let $N_{\textrm{ver}}(T;Z'_{M})$ denote the
number of non-trivial zeros of $Z'_{M}(s)$ with height bounded by
$T$; in other words, where $s=\sigma + it$ with $\sigma\geq 1/2$
and $0 < t < T$. Extending the results in \cite{Luo05}, it is
proved in \cite{Gar08} and \cite{Gar09} that

\begin{equation}\label{verSelbergderivcpt}
N_{\textrm{ver}}(T;Z'_{M}) = \frac{\textrm{vol}(M)}{4\pi}T^{2} -
\frac{\ell_{M,0}}{2\pi}T + o(T) \,\,\,\,\,\textrm{as $T
\rightarrow \infty$.}
\end{equation}

\vskip .10in Let
$$
N_{\textrm{hor}}(T;Z'_{M}) = \sum\limits_{Z_{M}(\sigma + it) = 0
\atop 0 < t < T ,  \sigma > 1/2} (\sigma - 1/2)
$$
Then, building on the results form \cite{Luo05}, it is proved in
\cite{Gar08} and \cite{Gar09} that

\begin{equation}\label{horSelbergderivcpt}
N_{\textrm{hor}}(T;Z'_{M}) = \frac{T}{2\pi} \log T +
\frac{T}{2\pi} \left( \frac{1}{2} \ell_{M,0}+ \log \left(
\frac{\textrm{vol}(M)(1-\exp(-\ell_{M,0})}{ m_{M,0}
\ell_{M,0}}\right)-1 \right)  + o(T)
 \,\,\,\,\,\textrm{as $T \rightarrow \infty$.}
\end{equation}

\vskip .10in The study of the zeros of $Z'_{M}$ is of
particular interest because of the connection with spectral
analysis.  Recall that if $s$ is a non-trivial zero of $Z_{M}(s)$,
then $\lambda = s(1-s)$ is an eigenvalue of an
$L^{2}$-eigenfunction of the hyperbolic Laplacian which acts on
the space of smooth functions on $M$.  Common zeros of $Z_{M}$ and
$Z'_{M}$ are multi-zeros of $Z_{M}(s)$, which, for non-trivial zeros of $Z_{M}$, correspond
to multi-dimensional eigenspaces of the Laplacian. As shown on page 1143 of \cite{Luo05},
all zeros of $Z'_{M}(s)$ on the line
$\textrm{Re}(s) = 1/2$, except possibly at $s=1/2$,
correspond to multiple zeros of $Z_{M}$.  The problem of
obtaining non-trivial bounds for the
dimension of eigenspaces of the Laplacian is very difficult;
see page 160 of \cite{Iw02}.  It is possible that refined information
regarding (\ref{verSelbergderivcpt}) could possibly shed light on
this important, outstanding question.

\vskip .10in
\subsection{Non-compact Riemann surfaces}

Let $\mathbb H$ denote the hyperbolic upper half plane.
Let $\Gamma\subseteq\mathrm{PSL}_{2}(\mathbb{R})$ be
any Fuchsian group of the first kind acting by fractional linear transformations
on $\mathbb H$, and let
$M$ be the quotient space $\Gamma\backslash\mathbb H$.
The question of studying the zeros of $Z'_{M}$ when $M$ is not compact begins with
one possible difficulty stemming from the structure of the Selberg
zeta function.
As stated on page 498 of \cite{Hej83} as well as
page 48 of \cite{Ve90}, the Selberg zeta function itself has an
infinite of zeros in the half-plane $\textrm{Re}(s) < 1/2$, namely
at all the points where $\phi_{M}(s)$, the determinant of the
scattering matrix, has poles.  As we see with the Selberg zeta
function for compact surfaces, and with the Riemann zeta function
assuming the Riemann hypothesis, it seems necessary to study a
function which itself has only trivial zeros in the left-half
plane $\textrm{Re}(s) < 1/2$.  It is from this point that the
analysis of the present paper begins.

\vskip .10in
The function $\phi_{M}(s)$ has a decomposition into a
product of a general Dirichlet series and Gamma functions.
Specifically, from \cite{Iw02}, \cite{Hej83} or \cite{Ve90}, we
can write
$$
\phi_{M}(s) = \pi ^{\frac{n_{1}}{2}}\left( \frac{\Gamma \left(
s-\frac{1}{2} \right) }{\Gamma \left( s\right) }\right)
^{n_{1}}\overset{\infty }{\underset {n=1}{\sum }}\frac{d\left(
n\right) }{\mathfrak{g}_{n}^{2s}}
$$
where $n_{1}$ is the number of cusps of $M$, and $\{d(n)\}$ and
$\{\mathfrak{g}_{n}\}$ are sequences of real numbers with
$$
0<\mathfrak{g}_{1}<...<\mathfrak{g}_{n}<\mathfrak{g}_{n+1}<...;
$$
see also page 33, Theorem 1.5.3 of \cite{Fis87}.  Let us write
$$
\phi_{M}(s) = K_{M}(s)\cdot H_{M}(s)
$$
where
\begin{equation}\label{constants_c1_c2}
K_{M}\left( s\right) =\pi ^{\frac{n_{1}}{2}}\left( \frac{\Gamma \left( s-\frac{1%
}{2}\right) }{\Gamma \left( s\right) }\right) ^{n_{1}}e^{c_{1}s+c_{2}}\,\,\,\,\,\text{%
with}\,\,\,\,\,c_{1}=-2\log \mathfrak{g}_{1}\,\,\,\,\,\text{%
and}\,\,\,\,\, c_{2} = \log d(1),
\end{equation}
and
\begin{equation}\label{Dirichletseriespart}
H_{M}\left( s\right) =1+\overset{\infty }{\underset{n=2}{\sum
}}\frac{a\left( n\right) }{r_{n}^{2s}}
\,\,\,\,\,\text{%
with}\,\,\,\,\,r_{n}=\mathfrak{g}_{n}/\mathfrak{g}_{1} > 1\,\,\,\,\,\text{%
and}\,\,\,\,\,a(n) = d(n)/d(1).
\end{equation}
The Dirichlet series expansion for $H_{M}(s)$ converges for all
$\textrm{Re}(s) > 1$.  We call the function $H_{M}$ the Dirichlet series
portion of the scattering determinant $\phi_{M}$.

\vskip .10in  In general, the function $H_{M}$ can be expressed as the determinant
of a matrix whose entries are general Kloosterman
sums; see, for example, Theorem 3.4, page 60 of \cite{Iw02} as
well as Chapter 4 of \cite{Iw97}.  The constants $\mathfrak{g}_{1}$ and $\mathfrak{g}_{2}$
are explained in terms of the left lower entries of the matrices appearing in the double
coset decomposition of $\Gamma$. Therefore, the constants $\mathfrak{g}_{1}$ and
$\mathfrak{g}_{2}$ are precisely connected to the Fuchsian group $\Gamma$.

\vskip .10in  Since $\phi_{M}$ satisfies the function equation
$\phi_{M}(s)\phi_{M}(1-s)=1$, the zeros and poles of $\phi_{M}$
are symmetrically located about the critical line $\Re(s) = 1/2$.
Furthermore, from Theorem 5.3, page 498 of \cite{Hej83}, one has
that $Z_{M}\phi_{M}$ has no non-trivial zeros in the half-plane
$\text{\rm Re}(s) < 1/2$.  Consequently, the function $Z_{M}H_{M}$
has no non-trivial zeros in the half-plane $\text{\rm Re}(s) <
1/2$. In fact,
$(Z_{M}H_{M})(s)$ is a holomorphic function on $\mathbb{C}\backslash
\left( -\infty ,1/2\right] $.

\vskip .10in As a result, rather than study zeros of $Z'_{M}$,
we shall study the zeros of $(Z_{M}H_{M})'$.

\vskip .10in
\subsection{The main result}

\vskip .10in The function $H_{M}^{\prime }/H_{M}$ has admits the
general Dirichlet series expansion
\begin{equation} \label{Hlogder}
\frac{H_{M}^{\prime }}{H_{M}}\left( s\right)
=\sum _{i=1}^{\infty}\frac{b\left( q_i \right)
}{q_i^{s}}\text{,}
\end{equation}
where series on the right converges absolutely and uniformly for $\text{Re}(s)
\geq \sigma_{0}+\epsilon >\sigma_0$ for sufficiently large $\sigma_{0}$ and where $\left\{ q_i \right\} $
is a non-decreasing sequence of positive real numbers consisting of all
finite products of numbers $r_{n}^{2}>1$. Obviously, $q_2 > q_1 = \inf q_i = \left(
\mathfrak{g}_{2}/\mathfrak{g}_{1}\right) ^{2}$. Furthermore,
$$
b(q_1) = - a(2) \log q_1 = -2 (d(2)/d(1)) \log (\mathfrak{g}_{2}/\mathfrak{g}_{1}).
$$

\vskip .10in Let $\ell_{M,0}$ be the length of a shortest closed
geodesic, or systole, on $M$.  With our notation from above, let
\begin{equation}\label{definitionofA}
A_M=\min \left\{e^{\ell_{M,0}},\left( \mathfrak{g}_{2}/\mathfrak{g}%
_{1}\right) ^{2}\right\}.
\end{equation}
Here, we have dropped the subscript $M$ on $\left(
\mathfrak{g}_{2}/\mathfrak{g}_{1}\right) ^{2}$ in order to ease the notation; however,
it is clear that $\left(
\mathfrak{g}_{2}/\mathfrak{g}_{1}\right) ^{2}$ depends on $M$.
Let $m_{M,0}$ denote the number of inconjugate closed geodesics on
$M$ with length $\ell_{M,0}$. If $e^{\ell_{M,0}} \neq \left( \mathfrak{g}_{2}/\mathfrak{g}
_{1}\right) ^{2}$, let
\begin{equation}\label{definitionofa1}
a_{M} = \left\{\begin{matrix}\displaystyle
\frac{m_{M,0}\ell_{M,0}}{1-e^{-\ell_{M,0}}}; & \textrm{if} \, \,
e^{\ell_{M,0}} < \left( \mathfrak{g}_{2}/\mathfrak{g} _{1}\right) ^{2} \\ \\b(\left(
\mathfrak{g}_{2}/\mathfrak{g}_{1}\right) ^{2}); & \textrm{if} \, e^{\ell_{M,0}} > \left( \mathfrak{g}_{2}/\mathfrak{g} _{1}\right) ^{2}
\end{matrix} \right\}.
\end{equation}
If $e^{\ell_{M,0}}=\left( \mathfrak{g}_{2}/\mathfrak{g}
_{1}\right) ^{2}$, let

\begin{equation}\label{definitionofa1two}
a_{M} = \frac{m_{M,0}\ell_{M,0}}{1-e^{-\ell_{M,0}}}+ b(\left(
\mathfrak{g}_{2}/\mathfrak{g}_{1}\right) ^{2}).
\end{equation}

\noindent
Observe that $a_{M}$ is the sum of the two terms which appear in the two cases in (\ref{definitionofa1}),
not the arithmetic average as one would expect from elementary Fourier analysis.

\vskip .10in With all this, the main result of this article is the
following.

\it \vskip .10in \noindent\textbf{Theorem.} Let $\Gamma\subseteq\mathrm{PSL}_{2}(\mathbb{R})$ be
any Fuchsian group of the first kind acting by fractional linear transformations
on $\mathbb H$, and let
$M$ be the quotient space $\Gamma\backslash\mathbb H$.
Let $Z_{M}(s)$ be the associated
Selberg zeta function, and $H_{M}(s)$ be the Dirichlet series
portion of the determinant of the associated scattering matrix.

\vskip .10in
a) There are a finite number of non-trivial zeros of $(Z_{M}H_{M})'(s)$ in
the half-plane $\text{\rm Re}(s) < 1/2$.  In addition, there exist some $t_0>0$ such that any zero
of $(Z_{M}H_{M})'(s)$ on the line $\Re (s)=1/2$ with property $\vert \Im (s) \vert >t_0$
arises from a multiple zero of $Z_M(s)$.

\vskip .10in
b) Let us define the vertical counting function
\rm
$$
N_{\textrm{ver}}(T;(Z_{M}H_{M})') = \#\{\rho = \sigma + it \,\big|
\, (Z_{M}H_{M})'(\rho) = 0 \,\,\,\textrm{with}\,\,\, 0 < t < T\}.
$$
\it
Then
\rm
$$
N_{\textrm{ver}}(T;(Z_{M}H_{M})')=\frac{\textrm{\rm vol}(M) }{4\pi }T^{2}-\frac{T}{2\pi }%
\left( \log A_M+2n_{1}\log 2+2\log \mathfrak{g}_{1}\right)
+o(T),\,\,\textit{as $T\rightarrow \infty$.}
$$
\it
In particular, if $M$ is co-compact, then \eqref{verSelbergderivcpt} holds true.
\vskip .05in
\it
c) Let us define the horizontal counting function
\rm
$$
N_{\textrm{hor}}(T;(Z_{M}H_{M})') =
\sum\limits_{(Z_{M}H_{M})'(\sigma + it) = 0 \atop 0 < t < T
\,\,\textrm{\rm and}\,\, \sigma > 1/2} (\sigma - 1/2).
$$
\it
Then
\rm
\begin{align*}
N_{\textrm{hor}}(T;(Z_{M}H_{M})') &=\left( \frac{n_{1}}{%
2}+1\right) \frac{T\log T}{2\pi }+\frac{T}{2\pi }\left( \log \frac{%
\textrm{\rm vol}(M)  A_M^{1/2}}{\left\vert a_{M}\right\vert }%
-1\right)  \\& +\frac{T}{2\pi }\left( \log \left(
\frac{\mathfrak{g}_{1}}{\pi
^{n_{1}/2}\left\vert d(1)\right\vert }\right) -\frac{n_{1}}{2}\right) +o(T)%
,\,\,\textit{as $T\rightarrow \infty$.}
\end{align*}
\it
In particular, if $M$ is co-compact, then \eqref{horSelbergderivcpt} holds true.
\rm
\vskip .10in
As stated in the Theorem, the above asymptotic formulas specialize in the case $M$ is
compact to give the main results in \cite{Luo05}, \cite{Gar08}, \cite{Gar09},
\cite{Min08}, \cite{Min08b} and \cite{Min10}.  Similar results for zeros of
higher derivatives of $Z_{M}H_{M}$ are presented in a later section.  In
addition, corollaries of the main theorem, analogous to results from \cite{LM74},
are derived.

\vskip .10in
\subsection{Remarks concerning the Main Theorem}

\vskip .10in
Aspects of the spectral analysis of the Laplacian acting on smooth functions on a hyperbolic Riemann
surface can be measured by studying the zeros of the Selberg zeta function.  As we discussed above,
one equivalently can study the zeros of $Z_{M}H_{M}$.  Indeed, the Selberg zeta function can be constructed
using its divisor, which comes from the eigenvalues of the Laplacian and poles of the scattering matrix
(see page 498 of \cite{Hej83}) together with general characterizing properties associated to its asymptotic
behavior as $\textrm{\rm Re}(s) \rightarrow +\infty$.  Therefore, by slight extension, the zeros of
$(Z_{M}H_{M})'$ provide another measure, in a sense, of the spectral analysis of the Laplacian.  In
this regard, the quantity $\left( \mathfrak{g}_{2}/\mathfrak{g} _{1}\right) ^{2}$ is a new spectral
invariant.  Additionally, our Main Theorem indicates that for any given surface, the spectral analysis
depends on the comparison of $e^{\ell_{M,0}}$ and $\left( \mathfrak{g}_{2}/\mathfrak{g} _{1}\right) ^{2}$.

\vskip .10in
In section 7, we will show that for congruence subgroups one has the inequality
$e^{\ell_{M,0}} > \left( \mathfrak{g}_{2}/\mathfrak{g} _{1}\right) ^{2}$.  However, this inequality
does not hold for all arithmetic groups.  We examine in detail the two ''moonshine groups''
$\Gamma_{0}^{+}(5)$ and $\Gamma_{0}^{+}(6)$.  These two groups are arithmetic and have the same
topological signature.  However, for $\Gamma_{0}^{+}(5)$, we have that $e^{\ell_{M,0}} < \left( \mathfrak{g}_{2}/\mathfrak{g} _{1}\right) ^{2}$ whereas for $\Gamma_{0}^{+}(6)$ we have that
$e^{\ell_{M,0}} > \left( \mathfrak{g}_{2}/\mathfrak{g} _{1}\right) ^{2}$.  We find it very interesting
that, in the sense of our Main Theorem, not all arithmetic surfaces, even those with the
same topological signature, have the same behavior.

\vskip .10in
Also in section 7, we argue that if one considers a degenerating family of hyperbolic Riemann surfaces
within the moduli space of surfaces of fixed topological type, one eventually has the
inequality $e^{\ell_{M,0}} < \left( \mathfrak{g}_{2}/\mathfrak{g} _{1}\right) ^{2}$ near the boundary.
As a result, if one begins with congruence group and degenerates the corresponding surface,
one will ultimately encounter
a surface where $e^{\ell_{M,0}} = \left( \mathfrak{g}_{2}/\mathfrak{g} _{1}\right) ^{2}$.
More generally, however, it seems as if moduli space can be separated into sets defined by the
sign of $e^{\ell_{M,0}} - \left( \mathfrak{g}_{2}/\mathfrak{g} _{1}\right) ^{2}$ where most, but not all,
arithmetic surfaces are in the component
where $e^{\ell_{M,0}} - \left( \mathfrak{g}_{2}/\mathfrak{g} _{1}\right) ^{2}>$,
and the Deligne-Mumford boundary lies in the component where
$e^{\ell_{M,0}} - \left( \mathfrak{g}_{2}/\mathfrak{g} _{1}\right) ^{2}< 0$.

\vskip .10in
We could not explicitly construct a surface where
$e^{\ell_{M,0}} - \left( \mathfrak{g}_{2}/\mathfrak{g} _{1}\right) ^{2}=0$, even though
we prove, in section 7, that such surfaces exist.

\vskip .10in
\subsection{A comparison of counting functions}

\vskip .10in On page 456 of \cite{Hej83}, D.~Hejhal establishes
the asymptotic behavior of the horizontal distribution of zeros of
$\phi_{M}$ within the critical strip.  In our notation, the zeros
of $\phi_{M}$ within the critical strip coincide with the zeros of
the Dirichlet series $H_{M}$, so then Theorem 2.22, page 456 of
\cite{Hej83} establishes the asymptotic behavior of the horizontal counting function
$N_{\textrm{hor}}(T;H_{M})$.

\vskip .10in Let $M$ be any finite volume hyperbolic Riemann
surface. We claim theres exist a co-compact hyperbolic
Riemann surface $\widetilde{M}$ such that $\textrm{vol}(M) = \textrm{vol}(\widetilde{M})$, $\ell_{M,0} = \ell_{\widetilde{M},0}$ and $m_{M,0} = m_{\widetilde{M},0}$, which we argue as follows.  In the case when the
number $n_1$ of cusps of the surface $M$ is even, we choose the surface $\widetilde{M}_1$ to be any
co-compact surface with genus $g_{\widetilde{M}}= g_M + n_1/2$ and the same structure of elliptic points as
$M$, hence $\textrm{vol}(M) = \textrm{vol}(\widetilde{M}_1)$.
If the number of cusps of the surface $M$ is odd, we choose the surface
$\widetilde{M}_1$ to be any co-compact surface with genus $g_{\widetilde{M}}= g_M + (n_1-1)/2$ such that it has the same structure of elliptic points as $M$, plus one additional elliptic point of order 2. By the Gauss-Bonnet formula, $\textrm{vol}(M) = \textrm{vol}(\widetilde{M}_1)$.  We then deform the surface $\widetilde{M}_1$
in moduli space so that its shortest geodesic has the length equal to $\ell_{M,0}$ and the number of
inconjugate geodesics of length $\ell_{M,0}$ is $m_{M,0}$.

\vskip .10in
Assume that $M$ is such that, in the notation of
(\ref{definitionofA}), $A_{M} = \exp(\ell_{M,0})$. Then, as we
will prove in a later section, one can combine Hejhal's theorem regarding $N_{\textrm{hor}}(T;H_{M})$
with part (c) of the Main Theorem to
establish the simple asymptotic relation
\begin{equation}\label{countcomparison}
N_{\textrm{hor}}(T;(Z_{M}H_{M})')
=N_{\textrm{hor}}(T;Z'_{\widetilde{M}}) +
N_{\textrm{hor}}(T;H_{M}) +o(T)  \,\,\,\text{as $T\rightarrow
\infty$.}
\end{equation}

\noindent
In a later section, we will show that the relation
\eqref{countcomparison} holds true when the derivative is replaced
by the $k$th derivative, for all $k\geq2$.

\vskip .10in We find it very interesting that, in the case when $M$
is such that $\exp(\ell_{M,0}) < \left(
\mathfrak{g}_{2}/\mathfrak{g}_{1}\right) ^{2}$, the
coefficients of the first two terms, namely $T\log T$ and $T$,  in
the asymptotic development of the counting function
$N_{\textrm{hor}}(T;(Z_{M}H_{M})')$ coincide with known results,
namely Hejhal's theorem and (\ref{horSelbergderivcpt}).
The comparison (\ref{countcomparison}) is vaguely
reminiscent of the main result of \cite{TZ}.  In that article, the
authors compute the curvature of a determinant line bundle on the
moduli space of finite volume hyperbolic Riemann surfaces, showing
that the curvature form consists of two parts:  One part related to
the curvature form in the compact case, and a second part defined
using parabolic Eisenstein series.

\vskip .10in
\subsection{Further comments}

Weyl's law in its classical form evaluates
the lead asymptotic behavior of the vertical counting function
$N_{\textrm{ver}}(T;Z_{M})$ for compact $M$.  As far as is known, the
expansion in $T$ involves $\textrm{vol}(M)$ and no other information associated
to the uniformizing group $\Gamma$.  If $M$ is non-compact,
the generalization of Weyl's law addresses the asymptotic behavior of
\begin{equation}\label{weyllawfunction}
\#\{\lambda_{j,M} < 1/4+T^{2} \} - \frac{1}{4\pi}\int\limits_{-T}^{T}
\phi'_{M}/\phi_{M}(1/2+ir)dr
\end{equation}
where $\lambda_{j,M}$ is the eigenvalue of an $L^{2}$ eigenfunction on $M$.
The asymptotic expansion of (\ref{weyllawfunction}) is recalled below (formula \eqref{Weyl}) and, as in the compact case, all terms in the
expansion involve elementary quantities associated to the uniformizing group
$\Gamma$.
\vskip .10in In section 9.1, we will express the function in
(\ref{weyllawfunction}) in terms of $N_{\textrm{ver}}(T;Z_{M}H_M)$,
obtaining an expression which involves the constant
$\mathfrak{g}_{1}$.  As a result, we accept the appearance of the
term $\mathfrak{g}_{2}/\mathfrak{g}_{1}$ in our Main Theorem as
being an appropriate generalization of a version of Weyl's law.

\vskip .10in In a different direction, if one considers a
degenerating family of finite volume hyperbolic Riemann surfaces,
then it was shown in \cite{HJL97} that the asymptotic behavior of
the associated sequence of vertical counting functions
$N_{\textrm{ver}}(T;Z_{M})$ has lead asymptotic behavior, for
fixed $T$, which involves the lengths of the pinching geodesics;
see Theorem 5.5 of \cite{HJL97}. As a result, we do not view the
appearance of the invariant $\ell_{M,0}$ in
(\ref{verSelbergderivcpt}) and (\ref{horSelbergderivcpt}) as a new
feature when using Weyl's laws to understand refined information
associated to the uniformizing group $\Gamma$.

\vskip .10in However, we find the appearance of the constants
$A_{M}$ and $a_{M}$, as defined in (\ref{definitionofA}),
(\ref{definitionofa1}) and (\ref{definitionofa1two}) to be
surprising.  In particular, for any given surface $M$, we do not
know if there are conditions which will determine the value taken
by $A_{M}$.  If $\Gamma = \Gamma_{0}(N)$, a congruence subgroup,
then we show that $A_{M} =
(\mathfrak{g}_{2}/\mathfrak{g}_{1})^{2}$. If $\Gamma =
\overline{\Gamma_{0}(5)^{+}}$, an arithmetically defined
``moonshine group'' (see \cite{Cum04} and references therein), we
show that $A_{M} = e^{\ell_{M,0}}$. It is even more surprising
that in the case of the group $\Gamma =
\overline{\Gamma_{0}(6)^{+}}$, that is of the same signature as
$\Gamma = \overline{\Gamma_{0}(5)^{+}}$, we again have the
relation $A_{M} = (\mathfrak{g}_{2}/\mathfrak{g}_{1})^{2}$.

\vskip .10in
Consequently, we conclude that the study of the vertical counting
function $N_{\textrm{ver}}(T;(Z_{M}H_{M})')$ contains a term which
provides new information associated to $\Gamma$ which we do
not see as being previously detected.

\vskip .10in
The philosophy behind the Phillips-Sarnak conjecture suggests that the spectral
analysis of the Laplacian acting on smooth functions on $M$
should depend on the arithmetic nature of the underlying Fuchsian group $\Gamma$.
One realization of the spectral analysis of the Laplacian is the location of
zeros of $Z_{M}$, which includes as a special consideration the function
$N_{\textrm{ver}}(T;Z_{M})$, or, equivalently, the function $N_{\textrm{ver}}(T;Z_{M}H_M)$.
If we are allowed to view the vertical
counting function $N_{\textrm{ver}}(T;(Z_{M}H_{M})')$ as another measure of the spectral
analysis of the Laplacian on $M$, then our Main Theorem shows the existence of
refined information, namely $A_{M}$ with its conditional definition (\ref{definitionofA}),
about the uniformizing group $\Gamma$.  In addition, we found that the value of $A_{M}$
is different for two different arithmetically defined discrete groups with the same signature.
We view this conclusion quite surprising, yet in full support of the Phillips-Sarnak philosophy.

\vskip .10in
\subsection{Computations for the modular group}
After the completion of this article, W. Luo brought to our attention the unpublished article
\cite{Min08up} from 2008 in which the author undertakes a related study in the case when
$\Gamma = \textrm{\rm PSL}(2, \mathbb Z)$.  There are a number of important differences between
the results in the present paper and those in \cite{Min08up}, which we now discuss.

\vskip .10in
In \cite{Min08up}, as the title of the article states,
the author studies the zeros of the derivative of the zeta function
$Z_{M}(s)/\zeta_{\mathbb Q}(2s)$ where
$M = \textrm{\rm PSL}(2, \mathbb Z) \backslash \mathbb H$.  If we restrict our
analysis to the case when $\Gamma = \textrm{\rm PSL}(2, \mathbb Z)$, then the
function whose derivative we study is $Z_{M}(s) \zeta_{\mathbb Q}(2s-1)/\zeta_{\mathbb Q}(2s)$.
Since the article \cite{Min08up} studies a different function than in the present article,
one would expect that the statements of the main results are different, as, indeed, is the case.
More importantly, however, the asymptotic expansions obtained in \cite{Min08up} has an error
term of $O(T)$, whereas our error term is $o(T)$, which is significant since the coefficient of
the $T$ term contains the quantity $A_{M}$, which we view as a new spectral invariant.

\vskip .10in
Finally, we note that the article \cite{Min08up} studies the single group $\Gamma =\textrm{\rm PSL}(2, \mathbb Z)$.
The approach may extend to other settings when one has explicit knowledge of the scattering matrix; however,
we do not see how the approach of \cite{Min08up} would apply for general non-arithmetic surfaces.  By contrast,
our Main Theorem applies to an arbitrary, co-finite group $\Gamma$, and the issue of arithmeticity of $\Gamma$ plays
a role only when one is evaluating the invariant $A_{M}$.

\vskip .10in
\subsection{Outline of the paper}
This article is organized as follows.  In Section 2, we will
establish notation and recall necessary results from the
literature.  The zero-free region for $(Z_M H_M)'$, as stated in part
(a) of the Main Theorem,  will be proved in Section 3.  Various
lemmas leading up to the proof of parts (b) and (c) of the main
Theorem will be given in Section 4, the proof of parts (b) and
(c) will be completed in Section 5, and in Section 6 we will state
and prove several corollaries of the Main Theorem.  The examples of
congruence groups and ``moonshine'' groups will be given in
Section 7.  In Section 8, we prove results analogous to our Main
Theorem for higher derivatives of $Z_{M}H_{M}$.  Finally, in Section 9, we will give various concluding
remarks.  In particular, we will review Weyl's law for $M$ and
show that, under a natural restatement, the constant ${\mathfrak
g}_{1}$ appears.

\vskip .20in
\section{Background material}

\subsection{Basic notation}
Let $\Gamma \subseteq \textrm{PSL}(2,\mathbb{R})$ be a Fuchsian
group of the first kind acting on the upper half plane $\mathbb
H$, which we parameterize by $z = x + i y \in {\mathbb C}$ with $y
> 0$.  Let $M = \Gamma \backslash \mathbb{H}$ be the quotient
Riemann surface, which of course may have orbifold singularities
if $\Gamma$ has elliptic elements.  The upper half plane $\mathbb H$ is
equipped with the canonical metric with constant negative
curvature equal to $-1$, which induces a metric on $M$ whose
volume $\textrm{vol}(M)$ is finite.  We will assume standard
notions in hyperbolic geometry, referring to \cite{Bea83} for
further details.

\subsection{Counting functions.}

Let $F$ denote either a general Dirichlet series with a critical line; $F$ itself may be
the derivative of another general Dirichlet series.  We assume that $F$ is
normalized to be convergent in the half plane $\text{\rm Re}(s) >
1$ with critical line $\text{\rm Re}(s) = 1/2$.  We define the
vertical counting function of $F$ as
$$
N_{\textrm{ver}}(T;F) = \sum\limits_{F(\sigma + it) = 0 \atop 0 <
t < T , 0 < \sigma < 1} 1
$$
and the horizontal counting function of $F$ as
$$
N_{\textrm{hor}}(T;F) = \sum\limits_{F(\sigma + it) = 0 \atop 0 <
t < T , 1/2 < \sigma < 1} (\sigma - 1/2).
$$
Classical results study the vertical and horizontal counting
functions when $F$ is a zeta function from an algebraic number
field, with more recent attention turned to the setting when $F$
is the derivative of such a zeta function, as discussed in the
introduction.

\vskip .10in
\subsection{The Selberg zeta function}
Let $\mathcal{H}(\Gamma)$ denote a complete set of representatives of
inconjugate, primitive hyperbolic elements of
$\Gamma$.  For each $P \in \mathcal{H} (\Gamma)$, there exists an element $P_0 \in \mathcal{H} (\Gamma)$ such that $P=P_0^{n}$ for some positive integer $n$. The element $P_0$ is called a primitive element of $\mathcal{H}(\Gamma)$. If $\ell_{P}$
denotes the length of the geodesic path in the homotopy class
determined by $P$, then the norm of the element $P$, denoted by $N(P)$ is equal to $\exp (\ell_{P})$.  For $s \in \mathbf{C}$ with
$\textrm{Re}(s) > 1$, the Selberg zeta function $Z_{M}(s)$ is formally
defined by the Euler product
\begin{equation} \label{EulProd}
Z_{M}(s) = \prod\limits_{n=0}^{\infty}\prod\limits_{P_0 \in
\mathcal{H}(\Gamma)} \left( 1- e^{-(s+n)\ell_{P_0}}\right)= \prod\limits_{n=0}^{\infty}\prod\limits_{P_0 \in
\mathcal{H}(\Gamma)} \left( 1- N(P_0)^{-(s+n)}\right).
\end{equation}

\noindent
The product \eqref{EulProd} is defined for $\textrm{Re}(s) > 1$ and
admits a meromorphic continuation to the entire complex plane with the
functional equation $Z_M(s) \phi_M (s)= \eta_M(s) Z_{M}(1-s)$ where
\begin{equation*}
\eta_M (s)=\eta_M (1/2)\exp \left( \underset{1/2}{\overset{s}{\int
}}\frac{\eta_M ^{\prime }}{\eta_M }(u)du\right) ,
\end{equation*}
and
\begin{eqnarray}
\frac{\eta_M ^{\prime }}{\eta_M }(s) &=&\mathrm{vol} (M)
(s-1/2)\tan (\pi (s-1/2))- \pi \underset{0<\theta (R)<\pi
}{\underset{\left\{
R\right\} }{\sum }}\frac{1}{M_{R}\sin \theta }\frac{\cos (2\theta -\pi)(s-1/2)}{\cos \pi(s-1/2)} \notag \\
&&+2n_{1}\log 2+n_{1}\left( \frac{\Gamma ^{\prime }}{\Gamma }(1/2+s)+\frac{%
\Gamma ^{\prime }}{\Gamma }(3/2-s)\right) \notag \\
&&=\frac{\eta_M ^{\prime }}{\eta_M }%
(1-s). \label{DefEtaLogDer}
\end{eqnarray}
The set $\{R\}$ denotes the finite set of inconjugate elliptic elements of $\Gamma$ and $0< \theta(R)< \pi$ is uniquely determined real number such that elliptic element $R$ is conjugate to the matrix
$$
\left(\begin{array}{cc}\cos \theta(R) & -\sin \theta(R) \\
\sin \theta(R) & \cos \theta(R) \\\end{array}\right).
$$
We refer the interested reader to \cite{Hej83} for a proof of meromorphic
continuation and functional equation of $Z_{M}(s)$; specifically, see pp. 499-500 in the
case where, in the notation of \cite{Hej83}, $m=0,$ $r=1$ and $W= \mathrm{Id}$.

\vskip .10in
The divisor of the Selberg zeta function is stated as Theorem 5.3,
page 498 of \cite{Hej83}. In brief, there are trivial zeros and
poles at the negative integers and half-integers, as well as a
finite number of poles at $s=1/2$ and $s=s_{n} \in (0,1/2)$, where
$s_{n}(1-s_{n})\in (0,1)$ is a small eigenvalue of the Laplacian.
Additionally, there are zeros at points of
the form $1/2 + ir_{n}$, where $1/4 + r_{n}^{2}$ is an eigenvalue
of the Laplacian and at points in the half-plane $\text{\rm Re}(s)
< 1/2$ coming from the non-trivial factor $\phi_{M}$ in the
functional equation.  The function $\phi_{M}$ is the determinant
of the scattering matrix and is described in section 1.3 above.
The reader is referred to page 498 of \cite{Hej83} for additional
background material on this matter.

\subsection{Additional identities}

For any element $P \in \mathcal{H}(\Gamma)$, let $P_{0} \in \mathcal{H}(\Gamma)$ be the unique
primitive hyperbolic element such that $P=P_{0}^{n}$, for some positive integer $n$. The logarithmic derivative
\begin{equation} \label{logderDef}
D_{M}(s):=\frac{Z_{M}'}{Z_{M}}(s)
\end{equation}
of the Selberg zeta function may be expressed, for $\Re (s)>1$ as the absolutely convergent series
\begin{equation} \label{logderZ}
D_M (s) = \sum _{P \in \mathcal{H}(\Gamma)} \frac{\Lambda (P)}{N(P)^{s}},
\end{equation}
where
$$
\Lambda (P):= \frac{\log N(P_{0})}{1-N(P)^{-1}};
$$
one views $\Lambda(P)$ as the analogue of the classical von Mangoldt $\Lambda$ function.

\vskip .10in
Dirichlet series representation of the logarithmic derivative of the function $Z_M H_M$ is given by the following lemma.

\vskip .10in
\begin{lemma} \label{SeriesRepLogDerZ_M H_M} There exists a constant $\sigma_0' \geq 1$ such that
for all $s \in \mathbb{C}$ with $\Re (s)\geq \sigma_0' + \epsilon >\sigma_0'$, we have that
\begin{equation} \label{logderZ_M H_M}
\frac{(Z_M H_M)'}{(Z_M H_M)}(s)= \underset{P
\in \mathcal{H}(\Gamma) }{\sum }\frac{\Lambda (P)}{N(P)^{s}}+\sum _{i=1}^{\infty}\frac{%
b\left( q_{i}\right) }{q_{i}^{s}}.
\end{equation}
In addition, the series converge absolutely and uniformly on every compact subset of the half plane
$\Re (s)> \sigma_0'$.
\end{lemma}

\vskip .10in
\begin{proof}
The proof follows from the elementary formula
$$(Z_M H_M)'(s)= (Z_M(s)H_M(s))\left( \frac{Z_M'}{Z_M}(s) + \frac{H_M'}{H_M}(s)\right)$$
together with equations \eqref{Hlogder} and \eqref{logderZ}. We may take $\sigma_0'$ to be equal to $\sigma_0$,
which was defined in section 1.4.
\end{proof}

\begin{lemma} \label{lemmaFuctEq}
The derivative of the function $Z_M H_M$ satisfies the functional
equation
\begin{equation} \label{DerivFunctEq}
 \left( Z_M H_M\right) ^{\prime }(s) = f_M
(s)\eta_M (s) K_M ^{-1}(s)\widetilde{Z}_M(1-s) Z_M(1-s),
\end{equation}
where
\begin{equation} \label{definitOf f_M}
f_M (s):=\mathrm{vol}(M) (1/2 -s)\left( \tan \pi (1/2 -s)\right)
\end{equation}
and
\begin{equation} \label{ZtildaDef}
\widetilde{Z}_M(s):= \frac{1}{f_M (s)}\left( \frac{\eta_M ^{\prime }}{\eta_M }(s)-\frac{%
K_M^{\prime }}{K_M}(1-s)-\frac{Z_M^{\prime }}{Z_M}(s)\right).
\end{equation}
\end{lemma}

\vskip .10in
\begin{proof}
If one writes the functional equation of the Selberg zeta function as
\begin{equation} \label{functeqMain}
Z_M H_M (s)=\eta_M (s)K_M^{-1}(s)Z_M(1-s),
\end{equation}
the claimed result then follows from straightforward computations in calculus.
\end{proof}

\vskip .10in
Let
$$
\psi_{M}(x):=\sum_{N(P) \leq x} \Lambda(P)
$$
denote the prime geodesic counting function.  The prime geodesic theorem, which we
quote from \cite{Hej83}, states the asymptotic growth of $\psi_{M}(x)$.  Specifically, Theorem 3.4 on page 474,
with the notation that $W=\mathrm{Id}$), states the asympototic formula
$$
\psi_{M}(x) = \sum _{k=0}^{K} \frac{x^{s_{k}}}{s_{k}}+O(x^{3/4} \sqrt{\log x})
\,\,\,\,\, \text{  as  } x \rightarrow + \infty.
$$
In the prime geodesic theorem, we have the notation that $0=\lambda_{0} <\lambda_{1} \leq \lambda_{2} \leq ...
\lambda_{K} <1/4$ is the sequence of
discrete eigenvalues of the Laplacian $\Delta_{M}$ less than $1/4$, and $s
_{n}=1/2 + \sqrt{\frac{1}{4}-\lambda _{n}}$. Since $\lambda_0=0$, the leading term in the expansion of $\psi_{M}(x)$ is $x$.

\vskip .20in
\subsection{An integral representation for $D_M(s)$}

\vskip .10in
In this section we recall results from \cite{AS09} on the growth of the logarithmic derivative $D_M(s)$ and its derivatives $D_M^{(k)}(s)$ for $s=1/2 + \sigma + iT$, as $T \rightarrow \pm \infty$, for $\sigma \in
(0, 1/2)$.

\vskip .10in
The upper bounds on the growth of $D_{M}^{(k)}$, for
$k=0,1,2,...$ are deduced inductively from a new integral
representation of $D_M(s)$.  Since the bounds
for the logarithmic derivative represent an important
ingredient in the proof of the Main Theorem, we will recall the
new integral representation of \eqref{logderDef} obtained in
\cite{AS09}; see Theorem 5.1b).  We then briefly explain how to deduce the
appropriate bounds for $D_{M}^{(k)}$ for all $k \geq 0$.

\vskip .10in
For $n=0,...,K$, so then $\lambda_{n} <1/4$,
we set $\ r_{n}=-i\sqrt{\frac{1}{4}-\lambda _{n}}$. Set
\begin{equation*}
r\left( t\right) =\tanh \pi t-1,
\,\,\,\,\,\textrm{and}\,\,\,\,\,
H(t)=\frac{\Gamma
^{\prime }}{\Gamma }\left( 1+it\right) +\frac{\Gamma ^{\prime }}{\Gamma }%
\left( 1-it\right) -2\log t.
\end{equation*}
Let $N_{M}\left[ 0\leq r_{n}\leq t\right] $ be the counting function
 for the number of non-negative numbers $r_{n} \leq t$ such that
$\frac{1}{4}+r_{n}^{2}=\lambda _{n}$ is an eigenfunction of the Laplacian.  Let
\begin{equation*}
R_{M}\left( t\right) =N_{M}\left[ 0\leq r_{n}\leq t\right] -\frac{1}{4\pi }\overset%
{t}{\underset{-t}{\int }}\frac{\phi_{M} ^{\prime }}{\phi_{M} }\left( \frac{1}{2}%
+iu\right) du-\frac{\textrm{vol} (M)}{4\pi }t^{2}+\frac{t
n_{1}}{\pi } (\log 2t - 1),
\end{equation*}
which is the error term in the Weyl's law.  From \cite{Hej83}, namely Theorems 2.28 and 2.29 on pages 466-468,
we have that
\begin{equation} \label{errorBound}
R_{M}\left( t\right)  =O \left (\frac{t}{\log t} \right)    \text{  and  }  \int \limits _{0} ^{t} R_M(u) du = O \left (\frac{t}{\log ^{2} t} \right) .
\end{equation}

With all this, we quote the following result is from \cite{AS09}.

\vskip .10in
\begin{theorem} \label{IntRep} Let $s \in \mathbb C$ with $\mathrm{Re}(s)>1/2$ and write $s = 1/2 + \alpha$.
Let $y>0$ be arbitrary and set $x=e^{y}$. Then, we have the identity
\begin{align}
D_{M}\left( s\right) &=\frac{1}{1+x^{2\alpha }}
\sum_{P\in \mathcal{H}(\Gamma) \text{, }N(P)<x } \frac{\Lambda (P)}{N\left(
P\right) ^{\alpha +\frac{1}{2}}}\left( x^{2\alpha }-N\left( P\right)
^{2\alpha }\right)  \label{LogDerDva} \\
&+\frac{4\alpha x^{\alpha }}{1+x^{2\alpha }}\left( \overset{K}{\underset{n=0}{%
\sum }}\frac{\cos yr_{n}}{\alpha ^{2}+r_{n}^{2}}+\frac{R_{M}(0)}{\alpha ^{2}}+%
\underset{0}{\overset{\infty }{\int }}\frac{\cos y t d R_{M}\left( t\right) }{%
\alpha ^{2}+t^{2}}-\frac{\textrm{\rm vol}(M)}{2\pi }\underset{0}{%
\overset{\infty }{\int }}\frac{t\cdot r\left( t\right) \cos yt}{\alpha
^{2}+t^{2}}dt  \right. \notag \\
&+\frac{n_{1}}{2\pi }\underset{0}{\overset{\infty }{\int }}\frac{\cos
y t H\left( t\right) dt}{\alpha ^{2}+t^{2}}\left. -\underset{0<\theta \left( R\right) <\pi }{\sum_{\left\{ R\right\}
_{\Gamma }}}\frac{1}{2M_{R}\sin \theta }%
\overset{\infty }{\underset{0}{\int }}\frac{\cos y t}{\alpha ^{2}+t^{2}}\frac{%
\cosh 2\left( \pi -\theta \right) t}{\cosh 2\pi
t}dt  - \frac{1}{4} \textrm{\rm Tr} \left(I- \Phi_{M} (\frac{1}{2}) \right)\right) \notag
\end{align}
All integrals in \eqref{LogDerDva} converge uniformly
in $s$ on every compact subset of the half-plane $\Re (s) >1/2$.
\end{theorem}

\vskip .10in
The representation \eqref{LogDerDva} itself is of a
number-theoretic interest, since it allows us to pass to the limit on the
right hand side as $y\rightarrow \infty $, while the left hand side does not
depend upon $y$, and obtain some interesting special values of the function
$D_{M}(s)$; see \cite{AS06} and \cite{AS08}.  The identity \eqref{LogDerDva}
is used in \cite{AS09} to deduce the following theorem.

\vskip .10in
\begin{theorem}
For $s=1/2 + \sigma + iT$, $0< \sigma < 1/2$ and every non-negative integer $k$,
we have the asymptotic bound
\begin{equation}
D_M^{(k)}\left( s\right) =O\left( \min \left\{ \frac{\left| T\right| }{\sigma
^{k+1}\log \left| T\right| },\left| T\right| ^{1-2\sigma }\log ^{k-2\sigma
}\left| T\right| \cdot \underset{j=0,...,k}{\max }\left\{ \frac{1}{\sigma
^{j+1}\log ^{j+1}\left| T\right| },\log \left| \frac{T}{\sigma }\right|
\right\} \right\} \right) ,  \label{BoundDer}
\end{equation}
as $\left| T\right| \rightarrow \infty $.
\end{theorem}

\vskip .10in
\begin{proof}
For the sake of completeness, we will recall main steps of the proof from \cite{AS09}.
\vskip .05in
The first bound for $k=0$ is obtained by letting $y \rightarrow 0$ in the
representation \eqref{LogDerDva} and showing that we may pass to the
limit as $y \rightarrow 0$ inside the integrals in \eqref{LogDerDva}. The main obstacle is to prove that
$$
\underset{0}{\overset{\infty }{\int }}\frac{%
dR_{M}\left( t\right) }{\left( s-1/2\right) ^{2}+t^{2}} = O \left(
\frac{1}{\sigma \log \vert T \vert}\right)
\,\,\,\,\,\textrm{\rm as $\vert T\vert \rightarrow \infty$},
$$
which is done by integrating by parts two times and then
using the second  bound in \eqref{errorBound}. The bound for $k\geq 1$ is derived by induction.

\vskip .10in
The second bound in the case when $k=0$ is derived by choosing $x
\sim T^{2} \log^{2} \vert T \vert$ and estimating each term in \eqref{LogDerDva} separately. There are two main issues.

\vskip .10in
We use the prime geodesic theorem and integration by parts, for
$\alpha =\sigma +iT$, $\frac{1}{2}>\delta \geq \sigma >0$ in order to arrive at the bound
\begin{align}
\frac{1}{1+x^{2\alpha }}&\underset{N\left( P\right) <x}{\sum }\frac{\Lambda \left( P\right) }{N\left( P\right) ^{\alpha +\frac{1}{2}}}\left( x^{2\alpha }-N\left( P\right) ^{2\alpha }\right) \notag
\\ \ll & \frac{1}{1+\left( T\log \left| T\right| \right)
^{4\alpha }}\underset{\eta }{\overset{T^{2}\log ^{2}\left|
T\right| }{\int }}\frac{\left( \left( T\log \left| T\right| \right)
^{4\alpha }-t^{2\alpha }\right) }{t^{\alpha +\frac{1}{2}}}d\psi_{M}(t) \notag \\
= & O(\left| T\log\left| T\right| \right| ^{1-2\sigma })=O\left( \left| T\right| ^{1-2\sigma }\log ^{-2\sigma }\left| T\right| \cdot
\log \left| \frac{T}{\sigma }\right| \right), \text{  as  } \left| T\right| \rightarrow \infty\notag
\end{align}
for the first sum in \eqref{LogDerDva}. The lower limit $ \eta \in(1,  N(P_{00}))$ is arbitrary.

\vskip .10in
Secondly, from the bound
\begin{equation*}
\underset{0}{\overset{\infty }{\int }}\frac{\cos ytdR_{M}\left( t\right) }{%
\alpha ^{2}+t^{2}}=O\left( \max \left\{ \frac{1}{\sigma \log \left| T\right|
},\log \left| T/\sigma \right| \right\} \right) \text{, as }\left| T\right|
\rightarrow \infty
\end{equation*}
we deduce that
\begin{equation*}
\frac{4\alpha x^{\alpha }}{1+x^{2\alpha }}\underset{0}{\overset{\infty }{%
\int }}\frac{\cos ytdR_{M}\left( t\right) }{\alpha ^{2}+t^{2}}=O\left(
T^{1-2\sigma }\log ^{-2\sigma }\left| T\right| \max \left\{ \frac{1}{\sigma
\log \left| T\right| },\log \left| \frac{T}{\sigma }\right| \right\} \right)
\text{,}
\end{equation*}
for $\alpha =\sigma +iT$,
$\frac{1}{2}>\delta \geq \sigma >0$, as $\left| T\right| \rightarrow \infty $.

\vskip .10in
Other integrals on the right hand side of \eqref{LogDerDva} may be
easily estimated using the approximation of functions under the
integral sign. This completes the proof of the theorem in the case
when $k=0$.

\vskip .10in
The proof in the case $k =1$ is derived by multiplying the formula
\eqref{LogDerDva} by $(1+x^{2\alpha})$ then differentiating the
resulting identity, after which one estimates of the integrals
using the stated result when $k=0$. Differentiation under the integral sign is
justified due to the uniform convergence of the integral for
$1/2>\delta \geq \sigma \geq \sigma_0 >0$. The case $k>1$ is
proved analogously by induction.
\end{proof}

\vskip .20in
\subsection{The completed function $\Xi_{M}$}

\vskip .10in
In this section we recall the notation and results from
\cite{Fis87}.  The notation of \cite{Fis87} is adjusted to
our setting; we take $k=0$, dimension $d=1$ and
$\tau ^{\ast }=n_{1}$.

\vskip .10in
The completed function $\Xi_{M}$ associated to the Selberg zeta function is defined
by
$$
\Xi_{M}(s)=\Xi _{I}(s)\Xi _{M, \textrm {hyp}}(s)
\Xi _{M, \textrm {par}}(s)\Xi _{M, \textrm {ell}}(s)
$$
where $\Xi _{M, \textrm {hyp}}(s)=Z_{M}(s)$ is
the Selberg zeta function and the remaining functions are associated to the
identity, parabolic, and elliptic elements in the underlying uniformizing group.
The logarithmic derivative of the identity term $\Xi _{I}$ is given by
\begin{equation}
-\frac{1}{2s-1}\frac{\Xi _{I}^{\prime }(s)}{\Xi
_{I}(s)}=\frac{\mathrm{vol}(M) }{2\pi }\frac{\Gamma ^{\prime
}}{\Gamma }(s); \label{XiIdDerivative}
\end{equation}
see Remark 3.1.3 in \cite{Fis87}.  The function $\Xi _{M, \textrm {ell}}(s)$ is computed in
Corollary 2.3.5 of \cite{Fis87}; using Stirling's formula, one can show that
\begin{equation}
\frac{1}{2s-1}\frac{\Xi _{M, \textrm {ell}}^{\prime }(s)}{\Xi _{M, \textrm {ell}}(s)}=O\left( \frac{1}{%
\left\vert t\right\vert }\log \left\vert t\right\vert \right)
\label{XiEllDerivative}
\end{equation}
for any $s=\sigma +it$, $\sigma \leq 1/2$, as $\vert t\vert \rightarrow \infty$.

\vskip .10in
The function $\Xi _{M, \textrm {par}}(s)$ is described in \cite{Fis87}; see
Definition 3.1.4.  For our purposes, it suffices to relate $\Xi _{M, \textrm {par}}(s)$
to the scattering determinant $\phi_{M}(s)$, so that we obtain an expression for $Z_{M}H_{M}(s)$.
The following computations derive such an expression for $Z_{M}H_{M}(s)$.

\vskip .10in
Let $\left\{ p_{1},...,p_{N_0}\right\} $ denote the
set of poles of $\phi_{M} $ lying in $\left( 1/2,1\right] $,
counted with multiplicities; let $q_1, ..., q_{N_{1}}$ denote the
set of real zeros of $\phi_{M}$ larger than $1/2$ and let
$\{q_{n}\}_{n>N_{1}}$ denote the set of zeros of
$\phi_{M}$ with positive imaginary parts, counted with
multiplicities. In the notation of Definition 3.2.2 from \cite{Fis87}, we
set $\mathcal{P} _{M}\equiv 1$ if $n_{1}=0$, otherwise we define
$$
\mathcal{P}_M(s) := f_{1}(s) f_{2}(s)
$$
where
$$
f_{1}(s) := \prod \limits _{n=1}^{N_{1}} \left( 1+\frac{s-1/2}{q_n -1/2}\right) \exp \left[ \frac{1}{2} \left( \frac{s -1/2}{q_n -1/2} \right)^{2} \right]
$$
and
$$
f_{2}(s) := \prod _{n \geq N_1 +1}  \left( 1+\frac{s-1/2}{q_n -1/2}\right) \left( 1+\frac{s-1/2}{\overline{q_{n}} -1/2}\right) \exp \left[ \frac{1}{2} \left( \frac{s -1/2}{q_n -1/2} \right)^{2} + \frac{1}{2} \left( \frac{s -1/2}{\overline{q_{n}}-1/2} \right)^{2} \right].
$$

\noindent
The infinite product which defines $f_{2}$
converges uniformly on compact subsets of $\mathbb{C}$ and defines an entire function of finite order.

\vskip .10in
\begin{lemma} \label{ProdViaZ_M H_M}
For all $s\in \mathbb{C}$, the product $\left( \Xi_M \mathcal{P}_M\right)(1-s)$ can be expressed as
\begin{eqnarray} \label{prodformula}
\left( \Xi_M \mathcal{P}_M\right) (1-s) &=&(Z_M H_M)(s)\cdot \Xi
_{I}(s)\cdot \Xi
_{M, \mathrm{ell}}(s)\cdot \frac{\pi ^{n_{1/2}}d(1)}{\phi_M (1/2)}\mathfrak{g}
_{1}^{-s-1}\cdot  \\ \notag
&&\left( s-\frac{1}{2}\right) ^{\frac{1}{2}\textrm{\rm Tr}\left( I_{n_{1}}-\Phi_M (\frac{1}{%
2})\right) -n_{1}}\cdot \Gamma
(s)^{-n_{1}}\overset{N_0}{\underset{m=1}{\prod }}\left(
\frac{s-p_{m}}{1/2-p_{m}}\right). \notag
\end{eqnarray}
\end{lemma}

\vskip .10in
\begin{proof}
From the functional equation (3.2.4) on p. 123
of \cite{Fis87}, we have, for all $s\in \mathbb{C}$
\begin{equation}
\left( \Xi_M \mathcal{P}_M\right) (1-s)=\left( \Xi_M \mathcal{P}_M \right) (s)%
\mathfrak{g}_{1}^{2s-1}\overset{N_0}{\underset{m=1}{\prod }}\left( \frac{%
s-p_{m}}{1-s-p_{m}}\right) \frac{1}{\phi_M (1/2)}\phi_M (s).
\label{XIP}
\end{equation}
On the other hand, by the
Corollary 2.4.22 of \cite{Fis87} it is easy to see that
$$
\left( \Xi _{M, \mathrm{par}}\mathcal{P}_M\right) (s)=\left( s-1/2\right) ^{\frac{1}{2}%
\textrm{\rm Tr}\left( I_{n_{1}}-\Phi_M (\frac{1}{2})\right)
}\mathfrak{g}_{1}^{-s}\left(
\frac{1}{\Gamma (s+1/2)}\right) ^{n_{1}}\overset{N_0}{\underset{m=1}{\prod }}%
\left( 1+\frac{s-1/2}{p_{m}-1/2}\right).
$$

We now write $\phi_M$ as
$$
\phi_M (s) = \pi ^{n_1 /2} \mathfrak{g}_1 ^{-2s} d(1)
(s-1/2)^{-n_1} \left( \frac{\Gamma (s+1/2)}{\Gamma (s)} \right)^{
n_1} H_M (s).
$$
The result follows through direct and straightforward computations involving
the the definition of $\Xi _M$ together with \eqref{XIP}.
\end{proof}

\vskip .20in
\subsection{Littlewood's theorem}

\vskip .10in
Several components of the main theorem will be provided using a
general theorem from complex analysis due to Littlewood, which we
now quote from \cite{Ti}.

\vskip .10in Let $f(z)$ be a meromorphic functions which is
non-zero and has $n$ poles along the rectangular contour ${\cal C}$
which is bounded by the lines $x = x_{1}$, $x=x_{2}$, $y=y_{1}$
and $y=y_{2}$.  Let $F(z) = \log f(z)$ be the logarithm of $f(z)$
defined by analytic continuation along ${\cal C}$, and
$N(x';f,{\cal C})$ denote the number of zeros of $f$ minus the
number of poles of $f$ in the sub-region of ${\cal C}$ where $x >
x'$.  Then
\begin{equation}\label{littlewood}
\int\limits_{\cal C}F(z) dz = -2 \pi i \int\limits_{x_{1}}^{x_{2}}
N(x;f,{\cal C})dx.
\end{equation}

\noindent We refer the reader to \cite{Ti} for an
elementary proof of (\ref{littlewood}).

\vskip .20in
\subsection{On generalized Backlund equivalent for the Lindel\"{o}f hypothesis}

\vskip .10in
An important ingredient in the proof of the Main Theorem is a bound on the growth of the
function $Z_M H_M$ on the critical line $\textrm{\rm Re}(s) = 1/2$.
We obtain the bound using a slight modification of Proposition 2 from \cite{Gar07}, which we now state.

\vskip .10in
\begin{proposition} \label{Lindel}
Let $f(s)$ be a mermorphic function for all $s \in \mathbb C$ which is holomorphic in the
region for $\vert \Im (s) \vert \geq t_0 >0$, for some fixed $t_0$.  Let
$P(t): \mathbb{R} \rightarrow \mathbb{R}$ be non-decreasing function such that $P(t)\geq 2$.
Let $N(\sigma, f, T)$ denote the number of zeros $\rho$ of $f$ in the region
$\Re (\rho) > \sigma$; $0 \leq \Im (\rho) \leq T$.

\vskip .10in
Assume there exist constants $\sigma_0>1/2$ and $\omega >0$ such that for
$\sigma_0 - \omega \leq \Re (s) \leq \sigma_0 + \omega$ we have

\vskip .10in
$$
\vert f(s) \vert \geq c >0 \,\,\,\,\,\textrm{\rm  and} ,\,\,\,\,\ (f'/f)(s) = o(P(t))
\,\,\,\,\,\textrm{\rm as $t = \textrm{\rm Im}(s) \rightarrow \infty$.}
$$

\vskip .10in
\noindent
Furthermore, assume that $\vert f(s) \vert >0$ for $\Re (s) \geq \sigma_0 + \omega$ and that for some fixed number $D>0$ we have
$$
f(s) = O\left( (P( t)))^{D}\right)
\,\,\,\,\,\textrm{\rm as $t = \textrm{\rm Im}(s) \rightarrow \infty$, uniformly for $\Re (s) \geq 2-3\sigma_0$.}
$$

\vskip .10in
Then if the estimate $N(\sigma, f, T+1) - N(\sigma, f, T) = o(P(T))$holds for all
$\sigma \geq 1/2$ as $T \rightarrow \infty$, then
$f(1/2 + it) = O_{\epsilon} \left((P(t))^{\epsilon}\right)$ as $t \rightarrow \infty$.
\end{proposition}

\vskip .10in
There are slight differences between our statement above and Proposition 2 from \cite{Gar07}.  Firstly,
we assume the function $f$ depends on a single complex variable $s$, not necessarily a member of
a family of mermorphic functions.  Secondly, the author of \cite{Gar07}
assumes that $f$ has a finite number of poles which lie in a compact set, with the
goal of applying the result to the Selberg zeta function associated to a compact hyperbolic Riemann surface.
A review of the proof of Proposition 2 from \cite{Gar07} reveals that the argument is based on Landau's theorem
(see Lemma 8 of \cite{Gar07}) and Hadamard's three circles theorem (see Lemma 9 of \cite{Gar07}).  These classical
results are applied to the function $f(s)$ in the neighborhood $\vert s-s_0 \vert \leq 2(\sigma_0 - 1/2 -\delta)$
of the point $s_0=\sigma_0+iT$ for sufficiently large $T$.  The proof given in \cite{Gar07} carries
through without any changes whatsoever under the assumptions we state above.

\noindent
\vskip .10in
We refer the reader to \cite{Gar07} for the proof and various interesting
generalizations of Proposition \ref{Lindel}.

\vskip .20in
\section{Zeros in a half plane $\Re (s) <1/2$}

\vskip .10in
In this section we will prove part (a) of the Main Theorem. In
fact, we will prove more than stated, since our analysis will
yield regions where each of the functions $\textrm{Re}((Z_{M}H_{M})')$ and
$\textrm{Im}((Z_{M}H_{M})')$ are non-vanishing.

\vskip .10in
\begin{proposition} \label{LemmaNonVan}

\vskip .05in
a) For $\sigma <1/2$, there exists $t_{0}>0$, which may depend on $\sigma$,
such that
$$
\Re \left( \left( Z_M H_M \right)
^{\prime }(\sigma +it)\right) \neq 0
\,\,\,\,\,\textrm{\rm for all $t$ such that $\vert t \vert > t_{0}$.}
$$

\vskip .10in
b) For every constant $C>0$ and arbitrary $ -C < \sigma_0' < 1/2$
there are at most finitely many zeros of $ (Z_M H_M)'(s)$ inside
the strip $-C \leq \Re (s) \leq \sigma_0'$
\end{proposition}

\begin{proof}
We first present the proof of part (a).  By taking the logarithmic derivative of
the functional equation \eqref{functeqMain} we get for $s=\sigma +it$ with $\sigma <1/2$,
the equation

\vskip .10in
\begin{equation} \label{Z_M H_MderEq}
\frac{\left( Z_M H_M\right) ^{\prime }}{\left( Z_M H_M\right)
}(\sigma +it)=\frac{\eta_M
^{\prime }}{\eta_M }(\sigma +it)-\frac{K_M^{\prime }}{K_M}(\sigma +it)-\frac{%
Z_M^{\prime }}{Z_M}(1-\sigma -it).
\end{equation}

\vskip .10in
\noindent
From the definition \eqref{DefEtaLogDer} of $\eta_M'/\eta_M$ and $K_{M}$, one
 can use Stirling's formula, together with the bound $0<\theta < \pi $, to show that

\vskip .10in
$$
\text{Re }\left ( \frac{\eta_M ^{\prime }}{\eta_M }(\sigma
+it)\right )=- \mathrm{vol}(M) t+O(\log \vert t \vert)
\,\,\,\,\,\textrm{\rm and}\,\,\,\,\,
\frac{K_M^{\prime }}{K_M}(\sigma +it)=O(\log \vert t \vert)\text{,}
$$
for $\sigma < 1/2$ and as $\vert t\vert \rightarrow \infty$.  Therefore,
\vskip .10in
$$
\text{Re }\left( -\frac{\left( Z_M H_M\right) ^{\prime }}{\left( Z_M H_M\right) }%
(\sigma +it)\right) =\mathrm{vol}(M) t+O(\log \vert t \vert)+\text{Re }%
\left( \frac{Z_M^{\prime }}{Z_M}(1-\sigma -it)\right).
$$

\vskip .10in
\noindent
Replacing $\sigma$ by $1/2 - \sigma$ in \eqref{BoundDer} we get, for $\sigma <1/2$,
$$\Re \left( \frac{Z_M^{\prime
}}{Z_M}(1-\sigma -it)\right) = O \left(\frac{ \vert t \vert}{(1/2-\sigma) \log \vert t\vert} \right)
\,\,\,\,\,\textrm{\rm as $\vert t\vert \rightarrow \infty$},
$$
so then
\vskip .10in
\begin{equation*}
\Re \left ( - \frac{\left( Z_M H_M\right) ^{\prime }}{\left( Z_M
H_M\right) }(\sigma +it) \right )= \mathrm{vol}(M) t+ O \left(\frac{ \vert t \vert}{(1/2-\sigma) \log \vert t\vert} \right),
\end{equation*}

\vskip .10in
\noindent
for $\sigma <1/2$ as $ t\rightarrow \pm \infty $. Therefore,
there exists $t_0 >0$ such that
$$
\Re \left ( \frac{\left( Z_M H_M\right) ^{\prime }}{\left( Z_M H_M\right) }(\sigma +it) \right
) \neq 0
\,\,\,\,\,\textrm{\rm for all $s= \sigma +it$, with $\vert t \vert > t_0$.}
$$

\vskip .10in
On the other hand, the non-trivial zeros of the function $Z_M H_M$ are either non-trivial
zeros $\rho =\frac{1}{2}\pm ir_{n}$ of $Z_M$ or zeros $\rho $ of $\phi_M $.  All except finitely
many zeros of $\phi_M $ have real part bigger than $1/2$; therefore, $Z_M H_M (\sigma +it)\neq 0$
for $\sigma <1/2$ and $t>t_{0}$. Therefore, we conclude that
$\Re \left( \left( Z_M H_M \right) ^{\prime }(\sigma +it) \right)\neq 0$ for all $t>t_{0}$.
With all this, the proof of part a) is complete.

\vskip .10in
To prove part (b), we employ Lemma \ref{lemmaFuctEq}. Recall the function
$\widetilde{Z}_M(s)$ which is defined in (\ref{ZtildaDef}).  Let us write
$\widetilde{Z}_M(s)=1+Z_{M,1}(s)$. Then

\vskip .10in
\begin{align} \label{Z1def}
-f_{M}(s)&Z_{M,1}(s)= \pi \underset{0<\theta (R)<\pi
}{\underset{\left\{
R\right\} }{\sum }}\frac{1}{M_{R}\sin \theta }\frac{\cos (2\theta -\pi)(s-1/2)}{\cos \pi(s-1/2)}-2n_{1}\log 2
 \\
 &- n_{1}\left( \frac{\Gamma ^{\prime }}{\Gamma }(1/2+s)+\frac{%
\Gamma ^{\prime }}{\Gamma }(3/2-s) - \frac{\Gamma '}{ \Gamma}
(1/2-s) + \frac{\Gamma '}{ \Gamma} (1-s) \right)+ \frac{Z_M '}{Z_M}(s), \notag
\end{align}

\vskip .10in
\noindent
where $f_{M}$ is defined in (\ref{definitOf f_M}).

\vskip .10in
As in the proof of part (a), we can use Stirling's formula and \eqref{BoundDer}
to arrive at the bound
\vskip .10in
$$
Z_{M,1}(\sigma_1 + it) = O\left( \frac{(\vert t \vert \log \vert t \vert)^{2-2\sigma_1}}
{(\sigma_1 -1/2) \vert t \vert }\right)
\,\,\,\,\,\textrm{\rm as $\vert t \vert \rightarrow \infty$.}
$$

\vskip .10in
\noindent
for $\sigma_1 >1/2$ and $(\sigma_1 + it) \in  \mathbb{C}\backslash \underset{n\in \mathbb{Z}}{\cup }B_{n}$
where $B_{n}$ are small circles of fixed radius centered at integers.
In particular, for $\sigma_1
>1/2$ and $(\sigma_1 + it) \in  \mathbb{C}\backslash \underset{n\in \mathbb{Z}}{%
\cup }B_{n}$, function $Z_{M,1}(\sigma_1 +iT)$ is uniformly
bounded in $T$.  Therefore, $\widetilde{Z}_M(1-s)=1+o(1)$, as $\Im (s) \rightarrow \pm \infty$
in the strip $-C \leq \Re (s) \leq \sigma_0' <1/2$, hence $\widetilde{Z}_M(1-s)$ has
finitely many zeros in this strip.

\vskip .10in
Since  $Z_{M}(s)$ has only finitely many zeros for $\Re (s)>1/2$, the function $\widetilde{Z}_M(1-s)Z_M(1-s)$
will have finitely many zeros in the strip $-C \leq \Re (s) \leq \sigma_0' <1/2$.
Equation \eqref{DerivFunctEq} then implies that the set of zeros of $(Z_M H_M)'(s)$ in the strip $-C \leq \Re (s) \leq \sigma_0' <1/2$ is finite since the factor $\Phi_M(s):= f_M(s) \eta_M(s) K_M^{-1}(s) $ of the functional equation \eqref{DerivFunctEq} also has at most finitely many zeros in this strip.
\end{proof}

\vskip .10in
We note that the zeros of $(Z_M H_M)'(s)$ which arise from zeros of $\Phi_M$ can be viewed as trivial zeros.
The trivial zeros of $(Z_M H_M)'(s)$ are in the region $\Re (s)<1/2$ and arise at all negative integers.

\vskip .10in
The above method of examining zeros of the function $(Z_M H_M)'$
has the critical line as its limitation, since the integral
representation \eqref{LogDerDva} and the bounds for the
logarithmic derivative \eqref{BoundDer} hold true only in the half
plane $\Re (s)>1/2$. In order to derive results valid on the
critical line we need a representation on the critical
line. Such a representation exists for the complete zeta function
$\Xi_M (s)$ recalled in section 2.7.

\vskip .10in
\begin{proposition} \label{LemmaCritLine} There exists a number $t_{0}>0$ such that the
following statements hold:

\vskip .10in
a) $\displaystyle \frac{\left( Z_M H_M\right) ^{\prime }}{(Z_M H_M)}(\sigma +it)\neq 0$
for all $\sigma< 1/2$ and all $\vert t \vert > t_{0}$;

\vskip .10in
b) $\displaystyle \frac{\left( Z_M H_M\right) ^{\prime }}{(Z_M H_M)}(1/2 +it)\neq 0$
for all $\vert t \vert > t_{0}$, $t \neq r_n$ for all $n \geq 0$.
\end{proposition}

\vskip .10in
\begin{proof}
Let $\sigma <0$.  By Proposition \ref{LemmaNonVan}, there exists a constant
$t_0'>0$ such that for all $\sigma<0$ and all $ \vert t \vert
>t_0'$ we have $\displaystyle \frac{\left( Z_M H_M\right) ^{\prime }}{(Z_M H_M)}(\sigma +it)\neq
0$. Therefore, it is enough to prove the statement when $0\leq
\sigma < 1/2$. Without loss of generality, we will assume that $t>0$.

\vskip .10in
By taking the logarithmic derivatives of the both
sides of the equation \eqref{prodformula}, we get

\vskip .10in
\begin{align}
\frac{1}{2s-1}&\frac{\left( Z_M H_M\right) ^{\prime }}{\left( Z_M H_M\right) }(s) =-
\frac{1}{2s-1}\frac{\Xi _{I}^{\prime }(s)}{\Xi _{I}(s)}-
\frac{1}{2s-1}\frac{\Xi _{M, \mathrm{ell}}^{\prime }(s)}{\Xi _{M, \mathrm{ell}}(s)} +
\frac{\log \mathfrak{g}_{1}}{2s-1}+
\frac{n_{1}}{2s-1}\frac{\Gamma ^{\prime }}{\Gamma }(s)  \label{Z_M H_Mlogder} \\
& + \frac{n_{1}-\frac{1}{2}\textrm{\rm Tr}\left(I_{n_{1}}-\Phi (\frac{1}{2})\right)}
{2\left( s-1/2\right)^{2}}
-\frac{1}{2s-1}\sum\limits_{m=1}^{N_0}\frac{1}{s-p_{m}}-
\frac{1}{2s-1}
\frac{\left( \Xi_M \mathcal{P}_M\right) ^{\prime}}
{\left( \Xi_M \mathcal{P}_M\right)}(1-s),\notag
\end{align}

\vskip .10in
\noindent
for all $s\in \mathbb{C}$ different from zeros and poles of $Z_M$ and $\phi_M $.
Formula (3.4.1) on page 146 of \cite{Fis87} gives the formula
\begin{align}
\left( \Xi_M \mathcal{P}_M\right) (s) &=e^{Q(s)}\left( s-1/2\right) ^{2d_{1/4}}%
\underset{n\geq 0\text{, }r_{n}\neq 0}{\prod }\left( 1+\frac{(s-1/2)^{2}}{%
r_{n}^{2}}\right) \exp \left( -\frac{(s-1/2)^{2}}{r_{n}^{2}}\right) \notag \\&
\cdot \overset{N1}{\underset{n=1}{\prod }}\left( 1-\frac{s-1/2}{\eta _{n}}%
\right) \exp \left( -\frac{s-1/2}{\eta _{n}}+\frac{\left( s-1/2\right) ^{2}}{%
2\eta _{n}^{2}}\right) \notag  \\
 \underset{n\geq N_1 +1}{\prod }&\left( 1+\frac{s-1/2}{\eta
_{n}+i\gamma _{n}}\right)\left( 1+\frac{s-1/2}{\eta _{n}-i\gamma
_{n}}\right) \exp \left( -\frac{2\eta _{n}\left(
s-\frac{1}{2}\right) }{\eta _{n}^{2}+\gamma
_{n}^{2}}+\left( s-\frac{1}{2}\right) ^{2}\frac{\eta _{n}^{2}-\gamma _{n}^{2}%
}{\left( \eta _{n}^{2}+\gamma _{n}^{2}\right) ^{2}}\right), \notag
\end{align}%

\vskip .10in
\noindent
for all $s\in \mathbb{C}$ and where the notation is as follows:  $\eta_n := \Re (q_n)$;
$\gamma_n := \Im (q_n)$, $d_{1/4}$ is the multiplicity of $\lambda =1/4$ as an eigenvalue; and $Q(s)=a_{2}(s-1/2)^{2}+a_{1}(s-1/2)+a_{0}$ for some constants $a_{i}$, $i=0,1,2$ computed in \cite{Fis87}.
The constants $a_1$ and $a_2$ are defined by formulas (3.4.8) and
(3.4.9) on page 153 of \cite{Fis87}.  For our purposes it is important to know that $a_{1}$ and $a_{2}$ are real.

\vskip .10in
We now compute the logarithmic derivative of $\left( \Xi_M \mathcal{P}_M%
\right) (s)$ and substitute the expression into (\ref{Z_M H_Mlogder}).
After some elementary calculations, having in mind (\ref{XiIdDerivative}) and (\ref%
{XiEllDerivative}) we end up with

\vskip .10in
\begin{align}
\frac{1}{2s-1}\frac{\left( Z_M H_M\right) ^{\prime }}{\left( Z_M H_M\right) }(s)& =\frac{%
\mathrm{vol}(M) }{2\pi }\frac{\Gamma ^{\prime }}{\Gamma }%
(s)+O\left( \frac{\log \left\vert t\right\vert }{\left\vert t\right\vert }%
\right) +\frac{\log \mathfrak{g}_{1}}{2s-1}+\frac{n_{1}}{2s-1}\frac{\Gamma
^{\prime }}{\Gamma }(s)  \label{Z_M H_MderivFinal}
\\& +\frac{n_{1}-\frac{1}{2}\textrm{\rm Tr}\left( I_{n_{1}}-\Phi (\frac{1}{2})\right) }{%
2\left( s-1/2\right) ^{2}}-\frac{1}{2s-1}\underset{m=1}{\overset{N_0}{\sum }}%
\frac{1}{s-p_{m}}+a_{2}-\frac{a_{1}}{2\left( s-\frac{1}{2}\right) }+\frac{%
d_{1/4}}{\left( s-\frac{1}{2}\right) ^{2}}\notag
\\
&+\underset{n\geq 0\text{, }r_{n}\neq 0}{\sum }\left( \frac{1}{\left( s-%
\frac{1}{2}\right) ^{2}+r_{n}^{2}}-\frac{1}{r_{n}^{2}}\right) +\frac{1}{2}%
\underset{n=1}{\overset{N_1}{\sum }}\left( \frac{1}{\eta _{n}^{2}}-\frac{1}{%
\eta _{n}\left( \eta _{n}-s+1/2\right) }\right) \notag\\[3mm]
&+\underset{n\geq N_1 +1}{\sum }\left[ \frac{\eta _{n}^{2}-\gamma _{n}^{2}}{%
\left( \eta _{n}^{2}+\gamma _{n}^{2}\right) ^{2}}+\frac{\gamma _{n}^{2}-\eta
_{n}^{2}+\eta _{n}\left( s-1/2\right) }{\left( \left( \eta _{n}-s+1/2\right)
^{2}+\gamma _{n}^{2}\right) \left( \eta _{n}^{2}+\gamma _{n}^{2}\right) }%
\right].\notag
\end{align}

\vskip .10in
\noindent
Since $(\Gamma'/\Gamma) (\sigma + it) = O(\log \vert \sigma +it
\vert)$, as $t \rightarrow \infty$, from \eqref{Z_M H_MderivFinal}
we get

\vskip .10in
\begin{align}
\frac{1}{2s-1}\frac{\left( Z_M H_M\right) ^{\prime }}{\left( Z_M H_M\right) }(s)&=\frac{%
\mathrm{vol}(M) }{2\pi }\frac{\Gamma ^{\prime }}{\Gamma }%
(s) + \underset{n\geq 0\text{, }r_{n}\neq 0}{\sum }\left( \frac{1}{\left( s-%
\frac{1}{2}\right) ^{2}+r_{n}^{2}}-\frac{1}{r_{n}^{2}}\right)\notag
\\&
+a_2 + \underset{n\geq N_{1}+1}{\sum }\left[ \frac{\eta _{n}^{2}-\gamma _{n}^{2}}{%
\left( \eta _{n}^{2}+\gamma _{n}^{2}\right) ^{2}}+\frac{\gamma _{n}^{2}-\eta
_{n}^{2}+\eta _{n}\left( s-1/2\right) }{\left( \left( \eta _{n}-s+1/2\right)
^{2}+\gamma _{n}^{2}\right) \left( \eta _{n}^{2}+\gamma _{n}^{2}\right) }%
\right] \notag \\&
+ \frac{1}{2}%
\underset{n=1}{\overset{N_{1}}{\sum }}\left( \frac{1}{\eta _{n}^{2}}-\frac{1}{%
\eta _{n}\left( \eta _{n}-s+1/2\right) }\right) +
O\left( \frac{\log \left\vert t\right\vert }{\left\vert t\right\vert }%
\right),\label{prelimsum}
\end{align}

\vskip .10in
\noindent
as $t=\Im (s) \rightarrow \infty$.  We now set $s=\sigma+it$ with $t>0$ and $0\leq\sigma < 1/2$.
By computing the imaginary parts of both sides (\ref{prelimsum}) we get

\vskip .10in
\begin{align}\label{imequation}
\textrm{Im}&\left( \frac{1}{2\sigma-1 +2it}\frac{\left( Z_M
H_M\right) ^{\prime }}{\left( Z_M H_M\right) }(\sigma+it)\right)
=\frac{\mathrm{vol}(M)
}{2\pi }\cdot\left[ \frac{t}{\sigma^{2} + t^{2}} + \sum_{n=1}^{\infty}
\frac{t}{(n+\sigma)^{2} + t^{2}} \right] \notag \\[3mm]&
\hskip .5in + \underset{n\geq 0\text{, }r_{n}\neq 0}{\sum } \frac{t(1/2 - \sigma)}
{\left( (\sigma - 1/2)^{2}-t^{2} +r_n^{2} \right) ^{2} +4t^{2} (\sigma-1/2)^{2}} +
O\left(\frac{1}{t} \right)+ O\left( \frac{%
\log t }{t}\right) \notag \\[3mm]  &
\hskip .5in + \underset{n\geq N_1 +1}{\sum } \frac{ t\eta_n(3\gamma_n^{2}-t^{2})-t\eta_n^{3} +t(1/2 - \sigma)(2\gamma_n^{2}-2\eta_n^{2} -\eta_n (1/2-\sigma))}{\left [\left( \left( \eta _{n}-\sigma+1/2\right)
^{2}+\gamma _{n}^{2} -t^{2}\right)^{2} +4t^{2}(\eta_n+1/2-\sigma)^{2} \right] \left( \eta _{n}^{2}+\gamma _{n}^{2}\right) }.
\end{align}

\vskip .10in
\noindent
Since $0\leq\sigma < 1/2$ we have that $(n+\sigma)^{2} < (n+1/2)^{2}$, for all $n\geq0$. Therefore
\vskip .10in
$$
\frac{t}{\sigma^{2} + t^{2}} + \sum_{n=1}^{\infty} \frac{t}{(n+\sigma)^{2} + t^{2}} > \frac{t}{1/4 + t^{2}} + \sum_{n=1}^{\infty} \frac{t}{(n+1/2)^{2} + t^{2}} =\frac{\pi}{2} \tanh (\pi t).
$$
\vskip .10in
\noindent
Furthermore, since $0< \eta_n < c$, for some positive constant $c$ and all $n\geq 1$ and $\gamma_n \rightarrow \infty$, as $n\rightarrow \infty$, by the choice of $\sigma$ we have that
$$
(1/2 - \sigma) (2\gamma_n^{2}-2\eta_n^{2} -\eta_n (1/2-\sigma)) \geq 2\gamma_n^{2}-2\eta_n^{2} -\eta_n \geq 0
$$
for all but finitely many $n\geq (N_1 +1)$. Let $n_1 \geq (N_1 +1)$ be an integer such that $2\gamma_n^{2}-2\eta_n^{2} -\eta_n \geq 0$ for all $n\geq n_1$. For simplicity, we will introduce the notation
\vskip .10in
$$
D(n, \sigma, t) = \left [\left( \left( \eta _{n}-\sigma+1/2\right)
^{2}+\gamma _{n}^{2} -t^{2}\right)^{2} +4t^{2}(\eta_n+1/2-\sigma)^{2} \right].
$$

\vskip .10in
From \eqref{imequation}, we conclude the existence a constant
$C_1>0$ and a positive number $t_0 >t_0'$ such that for all $t>t_0$,

\begin{align}
& \textrm{Im}\left( \frac{1}{2\sigma-1 +2it}\frac{\left( Z_M H_M\right) ^{\prime
}}{\left( Z_M H_M\right) }(\sigma +it)\right)
\geq \frac{\mathrm{vol}(M)
}{4 }\tanh(\pi t) - C_1 \frac{\log t}{t} \notag
\\[3mm] & \hskip .65in
+\underset{\vert \gamma_n \vert < t/\sqrt{3}}{\sum }\frac{
t\eta_n(3\gamma_n^{2}-t^{2})}{D(n, \sigma, t) \left( \eta _{n}^{2}+\gamma
_{n}^{2}\right) } + \underset{N_1+1 \leq n \leq n_1}{\sum }\frac{
(1/2 - \sigma)t(2\gamma_n^{2}-2\eta_n^{2}-\eta_n)}{D(n, \sigma, t) \left( \eta
_{n}^{2}+\gamma _{n}^{2}\right)}\notag
\\[3mm] & \hskip .65in + \underset{n \geq N+1}{\sum
}\frac{ - t\eta_n^{3}}{D(n, \sigma, t) \left( \eta _{n}^{2}+\gamma
_{n}^{2}\right) }.\label{threeSums}
\end{align}

\vskip .10in
\noindent
Observe that each term in each summand in (\ref{threeSums}) is negative.
We investigate separately each of the three sums on the right hand side of \eqref{threeSums}.

\vskip .10in
Since $(\eta_n -\sigma +1/2)^{2}$ is bounded by some constant, by enlarging $t_0$ if necessary, we get that $D(n,\sigma,t)\geq t^{4}/4$ for all $n$ such that $\vert \gamma_n \vert < t/\sqrt{3}$ and all $t>t_0$. Therefore,

\vskip .10in
\begin{align*}
0  \leq \underset{\vert \gamma_n \vert < t/\sqrt{3}}{\sum }\frac{ t\eta_n(t^{2}-3\gamma_n^{2})}{D(n,\sigma,t) \left( \eta _{n}^{2}+\gamma _{n}^{2}\right) } &\leq \underset{\vert \gamma_n \vert < t/\sqrt{3}}{\sum }\frac{ t^{3}\eta_n}{D(n,\sigma,t) \left( \eta _{n}^{2}+\gamma _{n}^{2}\right)} \\[3mm] &
\leq \frac{2}{t} \underset{\vert \gamma_n \vert < t/\sqrt{3}}{\sum }\frac{ 2\eta_n}{\left( \eta _{n}^{2}+\gamma _{n}^{2}\right)} =O(1/t) ,
\end{align*}

\vskip .01in
\noindent
as $t \rightarrow \infty$ because the series
$$
\underset{n\geq N_1 +1}{\sum }\frac{2\eta _{n}}{\left( \eta _{n}^{2}+\gamma
_{n}^{2}\right) }
$$
converges; see Corollary 2.4.17 of \cite{Fis87}.

\vskip .10in
For the second and the third sum in \eqref{threeSums} we use the elementary inequality
$$
D(n,\sigma,t)\geq 4t^{2}(\eta_n +1/2 -\sigma)^{2} \geq 4t^{2} \eta_n^{2}
$$
to deduce that second sum is $O(1/t)$ as $t \rightarrow \infty$.
For the third sum, we have that
$$
\underset{n \geq N_1 +1}{\sum }\frac{ - t\eta_n^{3}}{D(n,\sigma,t)
\left( \eta _{n}^{2}+\gamma _{n}^{2}\right) } \geq - \frac{1}{8t}
\underset{n\geq N_1 +1}{\sum }\frac{2\eta _{n}}{\left( \eta _{n}^{2}+\gamma_{n}^{2}\right) }
= O(1/t)
$$
as $t \rightarrow \infty$, again by Corollary 2.4.17 of \cite{Fis87}.

\vskip .10in
The three sums in \eqref{threeSums} are $O(1/t)$, as $t \rightarrow \infty$.
Hence, there exists a constant $C_2 >0$, such that for all $t>t_0$, $t\neq r_{n}$
and $0\leq \sigma \leq 1/2$ one has
\[
\textrm{Im}\left( \frac{1}{2\sigma -1 + 2it}\frac{\left( Z_M H_M\right) ^{\prime
}}{\left( Z_M H_M\right) }(\sigma +it)\right)
 \geq \frac{\mathrm{vol}(M) }{4 }\cdot \tanh \pi t - \frac{C_1 \log t + C_2}{t}.
\]%
Since $\tanh \pi t=1+O(e^{-\pi t})$, as $
t\rightarrow \infty $ we conclude that
$$
\Im \left( \frac{\left( Z_M H_M\right) ^{\prime
}}{\left( Z_M H_M\right) } (\sigma +it)\right) \neq 0 \,\,\,\,\,\textrm{\rm for $t>t_0$.}
$$
With all this, the proof of part (a) is complete.

\vskip .10in
We now prove part (b).  We put $s=1/2 +it$ for $t>0$ with $t\neq r_n$ in \eqref{imequation} to get
\begin{multline*}
\textrm{Im}\left( \frac{1}{2it}\frac{\left( Z_M
H_M\right) ^{\prime }}{\left( Z_M H_M\right) }(\frac{1}{2}+it)\right)
=\frac{\mathrm{vol}(M)
}{2\pi }\cdot\frac{\pi}{2} \tanh (\pi t)+ O\left( \frac{%
\log t }{t}\right) + \underset{n\geq N_1 +1}{\sum } \frac{ t\eta_n(3\gamma_n^{2}-t^{2})-t\eta_n^{3}}{D(n,1/2,t) \left( \eta _{n}^{2}+\gamma _{n}^{2}\right) }.
\end{multline*}
Analogously as in the proof of part a) we deduce that there exist a constant $t_1>0$ such that
\[
\textrm{Im}\left( \frac{1}{2it}\frac{\left( Z_M H_M\right) ^{\prime
}}{\left( Z_M H_M\right) }(\frac{1}{2} +it)\right)
 \geq \frac{\mathrm{vol}(M) }{4 }\cdot \tanh \pi t - \frac{C_1 \log t + C_2}{t}.
\]
for all $t>t_1$, $t\neq r_n$ and the proof of part (a) is complete.
\end{proof}

\vskip .10in
We can now give a proof of part (a) of the Main Theorem.

\vskip .10in
The function $(Z_M H_M)(s)$ has finitely many
non-trivial zeros in the region $\Re (s)<1/2$.  Combining this statement with
Proposition \ref{LemmaCritLine}(a) immediately implies there existence a of constant $t_0$
such that $(Z_M H_M)'(\sigma + it) \neq 0$ for $\sigma <1/2$ and $\vert t \vert>t_0$.

\vskip .10in
Proposition \ref{LemmaCritLine}b) yields that $\displaystyle \frac{\left(
Z_M H_M\right) ^{\prime }}{(Z_M H_M)}(1/2 +it)\neq 0$ for $\vert t \vert>t_0$, $t\neq r_n$ for all $n\geq1$.
Therefore, the only zeros of $\left( Z_M H_M \right)^{\prime }$ on the line Re $s=1/2$,
with at most a finite number of exceptions,
are multiple zeros of $\left( Z_M H_M \right) $, or, equivalently, multiple zeros of $Z_M$.

\vskip .20in
\section{Preliminary lemmas}

In this section we will prove some preliminary results needed in the proof of the Main Theorem.
Let quantity $A_M$, resp. $a_M$, be defined by \eqref{definitionofA} and \eqref{definitionofa1}, resp. \eqref{definitionofa1two}. Let $P_{00} \in \mathcal{H}(\Gamma)$ denote
the primitive hyperbolic element of $\Gamma$ with the property that $N(P_{00})=e^{\ell_{M,0}}$;
or equivalently, with the property that $N(P_{00}) = \min \{N(P) : P \in \mathcal{H}(\Gamma) \}$.
In the case when $A_M=e^{\ell_{M,0}}$ we may write $a_M$ in terms of the norm of $P_{00}$, namely $a_{M}=m_{M,0} \Lambda(P_{00})$.

\vskip .10in
The proof of parts b) and c) of our Main Theorem relies essentially on application of Littlewood's theorem
to the function $X_M (s)$ defined for all complex $s$ by

\vskip .10in
\begin{equation} \label{definofX_M}
X_M(s):=\frac{A_M^{s}}{a_{M}}(Z_M H_M)^{\prime }(s).
\end{equation}

\vskip .10in
\begin{lemma} \label{X_MboundInfty} There exist $\sigma _{1} > 1$ and a
constant $0<c_{\Gamma }<1$ such that for $\sigma =\Re(s)\geq \sigma _{1}$,
we have the asymptotic formula
\[
X_M(s)=1+O(c_{\Gamma }^{\sigma })\neq 0%
\text{,}
\]
as $ \sigma \rightarrow +\infty$.
\end{lemma}

\vskip .10in
\begin{proof} From the Euler product definition \eqref{EulProd} of $Z_M$, we have that
$$
Z_M(s)=1+O \left(\frac{1}{N(P_{00})^{\Re (s)}}\right ),
$$
as $\Re (s) \rightarrow \infty$. Analogously, from \eqref{Dirichletseriespart}, we have that
$$
H_M(s)=1+O \left(\frac{1}{r_{2}^{2\Re (s)}} \right)
$$
as $\Re (s) \rightarrow +\infty$. Furthermore,
\begin{gather*}
\underset{\left\{ P\right\}  \in \mathcal{H}(\Gamma)}{\sum }\frac{\Lambda (P)}{
N(P)^{s}}+\sum_{i=1}^{\infty} \frac{b\left( q_{i}\right) }{q_{i}^{s}}
 =\frac{m_{M,0} \Lambda(P_{00})}{N(P_{00})^{s}}\left( 1 + O\left (\frac{1}
 {C_{\Gamma , 1}^{\Re (s)}}\right ) \right) + \frac{b((\mathfrak{g}_2 /\mathfrak{g}_1)^{2})}
 {(\mathfrak{g}_2 /\mathfrak{g}_1)^{2s}}\left( 1 + O\left (\frac{1}
 {C_{\Gamma , 2}^{\Re (s)}}\right ) \right)
\end{gather*}
as $\Re (s) \rightarrow +\infty$, where
$$
C_{\Gamma ,1} := \min \{N(P): N(P)\neq N(P_{00}) \} /N(P_{00}) >1
$$
and
$$
C_{\Gamma ,2} :=
\min \{ q_i : q_i \neq (\mathfrak{g}_2 /\mathfrak{g}_1)^{2}\} /(\mathfrak{g}_2 /\mathfrak{g}_1)^{2} >1
$$
are constants depending on the underlying group $\Gamma$.
By the definition of $A_M$ and $a_M$, we immediately deduce that
$$
\underset{\left\{ P\right\}  \in \mathcal{H}(\Gamma)}{\sum }\frac{\Lambda (P)}{
N(P)^{s}}+\sum_{i=1}^{\infty}\frac{b\left( q_{i}\right) }{q_{i}^{s}} = \frac{a_{M}}{A_M^{s}} \left( 1+O \left(\frac{1}{A_{\Gamma, 1}^{\Re (s) }}\right)\right), \text{  as  } \Re (s) \rightarrow +\infty,
$$
for some constant $A_{\Gamma,1} >1$.
We now use Lemma \ref{SeriesRepLogDerZ_M H_M}, where we multiply
the formula \eqref{logderZ_M H_M} by $(Z_M H_M)(s)$, to get
$$
(Z_M H_M)^{\prime }(s)=\frac{a_{M}}{A_M^{s}} \left[ 1+O \left(\frac{1}{N(P_{00})^{\Re (s)}}\right )\right] \left[1+O \left(\frac{1}{r_{1}^{2\Re (s)}} \right) \right]\left[1+O \left(\frac{1}{A_{\Gamma, 1}^{\Re (s) }}\right)\right],
$$
as $\Re (s) \rightarrow +\infty$. This completes the proof.
\end{proof}

\vskip .10in
The following lemma provides the bound for the growth of the function $Z_{M,1}(s)$; recall that
$Z_{M,1}(s)$ is defined by \eqref{Z1def}.

\vskip .10in
\begin{lemma} \label{Z1growth} Let $0<a<1/2$ be an arbitrary real number and let $\sigma \geq 1-a$. Then
\begin{equation}
\log \left\vert 1+Z_{M, 1}(\sigma \pm iT)\right\vert =O\left( \left\vert
Z_{M,1}(\sigma \pm iT)\right\vert \right) =O\left(T^{2a-1} \log^{2a} T \right) ,
\label{logZ1}
\end{equation}%
as $T\rightarrow \infty $.
\end{lemma}

\vskip .10in
\begin{proof} From the bound \eqref{BoundDer} with $k=0$ and $\sigma \geq 1-a$, we get
\[
\frac{Z_M^{\prime }}{Z_M}(\sigma \pm iT)=O\left( (T \log T) ^{1-2(\sigma - 1/2)}\right) =  O\left( (T \log T) ^{2a}\right) \text{, as }
T\rightarrow \infty,
\]
where the implied constant depends only upon $M$ and $a$. We then can argue in the same manner
as in the proof of Proposition \ref{LemmaNonVan}b.  Namely, we apply
Stirling's formula and the above estimate, we get, for $s=\sigma \pm iT$ and $T \geq 1$, the estimate

\vskip .10in
$$
\left\vert \frac{1}{f_M (s)}\left( \frac{\eta_M ^{\prime }}{\eta_M }(s)-
\frac{K_M^{\prime }}{K_M}%
(1-s)-\frac{Z_M^{\prime }}{Z_M}(s)\right) -1\right\vert =O\left( \frac{\log T}{T}+
\frac{ (T \log T)^{2a}}{T}\right),
$$

\vskip .10in
\noindent
as $T\rightarrow \infty$.   This implies the bound (\ref{logZ1}) as claimed.
\end{proof}

\vskip .10in
The following lemma is a Phragm\'{e}n-Lindel\"{o}f type bound for $\left(
Z_M H_M\right) $. The bound will be used to derive a similar bound for $\left(
Z_M H_M\right) ^{\prime }$ using the Cauchy formula.

\vskip .10in
\begin{lemma} \label{PraghLindBound} Let $\sigma _{2}\geq 1$ be a fixed real number,
such that $-\sigma
_{2} $ is not a pole of $\left( Z_M H_M\right) $. Then, for an arbitrary $\delta
>0$

\noindent
a)
\[
\left( Z_M H_M\right) (\sigma +it)=O_{\Gamma }\left( \exp \left( \frac{1}{2}%
+\sigma _{2}+\delta \right) \mathrm{vol} (M) t\right) \text{,}
\]
b)
\[
Z_M(\sigma +it)=O_{\Gamma }\left( \exp \left( \frac{1}{2}+\sigma _{2}\right)
\mathrm{vol} (M) t\right)
\]
for $t\geq 1$, uniformly in $\sigma \leq -\sigma _{2}$.
\end{lemma}

\vskip .10in
\begin{proof} To prove part (a), we will
apply the Phragmen-Lindel\"{o}f theorem (e.g. Chapter 5.61.of \cite{Ti})
to the function

\vskip .10in
$$
F(s)=\left( Z_M H_M\right) (s)\exp %
\left[ \textrm{vol} (M)\left( \frac{1}{2}+\sigma _{2}+\delta \right)is\right]
$$

\vskip .10in
\noindent
which is an entire function of finite order at most two
in the sector $D:=\left\{ \pi /4\leq \arg (s+\sigma _{2})\leq \pi /2\right\}
$. Obviously, $\left( Z_M H_M\right) (s)=O(1)$ along the line $\arg (s+\sigma
_{2})=\pi /4$, since $\left( Z_M H_M\right) (\sigma + it)=O(1)$, for $\sigma >\sigma _{1}$ and $t \geq 1$; see
the proof of Lemma \ref{X_MboundInfty}. Therefore,

\vskip .10in
$$
\left\vert F(s)\right\vert =O\left(1\right) \,\,\,\,\,\textrm{\rm along the line $\arg (s+\sigma _{2})=\pi /4$.}
$$

\vskip .10in
\noindent
To determine the
behavior of the function $F(s)$ along the vertical line $\arg (s+\sigma
_{2})=\pi /2$, i.e. for $s=-\sigma _{2}+it$, $t\geq 0$, we use the
functional equation \eqref{functeqMain} to get
\begin{eqnarray} \notag
\left\vert F(-\sigma _{2}+it)\right\vert &=&\exp (-\left(
\frac{1}{2}+\sigma _{2} + \delta \right) \mathrm{vol}(M) t)\left\vert \eta
(-\sigma _{2}+it)\right\vert \left\vert K_M^{-1}\left( -\sigma
_{2}+it\right)
\right\vert \left\vert Z_M(1+\sigma _{2}-it)\right\vert  \\ \label{definF}
&=&\exp (-\left( \frac{1}{2}+\sigma _{2} + \delta \right) \mathrm{vol}(M) t)\left\vert \eta_M (-\sigma _{2}+it)\right\vert \left\vert
K_M^{-1}\left( -\sigma _{2}+it\right) \right\vert \cdot O(1)\text{,}
\end{eqnarray}
since $1+\sigma _{2}\geq 2$. It remains to estimate $\left\vert \eta_M
(-\sigma _{2}+it)\right\vert $. Firstly, by \eqref{DefEtaLogDer}
\begin{equation} \label{EtaEstimate}
\left\vert \eta_M (-\sigma _{2}+it)\right\vert
=\exp \left( \textrm{Re}\underset{%
1/2}{\overset{-\sigma _{2}+it}{\int }}\frac{\eta_M ^{\prime
}}{\eta_M } (u)du\right) = \exp \left(
\textrm{Re}\underset{1/2}{\overset{-\sigma _{2}+it}{\int }}
\mathrm{vol}(M) (u)\tan \pi u\cdot du\right)\cdot
\end{equation}%
\[
\cdot \exp \left( 2n_{1}\log 2\cdot \left( -\sigma _{2}-1/2\right)
\right) \cdot \exp \left( n_{1}\textrm{Re}\left(
\underset{1/2}{\overset{-\sigma
_{2}+it}{\int }}\left( \frac{\Gamma ^{\prime }}{\Gamma }\left( \frac{1}{2}%
+u\right) +\frac{\Gamma ^{\prime }}{\Gamma }\left( \frac{3}{2}-u\right)
\right) du\right) \right)
\]%
\begin{equation}
\cdot \exp \left( \textrm{Re}\left( -\underset{0<\theta (R)<\pi }{\underset{%
\left\{ R\right\} }{\sum }}\frac{\pi }{M_{R}\sin \theta }\underset{1/2}{%
\overset{-\sigma _{2}+it}{\int }}\frac{\cos (2\theta -\pi )(s-1/2)}{\cos \pi
(s-1/2)}ds\right) \right) .  \nonumber
\end{equation}

\vskip .10in
\noindent
Formula (4.4) from page 76, of \cite{Hej76}, following the notation from the previous
page, states the estimate

\vskip .10in
\begin{multline} \label{term1}
\exp \left( \textrm{Re}\underset{0}{\overset{-\sigma _{2}-1/2+it}{\int }}%
\mathrm{vol}(M) (u-1/2)\tan \pi (u-1/2)du\right) \\
=\exp \left( \Re \left( \frac{i}{2}\mathrm{vol}(M) \left( it-\sigma
_{2}-1/2\right) ^{2} \right)+ O(1)\right) =\exp \left( \mathrm{vol}(M) \left( \frac{1}{2}+\sigma _{2}\right) t+O(1)\right)
\end{multline}%

\vskip .10in
\noindent
as $t \rightarrow \infty$.  By Stirling's formula,
\begin{equation}\label{term2}
\textrm{Re}\left( \underset{1/2}{\overset{-\sigma _{2}+it}{\int }}\left( \frac{%
\Gamma ^{\prime }}{\Gamma }\left( \frac{1}{2}+u\right) +\frac{\Gamma
^{\prime }}{\Gamma }\left( \frac{3}{2}-u\right) \right) du\right) = O(\log (t))
\,\,\,\,\,\textrm{as $t \rightarrow \infty$.}
\end{equation}

The contribution of the elliptic elements is given by

\begin{equation*}
\textrm{Re}\underset{1/2}{\overset{-\sigma _{2}+it}{\int
}}\frac{\cos (2\theta -\pi )(s-1/2)}{\cos \pi (s-1/2)}ds=
\textrm{Re}\left( \underset{1/2}{\overset{-\sigma _{2}}{\int
}}\frac{\cos (2\theta -\pi )(\sigma -1/2+it)}{\cos \pi (\sigma
-1/2+it)}d\sigma \right) .
\end{equation*}

\vskip .10in
\noindent
Trivially, one has that
$$
\left\vert \frac{\cos (2\theta -\pi )(\sigma -1/2+it)}{\cos \pi (\sigma
-1/2+it)}\right\vert  = O\left( \frac{\exp \left( \left\vert 2\theta
-\pi \right\vert t\right) }{\exp (\pi t)}\right)
\,\,\,\,\,\textrm{as $t\rightarrow +\infty $.}
$$

\vskip .10in
\noindent
Since
$0<\theta <\pi $, $\left\vert 2\theta -\pi \right\vert t-\pi t<0$, hence

\vskip .10in
\begin{equation} \label{term3}
\textrm{Re}\left( -\underset{0<\theta (R)<\pi }{\underset{%
\left\{ R\right\} }{\sum }}\frac{\pi }{M_{R}\sin \theta }\underset{1/2}{%
\overset{-\sigma _{2}+it}{\int }}\frac{\cos (2\theta -\pi )(s-1/2)}{\cos \pi
(s-1/2)}ds\right) =O(1) \,\,\,\,\,\textrm{as $t\rightarrow +\infty $.}
\end{equation}

\vskip .10in
\noindent
Substituting \eqref{term1}, \eqref{term2} and \eqref{term3} into (\ref{EtaEstimate}) we get
\begin{equation} \label{etabound}
\left\vert \eta_M (-\sigma _{2}+it)\right\vert =O\left( \exp \left(
\mathrm{vol}(M) \left( \frac{1}{2}+\sigma _{2}\right)
t+O(1)\right) \right) \,\,\,\,\,\textrm{as $t\rightarrow +\infty $.}
\end{equation}

\vskip .10in
\noindent
Formula (6.1.45) from \cite{AbSt64}, which itself is an application of
Stirling's formula, yields
$$
\left( \frac{%
\Gamma (-\sigma _{2}+it)}{\Gamma (-\sigma _{2}-1/2+it)}\right) ^{n_{1}} = O\left( t ^{n_1/2}\right)
\,\,\,\,\,\textrm{as $t\rightarrow +\infty $.}
$$

\vskip .10in
\noindent
Therefore,
\begin{equation} \label{Kbound}
\left\vert K_M^{-1}\left( -\sigma _{2}+it\right) \right\vert =O(\exp \left(
\frac{n_{1}}{2}\log t\right) )\,\,\,\,\,\textrm{as $t\rightarrow +\infty $.}
\end{equation}

\vskip .10in
\noindent
Substituting the bound (\ref{Kbound}) together with \eqref{etabound} into \eqref{definF} we get
\begin{align*}
\left\vert F(-\sigma _{2}+it)\right\vert = & O
\left( \exp \left( \mathrm{vol}(M) \left( \frac{1}{2}+\sigma _{2}\right) t+\frac{n_{1}}{2%
}\log t+O(\log(t))\right)\right) \\ & \hskip .15in \cdot O\left(\exp \left( -\mathrm{vol}(M) \left(
\frac{1}{2}+\sigma _{2}+\delta \right) t\right)\right)  \\
= & o(1)\,\,\,\,\,\textrm{as $t\rightarrow +\infty $.}
\end{align*}

\vskip .10in
One now can apply the Phragmen-Lindel\"{o}f theorem, which implies that $F(s)=O(1)$
in the sector $D:=\left\{ \pi /4\leq \arg (s+\sigma _{2})\leq \pi /2\right\}$ and the proof of a) is complete.

\vskip .10in
To prove (b), we repeat the proof given above for the function

\vskip .10in
$$
G(s)=Z_M(s)\exp \left[
\left( \frac{1}{2}+\sigma _{2}\right) \mathrm{vol}(M) is\right],
$$

\vskip .10in
\noindent
which is an entire function of finite order in the sector $D:=\left\{ \pi /4\leq
\arg (s+\sigma _{2})\leq \pi /2\right\} $. Obviously, $\left| G(s)\right|
=O\left( 1\right) $ along the line $\arg (s+\sigma _{2})=\pi /4$. To
determine the behavior of the function $G(s)$ along the vertical line $\arg
(s+\sigma _{2})=\pi /2$, i.e. for $s=-\sigma _{2}+it$, $t\geq 0$, we use the
functional equation for the zeta function $Z_M$ and get

\vskip .10in
\begin{eqnarray*}
\left| G(-\sigma _{2}+it)\right| &=&\exp (-\left( \frac{1}{2}+\sigma
_{2}\right) \mathrm{vol}(M) t)\left| \eta_M (-\sigma _{2}+it)\right|
\left| \phi_M \left( 1+\sigma _{2}-it\right) \right| \left| Z_M(1+\sigma
_{2}-it)\right| \\
&=&\exp (-\left( \frac{1}{2}+\sigma _{2}\right) \mathrm{vol}(M)
t)\left| \eta_M (-\sigma _{2}+it)\right| \left| K_M\left( 1+\sigma
_{2}-it\right) \right| \cdot O(1)\,\,\,\,\,\textrm{as $t\rightarrow +\infty $,}
\end{eqnarray*}

\vskip .10in
\noindent
since $1+\sigma _{2}\geq 2$ and the Dirichlet series for $H_M$ and $Z_M$
converge absolutely for Re $s\geq 2$ (see, e.g. page 160 of \cite{Iw02}). The estimate for $\left| \eta_M
(-\sigma _{2}+it)\right| $ is given by formula \eqref{etabound}. Analogously as above, we see that
\[
\left| K_M\left( 1+\sigma _{2}-it\right) \right| =O(\exp \left( -\frac{n_{1}}{2%
}\log t\right) )\,\,\,\,\,\textrm{as $t\rightarrow +\infty $.}
\]
The proof is completed using the argument in the proof of part (a).
\end{proof}

\vskip .10in
The following lemma is a Lindel\"{o}f type bound for the function $Z_M H_M$ which will be
used to deduce a sharper bound for the function $\arg X_M (\sigma + iT)$, when $\sigma$ is close to $1/2$.

\vskip .10in
\begin{lemma} \label{LindelofBound} For $\epsilon >0$ and $t\geq 1$
\[
\left( Z_M H_M\right) \left( \frac{1}{2}+it\right) =O_{\epsilon }\left( \exp
\left( \epsilon t\right) \right)\,\,\,\,\,\textrm{as $t\rightarrow +\infty $.}
\]
\end{lemma}

\vskip .10in
\begin{proof} Let us write

\vskip .10in
$$
\left| \left( Z_M H_M\right) \left( \frac{1}{2}+it\right) \right| =\left| Z_M(\frac{%
1}{2}+it)\right| \left| \phi_M (\frac{1}{2}+it)\right| \left| K_M^{-1}(\frac{1}{2%
}+it)\right|.
$$

\vskip .10in
\noindent
Since

\vskip .10in
$$
\left| \phi_M (\frac{1}{2}+it)\right| =1
\,\,\,\,\,
\textrm{\rm and}
\,\,\,\,\,
\left| K_M^{-1}(\frac{1}{2}%
+it)\right| =O(\exp (\frac{n_{1}}{2}\log t))
\,\,\,\,\,\textrm{\rm as $t \rightarrow +\infty$,}
$$

\vskip .10in
\noindent
it is enough to prove that

\vskip .10in
$$
Z_M\left( \frac{1}{2}+it\right) =O_{\epsilon }\left( \exp \left( \epsilon
t\right) \right) \,\,\,\,\,\textrm{\rm as $t \rightarrow +\infty$,}
$$

We apply Proposition \ref{Lindel}. In the notation of Proposition \ref{Lindel} we take
$$
f(s)=Z_{M}(s)\,\,\,\,\,\textrm{with}\,\,\,\,\,
\sigma_{0}=\sigma _{1}+\omega> \sigma_1\,\,\,\,\,\textrm{and}\,\,\,\,\,
P(t)=2\exp (t),
$$

\vskip .10in
\noindent
where $\sigma_1$ is defined in Lemma \ref{X_MboundInfty}.  Let us
verify that all assumptions of Proposition \ref{Lindel} are fulfilled.

\vskip .10in
The function $Z_{M}$ is meromorphic function of finite order, with poles at points on the real line;
see page 498 of \cite{Hej83}.  Hence $Z_{M}(s)$ is holomorphic function for
$\vert \Im (s) \vert \geq t_0>0$, for any $t_0 >0$.

\vskip .10in
From the proof of Lemma \ref{X_MboundInfty} it is obvious that $\left| Z_{M}(s)\right| \geq
c>0$ and
$$
Z_{M}'/Z_{M}(s)=O(1) \,\,\,\,\,\textrm{\rm as $t \rightarrow \infty$,}
$$

\vskip .10in
\noindent
for $s=\sigma +it$ and with $\sigma _{0}-\omega \leq \sigma \leq \sigma _{0}+\omega $.
Furthermore, $\left| Z_{M}(s)\right| >0$ for $\textrm{\rm Re}(s)>\sigma _{0}+\omega $.

\vskip .10in
From Lemma \ref{PraghLindBound}b), we have that
\[
Z_{M}(\sigma +it)=O_{\Gamma }\left( \exp \left( \frac{1}{2}+3\sigma
_{0}-2\right) \mathrm{vol}(M) t\right) =O_{\Gamma }\left(
P(t)^{D}\right) \text{,}
\]%
for a fixed $D =\left( 3\sigma _{0}-3/2\right) \mathrm{vol} (M) $, uniformly in $\sigma \geq 2-3\sigma _{0}$.

\vskip .10in
Theorem 5.3 from page 498 of \cite{Hej83} asserts that
$Z_{M}$ has no zeros in the half-plane $\textrm{\rm Re}(s) > 1/2$.  Therefore, the Lindel\"of condition
on the vertical distribution of zeros of $Z_{M}(s)$ in the half-plane $\textrm{\rm Re}(s) > 1/2$,
as required in Proposition \ref{Lindel}, is trivially fulfilled.

\vskip .10in
Therefore, all the assumptions of Proposition \ref{Lindel} are satisfied,
hence $Z_{M}(1/2 +it) =O_{\epsilon} \left(\exp(\epsilon t)\right)$ as $t \rightarrow \infty$.
\end{proof}

\vskip .10in
\begin{lemma} \label{derivboundArg} For an arbitrary $\epsilon>0$, $t \geq 1$ and
$\sigma_2$ defined in Lemma \ref{PraghLindBound} we have
\[
\left( Z_M H_M\right) ^{\prime }(\sigma +it)=\left\{
\begin{array}{ll}
O\left( \exp \epsilon t\right) &\,\,\,\,\,\text{for} \,\,\,\,\, \frac{1}{2}\leq
\sigma \leq \sigma _{0} \\[3mm]
O\left( \exp (1/2-\sigma +\epsilon)t\right) &\,\,\,\,\,\text{for} \,\,\,\,\,-\sigma _{2}\leq \sigma
< 1/2,%
\end{array}%
\right.
\]
as $t \rightarrow \infty$.
\end{lemma}

\vskip .10in
\begin{proof}
The proof involves an application of the
Phragmen-Lindel\"{o}f theorem to the open sector bounded by the lines
$$
\text{\rm Re}(s)=-\sigma _{2} \,\,\,\,\,\textrm{\rm and} \,\,\,\,\,\textrm{\rm Re}(s)=\frac{1}{2}
$$

\vskip .10in
\noindent
and $\text{\rm Im}(s)=1$. The bounds to be used come from Lemma \ref{PraghLindBound},
with $\delta = \epsilon$, and from Lemma \ref{LindelofBound}.  A direct application of
the Phragmen-Lindel\"{o}f theorem yields the bound

\vskip .10in
\begin{equation} \label{B1}
(Z_M H_M)(\sigma +it)=O\left( \exp (1/2 - \sigma +\epsilon)t\right),
\end{equation}

\vskip .10in
\noindent
for $t\geq 1$ and $-\sigma _{2}\leq \sigma \leq 1/2$. Similarly,
for $\sigma_0$ defined as in Lemma \ref{PraghLindBound}, one can apply the
Phragmen-Lindel\"{o}f theorem in the open sector bounded by the
lines $\text{\rm Re}(s)=\sigma _{0}$, $\text{\rm Re}(s)=\frac{1}{2}$ and $\text{\rm Im}(s)=1$, from
which one gets

\vskip .10in
\begin{equation} \label{B2}
Z_M H_M(\sigma +it)=O\left( \exp \epsilon t\right),
\end{equation}

\vskip .10in
\noindent
for $1/2\leq \sigma \leq \sigma _{0}$. The Cauchy integral
formula can be applied, from which we have the equation
$$
(Z_M H_M)^{\prime }(s)=\frac{1}{2\pi i}\underset{C}{\int }\frac{\left( Z_M H_M\right)
(z)}{\left( z-s\right) ^{2}}dz
$$
where $C$ is a circle of a small, fixed radius $r<\epsilon$, centered at $s = \sigma + it$. Applying \eqref{B2}  to $(Z_M H_M)(z)$, when $1/2 \leq \Re (z) \leq \sigma_0$ and \eqref{B1} when $\Re (z) < 1/2$, we deduce that
$$
(Z_M H_M)^{\prime }(\sigma + it) = O\left(\exp ((r + \epsilon)t)/r\right) = O\left(\exp (2 \epsilon t)\right)
$$

\vskip .10in
\noindent
for $1/2 \leq \sigma \leq \sigma_0$ and $t\geq 1$. This proves the first part of the Lemma when
replacing $\varepsilon$ by $\varepsilon/2$.

\vskip .10in
In the case when $\sigma < 1/2$, we can use the functional equation for $(Z_M H_M)'$ to arrive at the expression

\vskip .10in
\begin{eqnarray*}
\left\vert \left( Z_M H_M\right) ^{\prime }(-\sigma _{2}+it)\right\vert
&=&\left\vert \eta_M (-\sigma _{2}+it)\right\vert \left\vert K_M^{-1}(-\sigma
_{2}+it)\right\vert \left\vert Z_M(1+\sigma _{2}-it)\right\vert \cdot \\[3mm]
&&\cdot \left\vert \frac{\eta_M ^{\prime }}{\eta_M }(-\sigma _{2}+it)-\frac{%
K_M^{\prime }}{K_M}(-\sigma _{2}+it)-\frac{Z_M^{\prime }}{Z_M}(1+\sigma
_{2}-it)\right\vert.
\end{eqnarray*}

\vskip .10in
\noindent
Since $\sigma_2 \geq 1$, we have $Z_M '/Z_M (1+\sigma_2 - it) =O(1),$ as $t \rightarrow + \infty$. Elementary computations involving the definition of the function $\eta_M' / \eta_M$ and the Stirling formula imply that
\[
\frac{\eta_M ^{\prime }}{\eta_M }(-\sigma _{2}+it)-\frac{K_M^{\prime }}{K_M}(-\sigma
_{2}+it)-\frac{Z^{\prime }}{Z}(1+\sigma _{2}-it)=O(t)\,\,\,\,\,\textrm{\rm as $t \rightarrow \infty$;}
\]%
in brief, one sees the asymptotic bound by observing that the leading term in the above expression
is $ \mathrm{vol} (M) (1/2 + \sigma_2 -it) \tan (\pi (1/2 +\sigma_2 -it))$.
From the bounds \eqref{etabound} and \eqref{Kbound} obtained in the proof of Lemma \ref{PraghLindBound}, we arrive
at the bound

\vskip .10in
\[
\left\vert \left( Z_M H_M\right) ^{\prime }(-\sigma _{2}+it)\right\vert
=O\left( \exp \left( \left( \frac{1}{2}+\sigma _{2}+\epsilon \right)
\mathrm{vol}(M) t\right) \right)\,\,\,\,\,\text{\rm as $t\rightarrow \infty$.}
\]

\vskip .10in
\noindent
The bound claimed in the statement of the Lemma follows by
applying the Phragmen-Lindel\"{o}f theorem to the function $(Z_M H_M)'$ in the
open sector bounded by the lines $\Im (s)=1$, $\Re (s)= -\sigma_2$ and $\Re (s)=1/2$,
keeping in mind that $-\sigma_2 \leq \sigma < 1/2$.
\end{proof}

\vskip .20in
\section{Vertical and horizontal distribution of zeros}

\vskip .10in
In this section we will prove parts b) and c) of the Main Theorem.

\vskip .10in
We fix a large positive number $T$ and choose number $T'$ such that
$\vert T'-T\vert = O(1)$ independently of $T$ where no zero of $Z_{M}H_{M}$ has imaginary
part equal to $T'$.   Let $t_{0}>0$ be a number such that $(Z_M H_M)' / (Z_M H_M) (\sigma + it) \neq 0$ for all $\sigma < 1/2$ and $\vert t \vert < t_0$; the existence of such $t_0$ is established by Proposition \ref{LemmaCritLine}.  Let $\sigma_0 \geq 1$ be a constant chosen so that $\sigma_0 \geq \max\{ \sigma_0' , \sigma_1\}$, where $\sigma_0'$ is defined in Lemma \ref{SeriesRepLogDerZ_M H_M} and $\sigma_1$ is defined in Lemma \ref{X_MboundInfty}.  Let $0<a<1/2$ be arbitrary.

\vskip .10in
The function $X_{M}(s)$,  which was defined in \eqref{definofX_M}, is holomorphic
in the rectangle $R(a,T^{\prime })$ with vertices $a+it_{0}$, $\sigma _{0}+it_{0}$, $\sigma _{0}+iT^{\prime }$
and $a+iT^{\prime }$.  As in \cite{Luo05}, we will use Littlewood's theorem, as stated in section 2.8,
from which we get the formula

\vskip .10in
\begin{multline} \label{mainsum}
2 \pi \underset{t_{0}<\gamma <T^{\prime }, \text{  } \beta'
>a}{\underset{\rho ^{\prime }=\beta
^{\prime }+i\gamma }{\sum }}\left( \beta ^{\prime }-a\right) =\underset{t_{0}%
}{\overset{T^{\prime }}{\int }}\log \left\vert X_M(a+it)\right\vert dt-%
\underset{t_{0}}{\overset{T^{\prime }}{\int }}\log \left\vert X_M(\sigma
_{0}+it)\right\vert dt \\
-\underset{a}{\overset{\sigma _{0}}{\int }}\arg X_M(\sigma +it_{0})d\sigma +%
\underset{a}{\overset{\sigma _{0}}{\int }}\arg X_M(\sigma +iT^{\prime
})d\sigma =I_{1}-I_{2}-I_{3}+I_{4}.
\end{multline}

\vskip .10in
\noindent
The variable $\rho '$ denotes a zero of $(Z_M H_M)'$, and
the integrals $I_{1}$, $I_{2}$, $I_{3}$ and $I_{4}$ are defined to be the four
integrals in (\ref{mainsum}), in obvious notation.  By Proposition \ref{LemmaCritLine},
the condition that $\Im (\rho') >t_0$ implies that $\Re (\rho') \geq 1/2$, hence the sum on the left hand side
of \eqref{mainsum} is actually taken over all zeros of $(Z_M H_M)'$ with imaginary part in the interval
$(t_0, T')$.

\vskip .10in
We investigate integrals $I_1$, $I_2$, $I_3$ and $I_4$ separately.

\vskip .10in
Obviously, $I_{3}=O(1)$ as $T \rightarrow \infty$ since, in fact, $I_{3}$ is independent of $T$.  As
for $I_{2}$, we will follow the argument from page 1144 of \cite{Luo05}.  Let us write
$$
I_{2} = \underset{t_{0}}{\overset{T^{\prime }}{\int }}\log X_M(\sigma_{0}+it) dt
+ \underset{t_{0}}{\overset{T^{\prime }}{\int }}\arg X_M(\sigma_{0}+it) dt.
$$
The function $X_{M}(s)$ is holomorphic and non-vanishing in a half-plane which contains the line of integration,
so
$$
\underset{t_{0}}{\overset{T^{\prime }}{\int }}\arg X_M(\sigma_{0}+it) dt = O(1)
\,\,\,\,\,\textrm{\rm as $T \rightarrow \infty$.}
$$
By Cauchy's theorem,
$$
\underset{t_{0}}{\overset{T^{\prime }}{\int }}\log X_M(\sigma_{0}+it) dt = \int\limits_{\sigma_{0}}^{\infty}
\log X_{M}(\sigma_{0}+iT')d\sigma - \int\limits_{\sigma_{0}}^{\infty}\log X_{M}(\sigma + it_{0})d\sigma,
$$
which is bounded in $T'$ since the function $ \log X_M$ is holomorphic and bounded in the infinite strip $\{s \in \mathbb{C} : t_0 \leq \Im (s) \leq T' , \Re (s)\geq \sigma_0\}$.

\vskip .10in
It remains to evaluate $I_1$ and $I_4$.

\vskip .10in
\subsection{Evaluation of $I_1$}

\vskip .10in
We shall break apart further $I_{1}$ by using the functional equation \eqref{DerivFunctEq}
 for $\left( Z_M H_M\right) ^{\prime }$, the definition \eqref{definofX_M} of $X_M$,
and representation of $\widetilde{Z}_M (s)= 1+Z_{M,1}(s)$ which was used in
the proof of Proposition \ref{LemmaNonVan}b).  By doing so, we arrive at the expression

\vskip .10in
\begin{multline*}
I_{1}=-\underset{t_{0}}{\overset{T^{\prime }}{\int }}\log \left\vert
a_{M}A_M^{-(a+it)}\right\vert dt+\underset{t_{0}}{\overset{T^{\prime }}{\int }}%
\log \left\vert f_M (a+it) \eta_M (a+it) K_M ^{-1}
(a+it)\right\vert dt\\
+\underset{t_{0}}{\overset{T^{\prime }}{\int }}\log \left\vert Z_M\left( 1-(a+it)\right)
\right\vert dt+ \underset{t_{0}}{\overset{T^{\prime }}{\int }}\log
\left\vert 1+Z_{M,1}\left( 1-(a+it)\right) \right\vert
dt=I_{11}+I_{12}+I_{13}+I_{14},
\end{multline*}

\vskip .10in
\noindent
with the obvious notation for the integrals $I_{11}$, $I_{12}$, $I_{13}$ and $I_{14}$.
Clearly, we have that

\vskip .10in
\begin{equation}
I_{11}=-T\left( \log \left\vert a_{M}\right\vert -a\log A_M\right) +O(1)
\,\,\,\,\,\textrm{\rm as $T \rightarrow \infty$.}
\label{I11}
\end{equation}%

\vskip .10in
\noindent
From the computations on the bottom of page 1146 of \cite{Luo05}, we have that

\vskip.10in
\begin{equation} \label{Intf_m}
\underset{t_{0}}{\overset{T^{\prime }}{\int }}%
\log \left\vert f_M (a+it) \right\vert dt = T \log T + T\left(
\log \mathrm{vol}(M) -1\right) +O(\log T)
\,\,\,\,\,\textrm{\rm as $T \rightarrow \infty$.}
\end{equation}

\vskip .10in
\noindent
Therefore,

\vskip .10in
\begin{align} \label{I12begining}
I_{12}&=T\log T+T\left( \log \mathrm{vol}(M) -1\right) +O(\log
T) \\ \notag &+\underset{t_{0}}{\overset{T^{\prime }}{\int }}\log \left\vert
\eta_M (a+it)\right\vert dt
+\underset{t_{0}}{\overset{T^{\prime }}{\int }}\log \left\vert
K_M^{-1}(a+it)\right\vert dt\\ \notag &=T\log T+T\left( \log \mathrm{vol}(M) -1\right) +I_{121}+I_{122}+O(\log T)
\,\,\,\,\,\textrm{\rm as $T \rightarrow \infty$,}
\end{align}

\vskip .10in
\noindent
with obvious notation for $I_{121}$ and $I_{122}$.  Stirling's formula implies that
\[
\left\vert K_M^{-1}(a+it)\right\vert =\pi ^{-\frac{n_{1}}{2}}\exp
\left( -c_{1}a-\textrm{\rm Re}(c_{2})\right) \exp \left [n_{1}\left(
\frac{1}{2}\log \left\vert a-1/2+it\right\vert +O\left(
\frac{1}{t}\right) \right) \right] \left( 1+O(\frac{1}{t^{2}})\right),
\]
as $t \rightarrow \infty$, where $c_{1}$ and $c_{2}$ are constants
defined in (\ref{constants_c1_c2}).  Therefore,

\vskip .10in
\begin{align} \label{I122}
I_{122} &=\underset{t_{0}}{\overset{T^{\prime }}{\int }}\log
\left\vert K_M^{-1}(a+it)\right\vert dt=\left(
-c_{1}a-\textrm{Re}(c_{2})-\frac{n_{1}}{2}\log
\pi \right) T\notag \\&+\frac{n_{1}}{2}\underset{t_{0}}{\overset{T^{\prime }}{\int }}%
\log \left\vert a-1/2+it\right\vert \notag
+O(\log T) \\&=\frac{n_{1}}{2}T\log T-T\left( c_{1}a+\textrm{Re}c_{2}+\frac{%
n_{1}}{2}\log \pi +\frac{n_{1}}{2}\right) +O(\log T)
\,\,\,\,\,\textrm{\rm as $T \rightarrow \infty$.}
\end{align}

\vskip .10in
\noindent
Finally,
\begin{align} \label{I121}
I_{121}&=\underset{t_{0}}{\overset{T^{\prime }}{\int
}}\textrm{Re}\left(
\underset{1/2}{\overset{a+it}{\int }}\frac{\eta_M ^{\prime }}{\eta_M }%
(u)du\right) dt=\underset{t_{0}}{\overset{T^{\prime }}{\int }}\textrm{Re}
\left( \underset{1/2}{\overset{a+it}{\int }}\textrm{vol}(M) (u-1/2)\tan \pi (u-1/2)du\right) dt \notag \\&
\notag +2n_{1}\log 2\left( a-1/2\right) \left( T^{\prime }-t_{0}\right) +n_{1}%
\underset{t_{0}}{\overset{T^{\prime }}{\int }}\textrm{Re}\left( \underset{1/2}{%
\overset{a+it}{\int }}\left( \frac{\Gamma ^{\prime }}{\Gamma }(%
\frac{1}{2}+u)+\frac{\Gamma ^{\prime }}{\Gamma }(\frac{3}{2}-u)\right)
du\right) dt \notag \\&
+\textrm{\rm Re}\left( -\underset{0<\theta (R)<\pi }{\underset{\left\{ R\right\} }{%
\sum }}\frac{\pi }{M_{R}\sin \theta }\underset{1/2}{\overset{a+it}{%
\int }}\frac{\cos (2\theta -\pi )(s-1/2)}{\cos \pi (s-1/2)}ds\right)\notag \\ &
=\left( \frac{1}{2}-a\right) \frac{\mathrm{vol}(M) }{2}%
T^{2}+2n_{1}\log 2\left( a-1/2\right) T+O(\log T)
\,\,\,\,\,\textrm{\rm as $T \rightarrow \infty$.}
\end{align}

\vskip .10in
\noindent
As in the proof of Lemma \ref{PraghLindBound}, one can use
Stirling's formula to obtain the bound

\vskip .10in
$$
\underset{t_{0}}{\overset{T^{\prime }}{\int }}\textrm{Re}\left( \underset{1/2}{%
\overset{a+it}{\int }}\left( \frac{\Gamma ^{\prime }}{\Gamma }(%
\frac{1}{2}+u)+\frac{\Gamma ^{\prime }}{\Gamma }(\frac{3}{2}-u)\right)
du\right) dt=O(\log T)\,\,\,\,\,\textrm{\rm as $T \rightarrow \infty$.}
$$

\vskip .10in
\noindent
Elementary computations yield the same bound for the elliptic contribution.
By substituting \eqref{I121} and \eqref{I122} into \eqref{I12begining}, we arrive
at the bound

\vskip .10in
\begin{align}
I_{12} &= \left( \frac{1}{2}-a\right) \frac{\mathrm{vol}(M)}{2}T^{2}+
\left( \frac{n_{1}}{2}+1\right) T\log T+O(\log T)
\label{I12} \\
&+ T\left[ 2n_{1}\log 2\left( a-1/2\right) -c_{1}a+\log \mathrm{vol}(M) -1-\textrm{Re}(c_{2})-\frac{n_{1}}{2}\log \pi
-\frac{n_{1}}{2}\right] \,\,\,\,\,\textrm{\rm as $T \rightarrow \infty$.} \notag
\end{align}

\vskip .10in
\noindent
With all this, we arrive at the desired bound for $I_{12}$.

\vskip .10in
The integral $I_{13}$ is estimated by applying the Cauchy's
theorem to the function $\log Z_M(s)$ within in the rectangle with
vertices $1-a-iT^{\prime }$, $2-iT^{\prime }$, $2-it_{0}$ and $1-a-it_{0}$.
As in \cite{Luo05}, it easily is shown that

\vskip .10in
$$
I_{13}=-\underset{1-a}{\overset{2}{\int }}\arg Z_M(\sigma -iT^{\prime
})d\sigma +O(1)=O\left(\underset{1-a\leq \sigma \leq 2}{\max }\left\vert \log
Z_M(\sigma -iT^{\prime })\right\vert \right).
$$

\vskip .10in
\noindent
From
$$
\log Z_M(\sigma -iT^{\prime })=\log Z_M(2-iT^{\prime })-\underset{\sigma
-iT^{\prime }}{\overset{2-iT^{\prime }}{\int }}\frac{Z_M^{\prime }}{Z_M}(\xi
)d\xi,
$$

\vskip .10in
\noindent
and the bound in \eqref{BoundDer}, which we write as

\vskip .10in
$$
\frac{Z_M^{\prime }}{Z_M}(\alpha \pm iT^{\prime })=O((T \log T) ^{2-2\alpha })\,\,\,\,\,
\textrm{\rm for $1-a \leq \alpha < 1/2$,}
$$

\vskip .10in
\noindent
we obtain the expression

\vskip .10in
\begin{equation}
I_{13}=\underset{t_{0}}{\overset{T^{\prime }}{\int }} \log \vert Z_M
(1-a-it) \vert dt = O((T \log T)^{2-2(1-a)})=O((T \log T)^{2a})
\,\,\,\,\,\textrm{\rm as $T \rightarrow \infty$.}  \label{I13}
\end{equation}

\vskip .10in
\noindent
Directly from Lemma \ref{Z1growth}, we have the estimate
\begin{equation}\label{I14}
I_{14}=\underset{t_{0}}{\overset{T^{\prime }}{\int }}\log \left\vert
1+Z_{M,1}\left( 1-(a+it)\right) \right\vert dt=O\left( (T \log T)^{2a}\right)
\,\,\,\,\,\textrm{\rm as $T \rightarrow \infty$.}
\end{equation}

\vskip .10in
\noindent
Combining (\ref{I11}), (\ref{I12}), (\ref{I13}) and (\ref{I14})
yields

\vskip .10in
\begin{equation} \label{I1}
I_{1} =\left( \frac{1}{2}-a\right) \frac{\mathrm{vol}(M)
}{2}T^{2}+\left( \frac{n_{1}}{2}+1\right) T\log T+O((T \log T)^{2a})
+TC_{M,a} \,\,\,\,\,\textrm{\rm as $T \rightarrow \infty$,}
\end{equation}

\vskip .10in
\noindent
where
$$
C_{M,a} = \left( a-1/2\right) \cdot 2n_{1}\log 2+a(\log A_M-c_{1})-\log
\left\vert a_{M}\right\vert +\log \mathrm{vol}(M) -1-\textrm{%
Re}(c_{2})-\frac{n_{1}}{2}\log \pi -\frac{n_{1}}{2}.
$$

\vskip .10in
Finally, we have arrived at our estimate for $I_{1}$.

\vskip .20in
\subsection{Evaluation of $I_4$}

\vskip .10in
The evaluation of $I_4$ closely follows the lines of the proof
that the analogous integral in the compact case considered by Garunk\v{s}tis in \cite{Gar08}.
The new input being our Lemma \ref{derivboundArg}.

\vskip.10in
It is sufficient to prove that

\vskip .10in
\begin{equation} \label{argXbound}
\arg X(\sigma +iT^{\prime })=o(T)
\,\,\,\,\,\textrm{\rm for $a\leq \sigma \leq \sigma_0$ and as $T \rightarrow \infty$.}
\end{equation}

\vskip .10in
\noindent
We first will show that \eqref{argXbound} holds when $1/2-\delta \leq \sigma \leq \sigma_0$ for some small
$\delta$.  The argument is similar to the proof of the formula (3.4) in \cite{Gar08}. However, in order to keep
the exposition self-contained, we will repeat main steps of the proof.

\vskip .10in
Let $\delta >0$ be a small constant to be chosen later and let $N_{\delta}$ denote the number of zeros of
$\textrm{\rm Re} (X_M(\sigma + iT'))$ with $1/2 -\delta \leq \sigma \leq \sigma_0 $.  Then $\vert \arg X_M (\sigma + iT')\vert \leq \pi (N_{\delta}+1)$. In order to estimate $N_{\delta}$, consider the function

\vskip .10in
$$
g(z):= 1/2 (X_M(z+iT') + X_M(z-iT')).
$$

\vskip .10in
\noindent
Let $n(r, \sigma_0)$ denotes the number of zeros of $g(z)$ inside the circle $\vert z- \sigma\vert \leq r$.
Observe that

\vskip .10in
$$
g(\sigma) = \Re (X_M (\sigma + iT'))
\,\,\,\,\,\textrm{\rm and}\,\,\,\,\,
\vert \arg X_M (\sigma + iT')\vert \leq \pi ( n (\sigma_0 -1/2 + \delta, \sigma_0)+1).
$$

\vskip .10in
\noindent
By applying Jensen's theorem to $g$, we obtain the equation

\vskip .10in
\begin{equation} \label{Jensen}
\int \limits _{0}^{R} \frac{n(r)}{r} dr = \frac{1}{2 \pi} \int \limits _{0} ^{2 \pi} \log \vert g(\sigma_0 +Re^{i \theta}) \vert d \theta - \log \vert g(\sigma_0) \vert.
\end{equation}

\vskip .10in
\noindent
For large enough $T$ and $R=\sigma_0 -1/2 +2\delta$, Lemma \ref{derivboundArg} implies the bound

\vskip .10in
$$
\log\vert g (\sigma_0 +re^{i \theta})\vert < \left\{
\begin{array}{ll}
 \epsilon T, & \hbox{for } \Re (\sigma_0 + R e^{i \theta}) \geq \frac{1}{2} \\[3mm]
  (2\delta + \epsilon)T, & \hbox{for }  \frac{1}{2} - 2\delta \leq
  \Re (\sigma_0 + R e^{i \theta}) \leq \frac{1}{2}.
  \end{array}
  \right.
$$

\vskip .10in
\noindent
The circle $\vert z- \sigma_0 \vert=R$ has a very small part to the left of the line $\Re (s)=1/2$
namely a circular arc of length $O(\delta^{1/2})$.
Therefore, the right hand side of \eqref{Jensen} is $O(\epsilon T) + O(\delta^{1/2} (2\delta + \epsilon) T)$.
Hence,

\vskip .10in
$$
n(R-\delta)\leq \frac{R}{\delta} \int \limits _{0}^{R} \frac{n(r)}{r} dr = O\left( \left( \frac{\epsilon}{\delta} + \frac{\delta + \epsilon}{\sqrt{\delta}}\right )T\right).
$$

\vskip .10in
For $a\leq \sigma <1/2-\delta$ we use Proposition
\ref{LemmaCritLine} to deduce that $\Re (X(\sigma+iT^{\prime }))
\neq 0$. In particular, we have that $\arg X(\sigma +iT^{\prime })=O(1)$ when
$a\leq \sigma <1/2-\delta$ since $\arg $
is bounded by $\pi $ times the number of zeros.

\vskip .10in
Let us now take $\delta = \epsilon^{2/3}$.  We then get the bound

\vskip .10in
$$
\vert \arg X_M (\sigma + iT')\vert \leq \pi ( O(\epsilon^{1/3} T)+1)+ O(1) =o(T)
\,\,\,\,\,\textrm{\rm as $T \rightarrow \infty$,}
$$

\vskip .10in
\noindent
for all $a \leq \sigma \leq\sigma_0$.  Since $\epsilon >0$ is arbitrary, we conclude
that $I_{4}=o(T)$ as $T \rightarrow \infty$.

\vskip .20in
\subsection{Proof of the Main Theorem}

\vskip .10in
Since $0<a<1/2$, we have that $(T \log T)^{2a} =o(T)$.  We have shown that $I_2$ and $I_3$ are $O(1)$ as
$T \rightarrow \infty$ and that $I_4=o(T)$ as $T \rightarrow \infty$.   Hence, by substituting equation
\eqref{I1} into \eqref{mainsum} we get

\vskip .10in
\begin{equation} \label{finaleq}
2\pi \underset{t_{0}<\gamma <T^{\prime }}{\underset{\rho ^{\prime }=\beta
^{\prime }+i\gamma }{\sum }}\left( \beta ^{\prime }-a\right) =\left( \frac{1%
}{2}-a\right) \frac{\mathrm{vol}(M) }{2}T^{2}+\left( \frac{n_{1}}{2%
}+1\right) T\log T+TC_{M,a}+o(T)
\,\,\,\,\,\textrm{\rm as $T \rightarrow \infty$,}
\end{equation}

\vskip .10in
\noindent
where

$$
C_{M,a} = \left( a-1/2\right) \cdot 2n_{1}\log 2+a(\log A_M-c_{1})-\log \left|
a_{M}\right| +\log \mathrm{vol}(M) -1-\textrm{Re}(c_{2})-\frac{n_{1}}{2%
}\log \pi -\frac{n_{1}}{2}.
$$

\vskip .10in
Substituting $a/2$ instead of $a$ into \ref{finaleq}, subtracting
the obtained formulas, and then dividing by $a/2$
yields the statement b) of the Main Theorem.

\vskip .10in
As for part (c) of the Main Theorem, we begin with the formula

\vskip .10in
\begin{equation} \label{sums_for_c)}
\underset{0<\gamma \leq T}{\underset{\rho ^{\prime }=\beta
^{\prime
}+i\gamma }{\sum }}\left( \beta ^{\prime }-1/2\right) =\underset{%
0<\gamma <T^{\prime }}{\underset{\rho ^{\prime }=\beta ^{\prime
}+i\gamma }{\sum }}\left( \beta ^{\prime }-a\right) +\left(
a-1/2\right) \underset{0<\gamma <T^{\prime }}{\underset{\rho
^{\prime }=\beta ^{\prime }+i\gamma }{\sum }}1\text{.}
\end{equation}

\vskip .10in
\noindent
The first sum on the right hand side of (\ref{sums_for_c)}) is estimated by (\ref{finaleq}).
The second sum on the right hand side of (\ref{sums_for_c)}) is estimated by part b)
of the Main Theorem, keeping mind that the difference between the second sum in (\ref{sums_for_c)}
and the sum in part b) is the finite number of zeros in the half-plane $\textrm{\rm Re}(s) < 1/2$.

\vskip .10in
With all this, the proof of the Main Theorem is complete.

\vskip .10in
In the case when the surface is co-compact the statement of the Main Theorem is easily deduced,
since, in that case $n_1=c_1=c_2=0$, $H_M =1$, $A_{M} = \exp(\ell_{M,0})$ and

\vskip .10in
$$
\frac{\eta_M ^{\prime }}{\eta_M }(s) = \mathrm{vol} (M)
(s-1/2)\tan (\pi (s-1/2))- \pi \underset{0<\theta (R)<\pi
}{\underset{\left\{
R\right\} }{\sum }}\frac{1}{M_{R}\sin \theta }\frac{\cos (2 \theta -\pi)(s-1/2)}{\cos \pi (s-1/2)}.
$$

\vskip .20in
\section{Corollaries of the Main Theorem}

\vskip .10in
In this section we deduce three corollaries of our Main Theorem.
The results we prove are analogous to Theorem 2 and
Theorem 3 in \cite{LM74}, with, in their notation, $k=1$.
Similar results may be deduced for the horizontal distribution of
zeros of the $k$th derivative, based on the results of Section 8, with
suitably replaced constants.

\vskip .10in
\begin{corollary} \label{corCloseTo1/2} For $\delta > 1/2$, let
$N_{\mathrm{ver}}(\delta ,T;\left(
Z_{M}H_{M}\right) ^{\prime })$ denote the number of zeros $\rho' $ of $%
(Z_{M}H_{M})^{\prime }$ such that $\Re (\rho') >\delta $ and $0< \Im (\rho') <T$.
Then, for an arbitrary $\epsilon >0$
\[
N_{\mathrm{ver}}(\frac{1}{2}+\epsilon ,T;\left( Z_{M}H_{M}\right) ^{\prime })<\frac{1%
}{\epsilon }N_{\mathrm{hor}}(T;\left( Z_{M}H_{M}\right) ^{\prime }).
\]
\end{corollary}

\vskip .10in
\begin{proof} Trivially, we have the bounds
\begin{align}
N_{\mathrm{ver}}(\frac{1}{2}+\epsilon ,T;\left( Z_{M}H_{M}\right) ^{\prime })&<\frac{1%
}{1/2+\epsilon }\underset{\sigma >1/2+\epsilon \text{, }0<t<T}{\underset{%
\left( Z_{M}H_{M}\right) ^{\prime }(\sigma +it)=0}{\sum }}\sigma \\
=\frac{1}{1/2+\epsilon }&\underset{\sigma >1/2+\epsilon \text{, }0<t<T}{%
\underset{\left( Z_{M}H_{M}\right) ^{\prime }(\sigma +it)=0}{\sum
}}\left( \sigma -\frac{1}{2}\right) +\frac{1/2}{1/2+\epsilon
}N_{\mathrm{ver}}(\frac{1}{2}+\epsilon ,T;\left( Z_{M}H_{M}\right) ^{\prime }).
\end{align}

\vskip .10in
\noindent
Therefore,

\vskip .10in
\begin{equation*}
\frac{2\epsilon }{1+2\epsilon }N_{\mathrm{ver}}(\frac{1}{2}+\epsilon
,T;\left( Z_{M}H_{M}\right) ^{\prime })<\frac{2}{1+2\epsilon
}N_{\textrm{hor}}(T;\left( Z_{M}H_{M}\right) ^{\prime }),
\end{equation*}

\vskip .10in
\noindent
from which the result immediately follows.
\end{proof}

\vskip .10in
Observe that the lead term in the asymptotic expansion in part (b) of the Main Theorem is
$O(T^{2})$, whereas the lead term in the asymptotic expansion in part (c) of the Main Theorem is
$O(T\log(T))$.  Consequently, Corollary \ref{corCloseTo1/2} shows that zeros of $\left( Z_{M}H_{M}\right)
^{\prime }$ are concentrated very close the critical line $\Re (s)=1/2$.
The following corollary further quantifies this observation.

\vskip .10in
\begin{corollary} \label{corrlimitT} For any $\delta >1/2$, let
$N_{\mathrm{ver}}^{-}(\delta ,T;(Z_{M}H_{M})^{\prime })$ denote the
number of non-trivial zeros $\rho =\sigma +it$ of $(Z_{M}H_{M})^{\prime }$
with $\sigma <\delta $ and $0<t<T$. Then, for any constant $\epsilon >0$,

\vskip .10in
$$
\underset{T\rightarrow \infty }{\lim
}\frac{N_{\mathrm{ver}}^{-}(1/2+\epsilon ,T;(Z_{M}H_{M})^{\prime
})}{N_{\mathrm{vert}}(T;(Z_{M}H_{M})^{\prime })}=1.
$$
\end{corollary}

\vskip .10in
\begin{proof} Obviously,
\vskip .10in
$$
1\geq \frac{N_{\mathrm{ver}}^{-}(1/2+\epsilon ,T;(Z_{M}H_{M})^{\prime })}{%
N_{\mathrm{vert}}(T;(Z_{M}H_{M})^{\prime })}=1-\frac{N_{\mathrm{ver}}(1/2+\epsilon
,T;(Z_{M}H_{M})^{\prime })}{N_{\mathrm{vert}}(T;(Z_{M}H_{M})^{\prime })}+\frac{O(1)}{%
N_{\mathrm{vert}}(T;(Z_{M}H_{M})^{\prime })}\text{,}
$$

\vskip .10in
\noindent
where the last term represents the contribution from at most
finitely many
non-trivial zeros of $(Z_{M}H_{M})^{\prime }$ in the half-plane Re $(s)<1/2$.
Corollary \ref{corCloseTo1/2} implies that

\vskip .10in
\begin{equation} \label{limitOfN/N}
1\geq \frac{N_{\mathrm{ver}}^{-}(1/2+\epsilon ,T;(Z_{M}H_{M})^{\prime })}{%
N_{\mathrm{vert}}(T;(Z_{M}H_{M})^{\prime })}>1-\frac{1}{\epsilon }\frac{%
N_{\textrm{hor}}(T;(Z_{M}H_{M})^{\prime })}{N_{\mathrm{vert}}(T;(Z_{M}H_{M})^{\prime
})}.
\end{equation}

\vskip .10in
\noindent
From the Main Theorem b) and c) we deduce that

\vskip .10in
$$
\frac{%
N_{\textrm{hor}}(T;(Z_{M}H_{M})^{\prime })}{N_{\mathrm{vert}}(T;(Z_{M}H_{M})^{\prime
})} \rightarrow 0\,\,\,\,\,\textrm{\rm as $T \rightarrow \infty$.}
$$

\vskip .10in
\noindent
Therefore, by passing to the limit as $T\rightarrow \infty $ in \eqref{limitOfN/N},
the claimed result follows.
\end{proof}

\vskip .10in
A result similar to Corollary \ref{corrlimitT}, for the zeros of the %
derivative of the Selberg zeta function associated to a compact Riemann surface is obtained in \cite{Min08b}.

\vskip .10in
The following
corollary gives estimates of short sums of distances $(\sigma
-1/2)$.

\vskip .10in
\begin{corollary} Let $0<U<T$. Then,
\begin{align} \label{shortsum}
2 \pi \underset{\sigma >1/2 \text{,
}T<t \leq T+U}{\underset{\left(
Z_{M}H_{M}\right) ^{\prime }(\sigma +it)=0}{\sum }}\left( \sigma -\frac{1}{2}%
\right) & =\left( \frac{n_{1}}{2}+1\right) U\log (T+U)\\
&+\left( \log \frac{\mathfrak{g}_{1}\mathrm{vol}(M) A_M^{1/2}}{\left|\pi ^{n_{1}/2}\left| d(1)\right| a_{M}\right| }\right) U+o(T)+O\left(U^{2}/T\right)
\,\,\,\,\,\textrm{\rm as $T\rightarrow \infty $.}
\end{align}
\end{corollary}

\vskip .10in
\begin{proof}
The left hand side of the \eqref{shortsum} is equal to
$2\pi ( N_{\mathrm{hor}}(T+U;\left( Z_{M}H_{M}\right) ^{\prime }) - N_{\mathrm{hor}}(T;\left( Z_{M}H_{M}\right) ^{\prime })$, hence part (c) of the Main Theorem yields

\vskip .10in
\begin{align}
2 \pi \underset{\sigma >1/2 \text{,
}T<t \leq T+U}{\underset{\left(
Z_{M}H_{M}\right) ^{\prime }(\sigma +it)=0}{\sum }}\left( \sigma -\frac{1}{2}%
\right) &= \left( \frac{n_{1}}{2}+1\right) \left( T \log \left( 1+\frac{U}{T}\right) -U \right) \\
&+ \left( \log \frac{\mathfrak{g}_{1}\mathrm{vol}(M) A_M^{1/2}}{\pi ^{n_{1}/2}\left| d(1)a_{M}\right| }\right) U+o(T)
\,\,\,\,\,\textrm{\rm as $T \rightarrow \infty$.}
\end{align}

\vskip .10in
The elementary observation that  $ T\log \left( 1+\frac{U}{T}\right) -U = O\left( U^{2}/T\right)$
completes the proof.
\end{proof}

\vskip .20in
\section{Examples}

\vskip .10in
The Main Theorem naturally leads to the following question: Are there
examples of groups $\Gamma$ where $e^{\ell_{M,0}} < \left( \mathfrak{g}_{2}/
\mathfrak{g}_{1}\right) ^{2}$ as well as groups where
$e^{\ell_{M,0}}
> \left( \mathfrak{g}_{2}/ \mathfrak{g}_{1}\right) ^{2}$?  The purpose of this
section is to present examples of groups in each category.  In fact, there are
examples of both arithmetic and non-arithmetic groups in each category.

\vskip .20in
\subsection{Congruence subgroups}

\vskip .10in
Let $\Gamma = \overline{\Gamma_{0}(N)}$ be the congruence
subgroups defined by the arithmetic condition

\vskip .10in
$$
\overline{\Gamma _{0}(N)}:=\left\{ \left(
\begin{array}{cc}
a & b \\
c& d%
\end{array}%
\right) \in \textrm{SL}(2,\mathbb{Z}):c\equiv
0\,\,(\textrm{mod\,\,}N)\right\} \Big/\pm I,
$$

\vskip .10in
\noindent
where $I$ denotes the identity matrix and $N$ is a squarefree, positive integer. If $N=p_1 \cdots p_r$, for distinct primes $p_1$, ..., $p_r$; then, it is proved in \cite{Hej83}, pages 532-538, as well as in \cite{Hu84}, that the corresponding surface has $n_1=2^{r}$ cusps and the scattering determinant is given by the formula

\vskip .10in
$$
\varphi_N (s)= \left[ \sqrt{\pi}\frac{\Gamma(s-1/2)}{\Gamma(s)} \right]^{n_1} \left[ \frac{\zeta_{\mathbb Q}(2s-1)}{\zeta_{\mathbb Q}(2s)} \right ]^{n_1} \prod_{p \mid N} \left( \frac{1-p^{2-2s}}{1-p^{2s}}\right)^{n_1 /2}
$$

\vskip .10in
Now, it is easy to show that $(\mathfrak{g}_{2}/\mathfrak{g}_{1})^{2} = 4$.

\vskip .10in
All elements of $\overline{\Gamma_{0}(N)}$ have
integer entries, so any hyperbolic element has trace whose absolute value is at
least equal to $3$.  Therefore, $e^{\ell_{M,0}} \geq u$ where
$u$ is a solution to $u^{1/2} + u^{-1/2} = 3$.  Solving, we get
that $u = ((3+\sqrt {5})/2)^{2} > 4$.
Therefore, for any such group $\overline{\Gamma_{0}(N)}$,
one has that $e^{\ell_{M,0}} > (\mathfrak{g}_{2}/\mathfrak{g}_{1})^{2}$.

\vskip .10in
Further examples of arithmetic Fuchsian groups where $e^{\ell_{M,0}} >
(\mathfrak{g}_{2}/\mathfrak{g}_{1})^{2}$ are the groups $\overline{\Gamma(N)}$,
where $\Gamma(N)$ denotes the principal congruence subgroup. The scattering determinant
can be computed using the analysis presented in \cite{Hej83} and \cite{Hu84}.  As above,
one shows that $(\mathfrak{g}_{2}/\mathfrak{g}_{1})^{2}=4$ because the Dirichlet series portion
of the scattering determinant is shown to give by ratios of classical Dirichlet series.
Furthermore, the matrices in $\Gamma(N)$ also have integral entries, so
$e^{\ell_{M,0}}\geq ((3+\sqrt(5))/2)^{2} > 4$.

\vskip .20in
\subsection{Moonshine subgroups}

\vskip .10in
We now present an example of a non-compact, arithmetic Riemann surface
where $e^{\ell_{M,0}}<\left( \mathfrak{g}_{2}/\mathfrak{g}_{1}\right)
^{2}$.

\vskip .10in
Following \cite{Ga06a}, we will use the term "moonshine group" for any subgroup
$\Gamma$ of $\textrm{PSL}(2, \mathbb{R})$ which satisfies the
following two conditions. First, there exists an integer $N \geq 1$ such
that $\Gamma$ contains $\overline{\Gamma _{0}(N)}$. Second,
$\Gamma$ contains the element

\vskip .10in
$$
\left(\begin{array}{cc}1 & k \\ 0 & 1 \end{array}\right)
\,\,\,\,\,\textrm{if and only if} \,\,\,\,\, k \in \mathbf{Z}.
$$

\vskip .10in
\noindent
Such groups also appear in \cite{Sh71}, defined as $\overline{\Gamma_0^{+}(N)}:= \overline{\Gamma _0(N)} \cup \overline{\Gamma _0(N)\tau} $, where

\vskip .10in
$$
\tau= \frac{1}{\sqrt{N}} \left(
                           \begin{array}{cc}
                             0 & -1 \\
                             N & 0 \\
                           \end{array}
                         \right).
$$

\vskip .10in
We employ the term ``moonshine groups'' following the discussion
in \cite{CMS04}, \cite{CG97}, \cite{Cum04}, \cite{Cum10} and
\cite{Ga06b}.  Specifically, we refer to Proposition 7.1.2, page
408 of \cite{Ga06b} which cites results from \cite{CMS04}
regarding a classification of genus zero subgroups of $\text{\rm
SL}(2,\mathbb{R})$ which are manifest in the ``monstrous moonshine''
conjectures that were proved by Borcherds.

\vskip .10in
In this section we will examine two of the genus zero ``moonshine groups''
which were determined in \cite{CMS04}, showing that for one group
one has that $e^{\ell_{M,0}} > (\mathfrak{g_{2}}/\mathfrak{g_{1}})^{2}$
and for another one has that $e^{\ell_{M,0}} <
(\mathfrak{g_{2}}/\mathfrak{g_{1}})^{2}$.  In fact, the two groups
have the additional feature of having the same topological signature.

\vskip .10in
Following \cite{Hel66} as well as page 363 of
\cite{Cum04}, let $f$ be a square-free, non-negative integer, and
consider the group

\vskip .10in
$$
\Gamma _{0}(f)^{+}:=\left\{ e^{-1/2}\left(
\begin{array}{cc}
a & b \\
c & d%
\end{array}%
\right) \in \textrm{SL}(2,\mathbb{R}):a,b,c,d,e\in \mathbb{Z}\text{, }e\mid f\text{, }%
e\mid a\text{, }e\mid d\text{, }f\mid c\text{, }ad-bc=e)\right\}
$$

\vskip .10in
\noindent
In \cite{Hel66}, it is proved that if a subgroup $G \subseteq
\textrm{SL} (2,\mathbb{R})$ is commensurable with $\textrm{SL}
(2,\mathbb{Z})$, then there exists a square-free, non-negative
integer $f$ such that $G$ is a subgroup of $\Gamma_{0}(f)^{+}$.
 Lemma 2.20 on page 368 of
\cite{Cum04} proves that the parabolic elements of $\Gamma
_{0}(f)^{+}$ have integral entries.  Therefore, $\Gamma
_{0}(f)^{+}$ is a moonshine group. Let $\overline{\Gamma _{0}(f)^{+}} = \Gamma
_{0}(f)^{+}/\pm I$.  The Riemann surface $\overline{\Gamma
_{0}(f)^{+}} \backslash \mathbb{H}$ has finite volume since the
fundamental domain has smaller area than the  fundamental domain
of $\overline{\Gamma _{0}(f)}$. Since all parabolic elements of $\Gamma_{0}(f)^{+}$
have integral entries and cusps of the corresponding surface are uniquely determined
by the parabolic elements, we conclude that the surface $\overline{\Gamma
_{0}(f)^{+}} \backslash \mathbb{H}$ has at most $n_1=2^{r}$ inequivalent cusps, where
$r$ is the number of prime factors of $f$. The number of inequivalent cusps of the
surface $\overline{\Gamma
_{0}(f)^{+}} \backslash \mathbb{H}$ may be strictly less than $2^{r}$, as we will see
in the following example. Heuristically, this is expected, since the surface
$\overline{\Gamma_{0}(f)^{+}} \backslash \mathbb{H}$ has a "smaller" fundamental polygon.

\vskip .10in
Consider the case when $f=5$. As proved in \cite{Cum04},
the surface $\overline{\Gamma_{0}(5)^{+}} \backslash \mathbb{H}$
has one cusp at $\infty$. The scattering matrix in this case has a single
entry which is given by

\vskip .10in
$$
\Phi_5(s)= \sqrt{\pi} \frac{\Gamma(s-1/2)}{\Gamma(s)}\varphi_{5,\infty \infty}(s),
$$

\vskip .10in
\noindent
for $\Re (s) \gg 0$ and where
$$
\varphi_{5, \infty \infty}(s)= \sum_{c \in C_5} c^{-2s} S(c),
$$

\vskip .10in
\noindent
with

\vskip .10in
$$
C_5= \left\{c>0 : c=5n \text{  or  } c= \sqrt{5}\cdot m,
\text{  for some positive integers  } n, m, \text{  such that } 5\nmid m \right\}
$$

\vskip .10in
\noindent
and $S(c)$ denotes the number of distinct numbers $d\, \mathrm{ mod} (c)$
such that $d$ is the right lower entry of the matrix from $\Gamma_0 (5)^{+}$
whose left lower entry is $c \in C_5$. We refer to \cite{Iw02}, Sections 2.5 and 3.4 for details.
In \cite{JST12} it is proved that

\vskip .10in
$$
\varphi_{5, \infty \infty}(s)= \left( \frac{5^{s}+5}{5^{s}(5^{s}+1)} \right)
\cdot \frac{\zeta_{\mathbb{Q}}(2s-1)}{\zeta_{\mathbb{Q}}(2s)}.
$$

\vskip .10in
\noindent
With all this, one immediately can show that $\left( \mathfrak{g}_{2}/
\mathfrak{g}_{1}\right) ^{2}=4$.

\vskip .10in
It is easily to confirm that

\vskip .10in
$$
\gamma = \left( \begin{array}{cc} 0 & -1/\sqrt{5} \\\sqrt{5} &
\sqrt{5} \end{array}
\right) \in  \Gamma _{0}(5)^{+},
$$
which is seen by taking $e=5$, $a=0$, $b=-1$, and $c=d=5$.  The
trace of $\gamma$ is $\sqrt{5} > 2$, hence $\gamma$ is hyperbolic.
Therefore, $e^{\ell_{M,0}}\leq u$ where $u$ is a positive solution
of $u^{1/2}+u^{-1/2}=\sqrt{5}$.  Solving, we have that $u=\left(
(1+\sqrt{5})/2\right) ^{2}<4$.

\vskip .10in With all this, we have an example of an arithmetic Riemann
surface where $e^{\ell_{M,0}} <
(\mathfrak{g}_{2}/\mathfrak{g}_{1})^{2}$.

\vskip .10in
Since the surface $\overline{\Gamma
_{0}(5)^{+}} \backslash \mathbb{H}$ has a signature (0;2,2,2;1) meaning that its genus is zero, it has three inequivalent elliptic points of order two and one cusp.  With this, a natural question
 to consider is if the inequality $e^{\ell_{M,0}} <
(\mathfrak{g}_{2}/\mathfrak{g}_{1})^{2}$ holds for all surfaces of signature (0;2,2,2;1).
The answer to this question is no, as we will show in the following example.

\vskip .10in
The surface $\overline{\Gamma
_{0}(6)^{+}} \backslash \mathbb{H}$ has the signature (0;2,2,2;1), as shown in Table C of \cite{Cum04}.
The scattering matrix in this case has a single entry which is given by

\vskip .10in
$$
\Phi_6(s)= \sqrt{\pi} \frac{\Gamma(s-1/2)}{\Gamma(s)}\varphi_{6,\infty \infty}(s),
$$

\vskip .10in
\noindent
for $\Re (s) \gg 0$, where

\vskip .10in
$$
\varphi_{6, \infty \infty}(s)= \sum_{c \in C_6} c^{-2s} S(c),
$$

\noindent
with

\vskip .10in
\begin{align}
C_6= &\{c>0 : c=6n \text{  or  } c=3 \sqrt{2}\cdot n, \text{  or  } c= 2\sqrt{3}\cdot n, \notag \\&
\hskip .5in
\text{  or  } c= \sqrt{6}\cdot m \ \text{  for integers  } \ n, m>0,
 \text{  such that } 6\nmid m \}
\end{align}

\vskip .10in
\noindent
and $S(c)$ denotes the number of distinct numbers $d \,\mathrm{ mod} (c)$ such that $d$
is the right lower entry of the matrix from $\Gamma_0(6)^{+}$ whose left lower entry
is $c \in C_6$. Obviously, $\mathfrak{g}_{1}= \sqrt{6}$ and $\mathfrak{g}_{2} =2 \sqrt{3}$
are two smallest elements of $C_6$, hence $\left( \mathfrak{g}_{2}/%
\mathfrak{g}_{1}\right) ^{2}=2$. On the other hand, traces of matrices from $\Gamma
_{0}(6)^{+}$ belong to the set
$T=\{ (a+d) \cdot h, \text{  where  } a,d \in \mathbb{Z} \text{  and  }
h \in \{ 1, \sqrt{2}, \sqrt{3}, \sqrt{6}\}\}$. Therefore,
$\min \{ \vert \mathrm{Tr} A \vert : A \in \mathcal{H}(\Gamma_0 (6)^{+})\} = \sqrt{6}$,
hence $e^{\ell_{M.0}} \geq u$ where $u>1$ is a solution of the equation
$u^{1/2} + u^{-1/2} = \sqrt{6}$. Since $u=((\sqrt{6} +\sqrt{2})/2)^{2} >2$,
we see that the surface $\overline{\Gamma
_{0}(6)^{+}} \backslash \mathbb{H}$ is an example of the surface where $e^{\ell_{M,0}} >
(\mathfrak{g}_{2}/\mathfrak{g}_{1})^{2}$.

\vskip .10in
More generally,
if $f \geq 5$ is a prime number, then the surface $\overline{\Gamma
_{0}(f)^{+}} \backslash \mathbb{H}$ has one cusp at infinity and the corresponding set
 $C$ has the form $C= \{c>0 : c=f\cdot n \text{  or  } c= \sqrt{f}\cdot n, \
 \text{  for some integer  } n>0 \}$, hence $\left( \mathfrak{g}_{2}/%
\mathfrak{g}_{1}\right) ^{2}=4$ for all surfaces $\overline{\Gamma
_{0}(f)^{+}} \backslash \mathbb{H}$, where $f\geq 5$ is prime. Furthermore, if $f\geq7$ is a
prime, then, the group $\Gamma _{0}(f)^{+}$ does not contain a hyperbolic element whose
corresponding geodesic has hyperbolic length less than $\log 4$. So, groups $\Gamma _{0}(f)^{+}$
with $f\geq 7$ prime yield to yet another example of the Riemann surface where $e^{\ell_{M,0}} >
(\mathfrak{g}_{2}/\mathfrak{g}_{1})^{2}$.

\vskip .10in
We refer the reader to the article \cite{JST12} for detail concerning the computations of the
scattering matrices for the ''moonshine'' groups discussed in this section, as well as
further analytic and numerical investigations of the distribution of eigenvalues of the Laplacian.

\vskip .20in
\subsection{On existence of surfaces where $e^{\ell_{M,0}} = (\mathfrak{g}_{2}/\mathfrak{g}_{1})^{2}$}

\vskip .10in
We now argue the existence of an abundance of surfaces for which $e^{\ell_{M,0}} < (\mathfrak{g}_{2}/\mathfrak{g}_{1})^{2}$.

\vskip .10in
Let $M_{\tau}$ denote a degenerating family of Riemann surfaces, parameterized
by the holomorphic parameter $\tau$, which approach
the Deligne-Mumford boundary of moduli space when $\tau$ approaches zero.
One can select distinguished points of $M_{\tau}$ which are either removed or whose
local coordinates $z$ are replaced by fractional powers $z^{1/n}$.  By doing so, one obtains a
degenerating sequence of hyperbolic
Riemann surfaces of any signature; we refer the reader to \cite{HJL97} and references therein for
further details regarding
the construction of the sequence of degenerating hyperbolic Riemann surfaces.

\vskip .10in
By construction, the length of the smallest geodesic on $M_{\ell}$ approaches zero, so then
$\exp(\ell_{M_{\tau},0})$ approaches one
as $\tau$ approaches zero.  In \cite{GJM08}, the authors prove that through degeneration,
parabolic Eisenstein series on $M_{\tau}$ converge
to parabolic Eisenstein series on the limit surface; see part (ii) of the Main Theorem on
page 703 of \cite{GJM08}.  To be precise, one needs
that the holomorphic parameter $s$ of the parabolic Eisenstein series lies in the half-plane
$\textrm{Re}(s) > 1$ and the spacial parameter $z$
to lie in a bounded region of $M_{\tau}$. However, in these ranges, one can compute the
scattering matrix by computing the zeroth Fourier coefficient
of the parabolic Eisenstein series, and, subsequently, compute the ratio
$\mathfrak{g}_{2}/\mathfrak{g}_{1}$ on $M_{\tau}$.  Since the parabolic
Eisenstein series converge through degeneration to the parabolic Eisenstein series the limit
surface, the associated scattering matrix converges to
a submatrix $\Phi$ of the full scattering matrix on the limit surface.  Clearly, the determinant
of $\Phi$ can be decomposed into a product of Gamma functions
and a Dirichlet series, where the Dirichlet series is such that
$\mathfrak{g}_{2}/\mathfrak{g}_{1} > 1$.

\vskip .10in
Therefore, we conclude that for all $\tau$ sufficiently close to zero, we have that
$e^{\ell_{M_{\tau},0}} < (\mathfrak{g}_{2}/\mathfrak{g}_{1})^{2}$.  In
fact, all surfaces near the Deligne-Mumford boundary of any given moduli space satisfy the inequality $e^{\ell_{M_{\tau},0}} < (\mathfrak{g}_{2}/\mathfrak{g}_{1})^{2}$.

\vskip .10in
Additionally, let us assume that one is considering a moduli space which contains a congruence
subgroup so then there exists a surface where
$e^{\ell_{M_{\tau},0}} > (\mathfrak{g}_{2}/\mathfrak{g}_{1})^{2}$.  Then by combining the
above argument with the computations from section 7.1, there exists surfaces for which
$e^{\ell_{M_{\tau},0}} = (\mathfrak{g}_{2}/\mathfrak{g}_{1})^{2}$.
However, we have not been successful in our attempts to explicitly construct such a surface.
In a sense, our Main Theorem shows that surfaces for which
$e^{\ell_{M_{\tau},0}} = (\mathfrak{g}_{2}/\mathfrak{g}_{1})^{2}$
have a larger number of zeros of $(Z_{M}H_{M})'$ than nearby surfaces for which the
inequality holds.

\vskip .20in
\section{Higher derivatives}

\vskip .10in
In this section, we will outline the proof of the Main Theorem
for higher order derivatives of $Z_{M}H_{M}$. The results
are analogous to theorems proved for the zeros of the higher
order derivatives of the Riemann zeta function; see \cite{Ber70} and
\cite{LM74}.

\vskip .20in
\subsection{Preliminary lemmas on higher derivatives}

\vskip .10in
In order to deduce the vertical and horizontal
distribution of zeros of the higher order derivatives of $(Z_M
H_M)$ we prove some preliminary lemmas, analogous to lemmas in Section 4.

\vskip .10in
\begin{lemma}
Let $f_M(s)$ be defined by \eqref{definitOf f_M} and
$\widetilde{Z} _{M}(s)$ defined by \eqref{ZtildaDef}.  Let us define, inductively,
the functions $\widetilde{Z}_{M,j}(s)$ as $\widetilde{Z} _{M,0}(s):= Z_M (s)$,
$\widetilde{Z} _{M,1}(s):= \widetilde{Z} _{M}(s)$ and, for $j \geq 2$,

\vskip .10in
\begin{equation} \label{defZ_M,i_tilda}
\widetilde{Z}_{M,j}(1-s) = \frac{1}{f_M(s)}\left( (j-1)\frac{f_M^{\prime }}{f_M}(s)+%
\frac{\eta_M ^{\prime }}{\eta_M }(s)-\frac{K_M^{\prime }}{K_M}(s)-\overset{j-1}{%
\underset{i=0}{\sum }}\frac{\widetilde{Z}_{M,i}^{\prime }}{\widetilde{Z}_{M,i}}%
(1-s)\right).
\end{equation}

\vskip .10in
\noindent
Then for every positive integer $k$ the $k$th
derivative of the function $Z_M H_M$ can be represented as

\vskip .10in
\begin{equation}  \label{Z_M H_M(n)}
\left( Z_{M}H_{M}\right) ^{(k)}(s)=\left( f_M(s)\right) ^{k}\eta_M
(s)K_M^{-1}(s)Z_{M}(1-s)\overset{k}{\underset{i=1}{\prod
}}\widetilde{Z} _{M,i}(1-s).
\end{equation}
\end{lemma}

\vskip .10in
\begin{proof}
The statement is true for $k=1$, by Lemma
\ref{lemmaFuctEq} and the definition of the function
$\widetilde{Z} _{M,1}(s)$. Assume that \eqref{Z_M H_M(n)} holds
true for all $1\leq j\leq k$. By computing the derivative of formula
\eqref{Z_M H_M(n)}, we get the expression

\vskip .10in
\begin{align}
\left( Z_{M}H_{M}\right) ^{(k+1)}(s)&=\left( f_M (s)\right)
^{k}\eta_M
(s)K_M^{-1}(s)Z_{M}(1-s)\overset{k}{\underset{i=1}{\prod
}}\widetilde{Z} _{M,i}(1-s) \notag \\&
\cdot  \left[ k\frac{f_M^{\prime }}{f_M}(s)+%
\frac{\eta_M ^{\prime }}{\eta_M }(s)-\frac{K_M^{\prime }}{K_M}(s)-\overset{k}{%
\underset{i=0}{\sum }}\frac{\widetilde{Z}_{M,i}^{\prime }}{\widetilde{Z}_{M,i}}%
(1-s) \right]\label{deriv1}\\&=\left( f_M (s)\right) ^{k+1}\eta_M
(s)K_M^{-1}(s)Z_{M}(1-s)\overset{k+1}{\underset{i=1}{\prod
}}\widetilde{Z} _{M,i}(1-s),\label{deriv2}
\end{align}

\vskip .10in
\noindent
where we have used the definition of $\widetilde{Z}_{M,k+1}$ (\ref{defZ_M,i_tilda}) to
go from (\ref{deriv1}) to (\ref{deriv2}).
\end{proof}

\vskip .10in
\begin{lemma} \label{Z1derivbound}
For $j \geq 1$, let $Z_{M,j}(s):= \widetilde{Z}_{M,j}(s)-1$.
For small $\delta > 0$ and $\delta_{1} > 0$, let $\sigma_1$ be a real number
such that $\sigma_{1}\geq 1/2+\delta_1 >1/2$ and
$(\sigma_1 \pm iT)$ is away from circles of a fixed, small radius $\delta>0$,
centered at integers.  Then for $k=0,1$

\vskip .10in
\begin{equation} \label{Z_M,iDeriv}
Z_{M,j}^{(k)}(\sigma_1 \pm iT)=O\left( \frac{(T \log
T)^{2-2\sigma_1} \log^{k} T}{(\sigma_1 -1/2)T}\right)
\,\,\,\,\,\textrm{as $T \rightarrow \infty$,}
\end{equation}

\vskip .10in
\noindent
and

\vskip .10in
\begin{equation} \label{ZtildaM,iDeriv}
\frac{\widetilde{Z}_{M,j}'}{\widetilde{Z}_{M,j}}(\sigma_1 \pm iT)= O\left(
\frac{(T \log T)^{2-2\sigma_1} \log T}{(\sigma_1 -1/2)T}\right)
\,\,\,\,\,\textrm{as $T \rightarrow \infty$.}
\end{equation}
\end{lemma}

\begin{proof}
We will prove the statement by induction in $j\geq 1$. When $j=1$,
we use formula \eqref{Z1def}, which we differentiate,
use the bound on the growth of the derivative of the
digamma function (see formula 6.4.12. in \cite{AbSt64})
and the bound \eqref{BoundDer} with $k=0$ or $k=1$.  These computations,
which are elementary, allows one to prove \eqref{Z_M,iDeriv}
for $\sigma_1 \geq 1/2 + \delta_1>1/2$ in the case when $j=1$.
In addition, by writing

\vskip .10in
\begin{equation*}
\frac{\widetilde{Z}_{M,1}'}{\widetilde{Z}_{M,1}}(\sigma_1 \pm
iT)=\frac{Z_{M,1}'(\sigma_1 \pm iT)}{1+Z_{M,1}(\sigma_1 \pm iT)}=O\left( \frac{(T
\log T)^{2-2\sigma_1} \log T}{(\sigma_1 -1/2)T}\right)
\,\,\,\,\,\textrm{\rm as $T \rightarrow \infty$.}
\end{equation*}

\vskip .10in
\noindent
With all this, we have proved \eqref{ZtildaM,iDeriv} for $j=1$.

\vskip .10in
Assume now that \eqref{Z_M,iDeriv} and
\eqref{ZtildaM,iDeriv} hold true for all $1\leq m \leq j$. Then,
by \eqref{defZ_M,i_tilda} we get

\vskip .10in
$$
1+Z_{M,j+1}(s)=\widetilde{Z}_{M,j+1}(s)=1+Z_{M,j}(s)+\frac{1}{f_M(s)}\left( \frac{f_M'}{f_M}(s)-
\frac{\widetilde{Z}_{M,k\j}'}{\widetilde{Z}_{M,j}}(s) \right).
$$

\vskip .10in
\noindent
Therefore, by the inductive assumption on
$\widetilde{Z}_{M,j}' /\widetilde{Z}_{M,j}$ and $Z_{M,j}$, we have

\vskip .10in
$$
Z_{M,j+1}(\sigma_1 \pm iT)= O \left( \frac{(T \log T)^{2-2\sigma_1} \log T}{(\sigma_1
-1/2)T}\right)
\,\,\,\,\,\textrm{\rm as $T \rightarrow \infty$.}
$$

\vskip .10in
In other words, \eqref{Z_M,iDeriv} holds true with $m=j+1$.  In addition,

\vskip .10in
$$
\frac{\widetilde{Z}_{M,j+1}'}{\widetilde{Z}_{M,j+1}}(\sigma_1 \pm
iT)=\frac{Z_{M,j+1}'(\sigma_1 \pm iT)}{1+Z_{M,j+1}(\sigma_1 \pm iT)}=O\left(
\frac{(T \log T)^{2-2\sigma_1} \log T}{(\sigma_1 -1/2)T}\right)
\,\,\,\,\,\textrm{\rm as $T \rightarrow \infty$.}
$$

\vskip .10in
\noindent
In other words, \eqref{ZtildaM,iDeriv} holds
true for $m=j+1$, which completes the proof.
\end{proof}

\vskip .10in
For any integer $k \geq 2$, let us define

\vskip .10in
$$
a_{M,k}:= (-1)^{k-1} a_{M} \log ^{k-1} A_{M}
$$

\vskip .10in
\noindent
where we set $a_{M,1}:= a_M$.
The analogue of the function $X_M (s)$, defined by \eqref{definofX_M}, is

\vskip .10in
\begin{equation} \label{definofX_M,k}
X_{M,k}(s):=\frac{A_M^{s}}{a_{M,k}}(Z_M H_M)^{(k)}(s).
\end{equation}

\vskip .10in
\noindent
where, of course, $X_{M,1}(s)=X_M(s)$.

\vskip .10in
\begin{lemma} \label{X_M,k boundInfty}
For any integer $k\geq 1$, there exists constants $\sigma _{k} > 1$  and $0<c_{\Gamma,k }<1$
such that for all $\sigma =\Re(s)\geq \sigma _{k}$,

\vskip .10in
$$
X_{M,k}(s)=1+O(c_{\Gamma,k }^{\sigma })\neq 0%
\,\,\,\,\,\textrm{\rm as $\sigma \rightarrow +\infty$.}
$$
\end{lemma}

\vskip .10in
\begin{proof} For $k=1$, the statement is Lemma \ref{X_MboundInfty}. Furthermore, from the proof of Lemma \ref{X_MboundInfty} and the definition of constants $A_M$ and $a_{M,1}$, we see that, for $\Re (s) \gg 0$

\vskip .10in
$$
\mathcal{D}_1(s):= \underset{\left\{ P\right\}  \in \mathcal{H}(\Gamma)}{\sum }\frac{\Lambda (P)}{
N(P)^{s}}+\sum_{i=1}^{\infty}\frac{b\left( q_{i}\right) }{q_{i}^{s}} = \frac{a_{M,1}}{A_M^{s}} \left( 1+O \left(\frac{1}{A_{\Gamma, 1}^{\Re (s) }}\right)\right)
\,\,\,\,\,\textrm{\rm as $\Re (s) \rightarrow +\infty$,}
$$

\vskip .10in
\noindent
for some constant $A_{\Gamma,1} >1$. Therefore,

\vskip .10in
\begin{equation} \label{DerivDirichlet}
\left( Z_M H_M\right) ^{\prime }(s)=Z_M(s)H_M(s)\mathcal{D}%
_{1}(s),
\end{equation}

\vskip .10in
\noindent
where $\mathcal{D}_1(s)$ is a Dirichlet series, converging absolutely for Re $%
s>\sigma _{1}$, for sufficiently large $\sigma_{1}$,
with the leading term equal to $a_{M,1}\cdot A_{M}^{-s}$ as $\Re (s) \rightarrow +\infty$.

\vskip .10in
Let us define, for $k \geq 1$ and $\Re (s) \gg 0$

\vskip .10in
$$
\left( Z_M H_M\right) ^{(k)}(s)=Z_M(s)H_M(s)\mathcal{D}_{k}(s).
$$

\vskip .10in
\noindent
We claim that $\mathcal{D}_{k}(s)$ is a
Dirichlet series with the leading term equal to $a_{M,k}\cdot A_M^{-s}$
as $\Re (s) \rightarrow +\infty$ .
The statement is obviously true for $k=1$. We assume that it is
true for all $1\leq j \leq k$. Differentiating the above equation
for $\Re (s) \gg 0$, we get, from \eqref{DerivDirichlet}

\vskip .10in
$$
\left( Z_M H_M\right) ^{(k+1) }(s)=\left( Z_M H_M\right) (s)\left( \mathcal{D}_k(s)\mathcal{D}
_{1}(s)+\mathcal{D}_{k}^{\prime }(s)\right) :=\left( Z_M H_M\right) (s)\mathcal{D}_{k+1}(s).
$$

\vskip .10in
\noindent
Since $\mathcal{D}_{k}(s)$ is a Dirichlet series with with the
leading term equal to $a_{M,k}\cdot A_M^{-s}$ as $\Re s \rightarrow +\infty$, we see that $\mathcal{D}_k(s)\mathcal{D}
_{1}(s)$ is a Dirichlet series with leading term
equal to $(a_{M,k}a_{M,1})\cdot (A_M^{2})^{-s}$
as $\Re s \rightarrow +\infty$,
while $\mathcal{D}_{k}^{\prime }(s)$ is a Dirichlet series
with the leading term equal to $(- a_{M,k} \log (A_M))\cdot A_M^{-s}$ as $\Re s \rightarrow +\infty$.
By the definition of $A_M$, it is obvious that $A_M >1$, hence $(A_M^{2})^{-s} < A_M^{-s}$
for $\Re (s) \gg 0$. Therefore, $\mathcal{D}_{k+1}(s)$ is a Dirichlet series with the
leading term equal to $(- a_{M,k} \log (A_M))\cdot A_M^{-s}$ as $\Re s \rightarrow +\infty$.
By the definition of coefficients $a_{M,k}$ we have that $a_{M,k+1}=- a_{M,k} \log (A_M)$,
hence the inductive proof is complete.

\vskip .10in
For $\Re(s) = \sigma \gg 0$, we may write
$$
\left( Z_M H_M\right) ^{(k)}(s)=Z_M(s)H_M(s) \frac{a_{M,k}}{A_M ^{s}}
\left( 1 + O(A_{\Gamma, k} ^{- \sigma}) \right )
\,\,\,\,\,\textrm{\rm as $\Re (s) \rightarrow \infty$.}
$$

\vskip .10in
\noindent
Since

\vskip .10in
$$
Z_M(s)=1+O \left(\frac{1}{N(P_{00})^{\Re (s)}}\right )  \text{  and   } H_M(s)=
1+O \left(\frac{1}{r_{1}^{2\Re (s)}} \right)
\,\,\,\,\,\textrm{\rm as $\Re (s) \rightarrow \infty$,}
$$

\vskip .10in
\noindent
there exists $\sigma_k \geq1$ and a constant $C_{\Gamma ,k}>1$ such that for $\Re s > \sigma_k$, we have

\vskip .10in
$$
(Z_M H_M)^{(k)}(s)=\frac{a_{M,k}}{A_M^{s}} \left[ 1+O \left(\frac{1}{C_{\Gamma,k}^{\Re (s)}}\right )\right]
\,\,\,\,\,\textrm{\rm as $\Re (s) \rightarrow \infty$.}
$$

\vskip .10in
\noindent
Setting $c_{\Gamma ,k}=1/C_{\Gamma ,k}$ completes the proof.
\end{proof}

\vskip .10in
\begin{lemma} \label{derivboundArgFork}
For arbitrary $\epsilon>0$, $t \geq 1$ and $\sigma_2\geq 1$
such that $-\sigma_2$ is not a pole of $(Z_MH_M)$ we have, for any positive integer $k$,

\vskip .10in
$$
\left( Z_M H_M\right) ^{(k) }(\sigma +it)=\left\{
\begin{array}{ll}
O\left( \exp \epsilon t\right), &\text{\,\,\, for \,\,\,}\frac{1}{2}\leq
\sigma \leq \sigma _{0} \\[3mm]
O\left( \exp (1/2-\sigma +\epsilon)t\right) &\text{\,\,\, for \,\,\,}-\sigma _{2}\leq \sigma
< 1/2,%
\end{array}%
\right.
$$
as $t \rightarrow \infty$.
\end{lemma}

\vskip .10in
\begin{proof}
When $k=1$, the statement is proved in Lemma \ref{derivboundArg}. Assume that the
statement of Lemma holds for an integer $k\geq 1$. Then for $ 1/2\leq \Re (s)= \sigma \leq \sigma _{0}$
the Cauchy integral formula yields

\vskip .10in
$$
(Z_M H_M)^{(k+1)}(s)=\frac{1}{2\pi i}\underset{C}{\int }\frac{\left( Z_M H_M\right)^{(k)}
(z)}{\left( z-s\right) ^{2}}dz
$$

\vskip .10in
\noindent
where $C$ is a circle of a small, fixed radius $r<\epsilon$, centered at $s$.
Using the inductive assumption on $\left( Z_M H_M\right)^{(k)}(z)$, we then get the bounds

\vskip .10in
$$
\left( Z_M H_M\right) ^{(k+1) }(\sigma +it) =O\left(\exp((r+\epsilon)t)/r\right) =
O\left(\exp(2\epsilon t)\right),
$$

\vskip .10in
\noindent
for $1/2 \leq \sigma \leq \sigma_0$ and $t\geq 1$. This proves the first part of Lemma for $\left( Z_M H_M\right)^{(k+1)}(z)$, hence, the first part of the Lemma holds true for all $k\geq1$.

\vskip .10in
In the case when $\sigma < 1/2$, we employ the functional equation \eqref{Z_M H_M(n)} for $(Z_M H_M)^{(k)}$
to deduce that

\vskip .10in
\begin{align}
&\vert \left( Z_{M}H_{M}\right) ^{(k+1)}(-\sigma_2 +it)\vert =
\vert \left( Z_{M}H_{M}\right) ^{(k)}(-\sigma_2 +it)\vert \notag \\ & \hskip .5in
\cdot  \left | \left[ k\frac{f^{\prime }}{f}(-\sigma_2 +it)+%
\frac{\eta_M ^{\prime }}{\eta_M }(-\sigma_2 +it)-\frac{K_M^{\prime }}{K_M}(-\sigma_2 +it)-\overset{k}{%
\underset{i=0}{\sum }}\frac{\widetilde{Z}_{M,i}^{\prime }}{\widetilde{Z}_{M,i}}%
(1+\sigma_2-it) \right] \right| . \notag
\end{align}

\vskip .10in
\noindent
Since $\sigma_2 \geq 1$, we have $Z_M '/Z_M (1+\sigma_2 - it) =O(1)$ as $t \rightarrow + \infty$. Furthermore, formula \eqref{ZtildaM,iDeriv} and the same computations as in the proof of Lemma \ref{derivboundArg} imply that

\vskip .10in
$$
k\frac{f^{\prime }}{f}(-\sigma_2 +it)+%
\frac{\eta_M ^{\prime }}{\eta_M }(-\sigma_2 +it)-\frac{K_M^{\prime }}{K_M}(-\sigma_2 +it)-\overset{k}{%
\underset{i=0}{\sum }}\frac{\widetilde{Z}_{M,i}^{\prime }}{\widetilde{Z}_{M,i}}%
(1+\sigma_2-it)=O(t) \,\,\,\,\,\textrm{\rm as $t \rightarrow \infty$,}
$$

\vskip .10in
\noindent
since the leading term in the above expression is
$\mathrm{vol} (M) (1/2 + \sigma_2 -it) \tan (\pi (1/2 +\sigma_2 -it))$.
By the inductive assumption on $(Z_M H_M)^{(k)} (-\sigma_2 +it)$, we get

\vskip .10in
$$
\left\vert \left( Z_M H_M\right) ^{(k+1)}(-\sigma _{2}+it)\right\vert
=O\left( \exp \left( \left( \frac{1}{2}+\sigma _{2}+\epsilon \right)
\mathrm{vol}(M) t\right) \right) ,\ \ \text{as \ }t\rightarrow
\infty .
$$

\vskip .10in
\noindent
As in the proof of Lemma \ref{derivboundArg}, one applies
the Phragmen-Lindel\"{o}f theorem to the function $(Z_M H_M)^{(k+1)}$
in the open sector bounded by the lines $\Im (s)=1$, $\Re (s)= -\sigma_2$ and $\Re (s)=1/2$.
As a result, the proof of the second
part of the Lemma is complete for $(Z_M H_M)^{(k+1)}$.
\end{proof}

\vskip .20in
\subsection{Distribution of zeros of $(Z_M H_M)^{(k)}$}

\vskip .10in
The following theorem is the analogue of the Main Theorem for
zeros of higher derivatives of $(Z_M H_M)$.

\vskip .10in
\begin{theorem} \label{higherDerTh}
With the notation as above, the
following statements are true for any positive integer $k$.

\vskip .10in
a) For $\sigma <1/2$, there exist $t_0 >0$ such that $\left( Z_{M}H_{M}\right)
^{(k)}(\sigma +it)\neq 0$ for all $\left|
t\right|>t_{0}$.

\vskip .10in
b) \begin{equation} \label{vertDisKthDer}
N_{\mathrm{ver}}(T;(Z_{M}H_{M})^{(k)})=N_{\mathrm{ver}}(T;(Z_{M}H_{M})')
+ o(T)\,\,\,\,\,\text{\rm as $T \rightarrow \infty$.}
\end{equation}

\vskip .10in
c) \begin{align} \label{horDisKthDer}
N_{\mathrm{hor}}(T;(Z_{M}H_{M})^{(k)})&=N_{\mathrm{hor}}(T;(Z_M
H_{M})') + \frac{(k-1)T}{2 \pi} \left[ \log \left( T \cdot
\mathrm{vol}(M)\right)
-1\right] \notag \\ &
- \frac{T}{2 \pi} \log ((k-1) \log A_M) +o(T)
\,\,\,\,\,\text{\rm as $T \rightarrow \infty$.}
\end{align}
\end{theorem}

\vskip .10in
\begin{proof}
We first outline the proof of part (a).  For $k\geq 2$,  $\sigma <1/2$ and $s=\sigma \pm
iT$ equation \eqref{Z_M H_M(n)} yields

\vskip .10in
\begin{align} \label{ZH_kthderiv}
\frac{\left( Z_{M}H_{M}\right) ^{(k)}}{\left( Z_{M}H_{M}\right) ^{(k-1)}}%
(s)&=\log \left( \left( Z_{M}H_{M}\right) ^{(k-1)}(s)\right) ^{\prime } \notag \\&
=(k-1)\frac{f^{\prime }}{f}(s)+\frac{\eta_M ^{\prime }}{\eta_M }(s)-\frac{%
K_M^{\prime }}{K_M}(s)-\frac{Z_{M}^{\prime }}{Z_{M}}(1-s)-\overset{k-1}{\underset%
{i=1}{\sum }}\frac{\widetilde{Z}_{M,i}^{\prime
}}{\widetilde{Z}_{M,i}}(1-s)
\end{align}

\vskip .10in
\noindent
We now apply \eqref{ZtildaM,iDeriv} with $\sigma_1 = 1 - \sigma >1/2$ and \eqref{BoundDer} to deduce that

\vskip .10in
$$
\frac{Z_{M}^{\prime }}{Z_{M}}(1-s)+\overset{k-1}{\underset%
{i=1}{\sum }}\frac{\widetilde{Z}_{M,i}^{\prime
}}{\widetilde{Z}_{M,i}}(1-s)=O\left(
\frac{(T \log T)^{2\sigma} \log T}{(1/2-\sigma)}\right)\,\,\,\,\,\text{\rm as $T \rightarrow \infty$.}
$$

\vskip .10in
\noindent
Since $\Re (\eta_M'/ \eta_M (\sigma \pm iT)) = - \mathrm{vol}(M)t + O(\log t)$ and $K_M' /K_M (\sigma \pm it) =O(\log t)$ as $t \rightarrow +\infty$, we immediately deduce from \eqref{ZH_kthderiv} that

\vskip .10in
\begin{equation*}
\textrm{Re}\left( -\frac{\left( Z_{M}H_{M}\right) ^{(k)}}{\left(
Z_{M}H_{M}\right) ^{(k-1)}}(\sigma \pm it)\right) =
\mathrm{vol}(M) t+O\left( \max \left \{ \log t, \frac{( t \log
t)^{2\sigma} \log t}{(1/2-\sigma)} \right \}\right)
\,\,\,\,\,\textrm{\rm as $t \rightarrow +\infty$,}
\end{equation*}%

\vskip .10in
\noindent
for any $\sigma<1/2$. This proves part a).

\vskip .10in
The proof of parts b) and c) closely follows
lines of the proof of parts b) and c) of the Main Theorem. We fix
a large positive number $T$ and choose number $T'$ to be a bounded distance from $T$ such that
$T'$ is distinct from the imaginary part of any zero of $Z_M H_M$. We fix a number
$a \in (0,1/2)$ and use part a) of the Theorem to choose $t_{0}>0$ to be the number such that
$(Z_M H_M)^{(k)}(\sigma + it) \neq 0$ for all $\sigma \leq a$ and
$\vert t \vert > t_0$. Let $\sigma_0 $ be a constant such that $\sigma_0 \geq \max\{ \sigma_0' , \sigma_k\}$, where
$\sigma_0'$ is defined in Lemma \ref{SeriesRepLogDerZ_M H_M} and
$\sigma_k$ is defined in Lemma \ref{X_M,k boundInfty}.

\vskip .10in
We apply Littlewood's theorem to the function $X_{M,k}(s)$,
defined by \eqref{definofX_M,k} which is holomorphic in the rectangle
$R(a,T^{\prime })$ with vertices $a+it_{0}$, $\sigma _{0}+it_{0}$,
$\sigma _{0}+iT^{\prime }$, $a+iT^{\prime }$.  The resulting formula is

\vskip .10in
\begin{align} \label{mainsum_k}
2 \pi \underset{t_{0}<\gamma^{(k) } <T^{\prime }, \text{
 } \beta^{(k)} >a}{\underset{\rho ^{(k) } =\beta
^{(k) }+i\gamma^{(k) } }{\sum }}\left( \beta ^{(k)}-a\right)
&=\underset{t_{0}}{\overset{T^{\prime }}{\int }}\log \left\vert X_{M,k}(a+it)\right\vert dt-%
\underset{t_{0}}{\overset{T^{\prime }}{\int }}\log \left\vert X_{M,k}(\sigma
_{0}+it)\right\vert dt \\ &
-\underset{a}{\overset{\sigma _{0}}{\int }}\arg X_{M,k}(\sigma +it_{0})d\sigma +%
\underset{a}{\overset{\sigma _{0}}{\int }}\arg X_{M,k}(\sigma +iT^{\prime
})d\sigma \\&=I_{1,k}+I_{2,k}+I_{3,k}+I_{4,k},
\end{align}

\vskip .10in
\noindent
where $\rho ^{(k)}$ denotes the zero of $(Z_M H_M)^{(k)}$. By the
choice of $t_0$, the sum on the left-hand side of
\eqref{mainsum_k} is actually taken over all zeros $\rho ^{(k)}$
of $(Z_M H_M)^{(k)}$ with imaginary part in the interval
$(t_0,T')$.

\vskip .10in
Trivially, $I_{3,k}=O(1)$ as $T \rightarrow + \infty$. The application of
Lemma \ref{X_M,k boundInfty} immediately yields that
$I_{2,k}=O(1)$ as $T \rightarrow + \infty$, once we apply the
same method as in evaluation of $I_2$.

\vskip .10in
One can follow the
steps of the proof that $ \vert \arg X_M(\sigma + iT') \vert
=o(T)$ as $T \rightarrow +\infty$ in the present setting.  One uses function $X_{M,k}$
instead of $X_M$ and Lemma \ref{derivboundArgFork}
instead of Lemma \ref{derivboundArg}.  From this,  we deduce that $\vert \arg
X_{M,k}(\sigma + iT') \vert =o(T)$ as $T \rightarrow + \infty$.
Therefore, it is left to evaluate $I_{1,k}$.

\vskip .10in
From definition of $X_{M,k}$, using the functional equation
\eqref{Z_M H_M(n)} for $(Z_M H_M)^{(k)}$, we get for $k \geq2$, the expression

\vskip .10in
\begin{multline} \label{I_1,k}
I_{1,k} = \underset{t_{0}}{\overset{T^{\prime }}{\int }}\log
\left\vert A_M ^{(a+it)}a_{M,k}^{-1}\right\vert dt + k
\underset{t_{0}}{\overset{T^{\prime }}{\int }} \log \vert
f_M(a+it) \vert dt + \underset{t_{0}}{\overset{T^{\prime }}{\int
}} \log \vert \eta_M (a+it) \vert dt + \\
+ \underset{t_{0}}{\overset{T^{\prime }}{\int }} \log \vert
K_M^{-1} (a+it) \vert dt + \underset{t_{0}}{\overset{T^{\prime
}}{\int }} \log \vert Z_M (1-a-it) \vert dt + \sum_{i=1}^{k-1}
\underset{t_{0}}{\overset{T^{\prime }}{\int }} \log \vert 1 +
Z_{M,i} (1-a-it) \vert dt.
\end{multline}

\vskip .10in
\noindent
By employing equation \eqref{Z_M,iDeriv} with $j=0$, we get

\vskip .10in
$$
\log \vert 1+ Z_{M,i} (1-a-it) \vert = O(\vert Z_{M,i} (1-a-it) \vert)=O(t^{2a-1}
(\log t)^{2a} ) \,\,\,\,\,\text{\rm as $t \rightarrow \infty$}
$$

\vskip .10in
\noindent
and for all $i=1,...,k-1$.  Hence,

\vskip .10in
$$
\underset{t_{0}}{\overset{T^{\prime }}{\int }} \log \vert 1+ Z_{M,i}
(1-a-it) \vert dt = O((T \log T)^{2a})
\,\,\,\,\,\text{\rm as $T \rightarrow \infty$}
$$

\vskip .10in
\noindent
and for all $i=1,...,k-1$. Substituting this equation, together
with \eqref{Intf_m}, \eqref{I122}, \eqref{I121} and \eqref{I13}
into \eqref{I_1,k}, we immediately deduce that

\vskip .10in
\begin{equation}
I_{1,k}= \left( \frac{1}{2} -a\right)
\frac{\mathrm{vol}(M)}{2}T^{2} + \left( \frac{n_1}{2} +k \right) T
\log T + C_{M,a,k} T   + O((T \log
  T)^{2a}) \,\,\,\,\,\text{\rm as $T \rightarrow \infty$,}\notag
\end{equation}

\vskip .10in
\noindent
where

\vskip .10in
\begin{align}
C_{M,a,k} &= 2\left (a-\frac{1}{2} \right)n_1 \log 2 +
a \log A_M - \log \vert
a_{M,k} \vert  \notag
\\&+k (\log (\mathrm{vol}(M))-1) + 2a \log \mathfrak{g}_1 -
  \log \vert d(1) \vert -\frac{n_1}{2} (\log \pi +1).\notag
\end{align}

\vskip .10in
\noindent
Combining this equation with the
bounds on $I_{2,k}$, $I_{3,k}$ and $I_{4,k}$ and
\eqref{mainsum_k}, we get

\vskip .10in
\begin{equation} \label{mainsum_k_evaluated}
 2 \pi \underset{t_{0}<\gamma^{(k) } <T^{\prime
}}{\underset{\rho ^{(k) }=\beta ^{(k) }+i\gamma^{(k) } }{\sum }}\left(
\beta ^{(k)}-a\right) =\left( \frac{1}{2} -a\right)
\frac{\mathrm{vol}(M)}{2}T^{2} + \left( \frac{n_1}{2} +k \right) T
\log T + C_{M,a,k} T  + o(T)
\,\,\,\,\,\textrm{\rm as $T \rightarrow \infty$.}
\end{equation}

\vskip .10in
Replacing $a$ by $a/2$ in
\eqref{mainsum_k_evaluated} and subtracting proves part b). Part
c) is proved by employing an analogue of equation
\eqref{sums_for_c)}, with $\beta'$ and $\rho'$ replaced by
$\beta^{(k)}$ and $\rho^{(k)}$.
\end{proof}

\vskip .10in
\begin{remark}\rm
The statement of Theorem \ref{higherDerTh} is true in the case of co-compact Riemann surfaces
$\Gamma \setminus \mathbb{H}$ when taking $H_M =1$ and $A_{M} = \exp(\ell_{M,0})$ in \eqref{vertDisKthDer}
and \eqref{horDisKthDer}.

\vskip .10in
In the case when $\Gamma \setminus \mathbb{H}$ is compact the statement b) of
Theorem \ref{higherDerTh} was announced by Luo in \cite{Luo05}, with the weaker error term
$O(T)$.  As one can see, we put consider effort into the analysis yielding the error term $o(T)$,
and the structure of the constant $C_{M,a,k}$ is, in our opinion, fascinating.
\end{remark}

\vskip .10in
\begin{remark} \rm
From the formula \eqref{horDisKthDer} for the horizontal distribution of zeros of $(Z_MH_M)^{(k)}$,
 we see that the differentiation of $(Z_MH_M)^{(k)}$ increases the sum $N_{\mathrm{hor}}(T; (Z_MH_M)^{(k)})$
 by  the quantity $ [( \mathrm{vol}(M)/ 2 \pi) \cdot T \log T - O(T)] $ as $T \rightarrow \infty$.
 Hence, after each differentiation, zeros of $(Z_MH_M)'$ move further to the right of $1/2$. Since
 every zero of $(Z_MH_M)'$ on the line $\Re (s)=1/2$ (up to finitely many of them) is a multiple
 zero of $Z_M$,  this result fully supports the "bounded multiplicities conjecture". To recall, the
 "bounded multiplicities conjecture" asserts that the order of every multiple zero
 of $Z_M$ is uniformly bounded, or, equivalently, that the dimension of every
 eigenspace associated to the discrete eigenvalue of the Laplacian on $M$ is uniformly bounded,
 with a bound depending solely upon $M$.
\end{remark}

\vskip .20in
\section{Concluding Remarks}

\vskip .10in
\subsection{Revisiting Weyl's law}

\vskip .10in
Weyl's law for an arbitrary finite volume hyperbolic Riemann
surface $M$ is the following asymptotic formula, which we quote
from \cite{Hej83}, p. 466:

\vskip .10in
\begin{equation}\label{Weyl}
N_{M,\textrm{dis}}(T) + N_{M,\textrm{con}}(T) =
\frac{\textrm{vol}(M)}{4\pi}T^{2} - \frac{n_{1}}{\pi} T \log T +
\frac{n_1 T}{\pi}(1-\log 2) + O\left(T/\log T\right) \,\,\,\,\,\textrm{as $T
\rightarrow \infty$},
\end{equation}

\vskip .10in
\noindent
where

\vskip .10in
$$
N_{M,\textrm{dis}}(T) = \#\{s = 1/2 + it | Z_{M}(s) = 0
\,\,\,\,\,\textrm{and}\,\,\,\,\,0\leq t \leq T\}
$$

\vskip .10in
\noindent
and
\vskip .10in
$$
N_{M,\textrm{con}}(T) = \frac{1}{4\pi}
\int\limits_{-T}^{T}\frac{-\phi'_{M}}{\phi_{M}}(1/2+it)dt.
$$

\vskip .10in
\noindent
For a proof
of Weyl's law, we refer to Theorem 2.28 on page 466 of
\cite{Hej83} and Theorem 7.3 of \cite{Ve90}.

\vskip .10in The term $N_{M,\textrm{dis}}(T)$ counts the number of
zeros of the Selberg zeta function $Z_{M}(s)$ on the critical line
$\textrm{Re}(s) = 1/2$, whereas the term $N_{M,\textrm{con}}(T)$
is related to the number of zeros of $Z_{M}(s)$ off the critical
line but within the critical strip.  In the following proposition,
we will relate the counting function
$N_{\textrm{ver}}(T;\phi_{M})$ with the function $N_{M,\textrm{con}}(T)$, showing that the constant
$\mathfrak{g}_{1}$ appears in the resulting asymptotic formula.

\vskip .10in
\begin{proposition} \label{WeylNew} There exists a sequence
$\{T_{n}\}$ of positive numbers tending toward infinity such that,
with the notation as above, we have the asymptotic formula
\end{proposition}
$$
N_{\textrm{ver}}(T_{n};\phi_{M}) = N_{M,\textrm{con}}(T_{n})
-\frac{\log \mathfrak{g}_{1}}{\pi }T_{n}+O\left(\log
T_{n}\right) \,\,\,\,\,\textrm{as $n \rightarrow \infty$.}
$$

\vskip .10in
\begin{proof} Let $R(T)$ denote the rectangle with vertices
$1/2-iT$, $\sigma _{0}^{\prime }-iT$, $\sigma _{0}^{\prime }+iT$,
$1/2+iT$, where $\sigma _{0}^{\prime }>\sigma_0$, where$\sigma_0$ is defined in section 1.4..
Therefore, the series

\vskip .10in
$$
\frac{H_{M}^{\prime }}{H_{M}}\left( s\right) =\sum_{i=1}^{\infty}\frac{b\left( q_{i}\right) }{q_{i}^{s}}
$$

\vskip .10in
\noindent
converges uniformly and absolutely for Re $s \geq \sigma _{0}^{\prime
}$,and all zeros of $\phi_M$ with real part greater than 1/2
lie inside $R(T)$. Recall that the zeros of $\phi_{M}$
appear in pairs of the form $\rho$ and $\overline{\rho}$.  As a
result, the proposition will follow by studying the expression

\vskip .10in
$$
2N_{\textrm{ver}}(T;\phi_{M})=\frac{1}{2\pi i}\underset{R(T)}{\overset{}{\int }}%
\frac{\phi_M ^{\prime }}{\phi_M }(s)ds.
$$

\vskip .10in
\noindent
Let us write

\vskip .10in
$$
2N_{\textrm{ver}}(T;\phi_{M})=-\frac{1}{2\pi }\underset{-T}{\overset{T}{\int }}%
\frac{\phi_{M} ^{\prime }}{\phi_{M} }(\frac{1}{2}+it)dt+\frac{1}{2\pi }\underset{-T}{%
\overset{T}{\int }}\frac{\phi_{M} ^{\prime }}{\phi_{M} }(\sigma
_{0}^{\prime }+it)dt+I_{1}(T)+I_{2}(T)
$$

\vskip .10in
\noindent
where $I_{1}$ and $I_{2}$ denote the integrals along the
horizontal lines which bound $R(T)$.  Theorem 7.1 on page 412 of
\cite{JL93} proves that $\phi_{M}$ is of regularized product type
with order $M=0$.  As a result, from Chapter 1, section 4 of
\cite{JL94}, we have the existence of a sequence of real numbers
$\{T_{n}\}$ tending to infinity such that

\vskip .10in
$$
I_{1}(T_{n}) =O(\log T_{n}) \,\,\,\,\, \textrm{and}\,\,\,\,\,
I_{2}(T_{n}) =O(\log T_{n}) \,\,\,\,\, \textrm{when $n \rightarrow
\infty$,}
$$

\vskip .10in
\noindent
so then

\vskip .10in
$$
2N_{\textrm{ver}}(T_n;\phi_{M})=-\frac{1}{2\pi }\underset{-T_n}{\overset{T_n}{\int }}%
\frac{\phi_{M} ^{\prime }}{\phi_{M} }(\frac{1}{2}+it)dt+\frac{1}{2\pi }\underset{-T_n}{%
\overset{T_n}{\int }}\frac{\phi_{M} ^{\prime }}{\phi_{M} }(\sigma
_{0}^{\prime }+it)dt +O(\log T_{n}) \,\,\,\,\, \textrm{when $n
\rightarrow \infty$.}
$$

\vskip .10in
\noindent
Using the notation as above, we now write

\vskip .10in
\begin{multline*}
\underset{-T_n}{\overset{T_n}{\int }}\frac{\phi_{M} ^{\prime
}}{\phi_{M} }(\sigma
_{0}^{\prime }+it)dt=\underset{-T_n}{\overset{T_n}{\int }}\frac{H_{M}^{\prime }}{H_{M}}%
(\sigma _{0}^{\prime }+it)dt+\underset{-T_n}{\overset{T_n}{\int }}\frac{%
K_{M}^{\prime }}{K_{M}}(\sigma _{0}^{\prime }+it)dt\\ =
\sum_{i=1}^{\infty}\frac{b\left( q_{i}\right) }{q_{i}^{\sigma
_{0}^{\prime }}}\underset{-T_n}{\overset{T_n}{\int }}\frac{dt}{q_{i}^{it}}%
-4T\log \mathfrak{g}_{1}+n_{1}\underset{-T_n}{\overset{T_n}{\int }}\left( \frac{\Gamma ^{\prime }
}{\Gamma }\left( \sigma _{0}^{\prime }+it-\frac{1}{2}\right)
-\frac{\Gamma ^{\prime }}{\Gamma }\left( \sigma _{0}^{\prime
}+it\right) \right)dt.
\end{multline*}

\vskip .10in
\noindent
Interchanging the sum and the integral above is justified by the fact that the series defining $H_M'/H_M (s)$  converges absolutely and uniformly for $\Re (s) >\sigma_0$.  Furthermore, we also have that

\vskip .10in
$$
\sum_{i=1}^{\infty}\frac{b\left( q_{i}\right) }{q_{i}^{\sigma
_{0}^{\prime }}}\underset{-T_n}{\overset{T_n}{\int
}}\frac{dt}{q_{i}^{it}} =O\left( 1\right)\,\,\,\,\,\textrm{as $n
\rightarrow \infty$,}
$$

\vskip .10in
\noindent
Using the series representation of the digamma function, we see that

\vskip .10in
$$
\frac{\Gamma ^{\prime }%
}{\Gamma }\left( \sigma _{0}^{\prime }+it-\frac{1}{2}\right)
-\frac{\Gamma ^{\prime }}{\Gamma }\left( \sigma _{0}^{\prime
}+it\right) = \sum_{k=0}^{\infty} \frac{-1/2}{(k+\sigma _{0}^{\prime }+it-1/2)(k+\sigma _{0}^{\prime }+it)}
$$

\vskip .10in
\noindent
The series on the right converges uniformly in $t \in [-T_n,T_n]$, hence

\vskip .10in
$$
\underset{-T_n}{\overset{T_n}{\int }}\left( \frac{\Gamma ^{\prime }%
}{\Gamma }\left( \sigma _{0}^{\prime }+it-\frac{1}{2}\right)
-\frac{\Gamma ^{\prime }}{\Gamma }\left( \sigma _{0}^{\prime
}+it\right) \right)dt = \sum_{k=0}^{\infty} \log \left[1- \frac{iT_n}{(k+\sigma _{0}^{\prime })(k+\sigma _{0}^{\prime }-1/2) +T_n^{2} + iT_n/2} \right].
$$

\vskip .10in
\noindent
Since

\vskip .10in
$$
\sum_{k=0}^{\infty}  \left \vert \log \left[1- \frac{iT}{(k+\sigma _{0}^{\prime })(k+\sigma _{0}^{\prime }-1/2) +T^{2} + iT/2} \right] \right \vert \ll \vert T \vert  \underset{0}{\overset{\infty}{\int }} \frac{dx}{x^{2}+T^{2}},
$$

\vskip .10in
\noindent
we get that

\vskip .10in
$$
\underset{-T_n}{\overset{T_n}{\int }}\left( \frac{\Gamma ^{\prime }%
}{\Gamma }\left( \sigma _{0}^{\prime }+it-\frac{1}{2}\right)
-\frac{\Gamma ^{\prime }}{\Gamma }\left( \sigma _{0}^{\prime
}+it\right) \right)dt =O(1)
\,\,\,\,\,\textrm{\rm as $n \rightarrow \infty$.}
$$

\vskip .10in
\noindent
With all this, the proof of the Proposition is complete.
\end{proof}

\vskip .10in
\begin{remark}\rm
The above proposition shows that the term $-(\log
\mathfrak{g}_{1}/\pi)T $ measures the discrepancy between the
number of zeros of $\phi_M $ with real part greater that $1/2$,
meaning $N_{\mathrm{ver}}(T;\phi_{M})$, and the quantity
$N_{M,\mathrm{con}}(T)$, appearing in the classical version of the Weyl's law.

\vskip .10in
Furthermore, one can restate Proposition \ref{WeylNew} as the relation representing Weyl's law

\vskip .10in
\begin{equation} \label{NewWeyl}
N_{\mathrm{ver}}(T_{n};Z_{M}H_M) =
\frac{\mathrm{vol}(M)}{4\pi}T_{n}^{2} - \frac{n_{1}}{\pi} T_{n}
\log T_n + \frac{T_n}{\pi} \left(n_1 (1-\log2) -\log \mathfrak{g}_{1} \right) + O\left(T_{n}/\log T_{n}\right),
\end{equation}

\vskip .10in
\noindent
as $n \rightarrow \infty$.
\end{remark}

\vskip .10in
A direct consequence of the relation \eqref{NewWeyl}, Main Theorem and Theorem \ref{higherDerTh} is the following reformulation of the Weyl's law:

\vskip .10in
\begin{corollary} \label{WeylNewCor}
There exist a sequence $\{T_n \}$ of positive real numbers tending to infinity such that, for every positive integer $k$
$$
N_{\mathrm{ver}}(T_{n};Z_{M}H_M) =N _{\mathrm{ver}}(T_{n};(Z_{M}H_M)^{(k)})- \frac{n_1}{\pi}T_n \log T_n + \frac{T_n}{2\pi} (2n_1 + \log A_M) + o(T_n),  \text{  as  } n \rightarrow \infty.
$$
\end{corollary}

\vskip .10in
\begin{remark}\rm
An interpretation of the constant $\log
\mathfrak{g}_{1}$, similar  to the one derived in Proposition \ref{WeylNew} is obtained in \cite{Fis87}, formula (3.4.15) on page 155, where it is shown, in our notation, that

\vskip .10in
$$
\log\mathfrak{g}_{1}= \lim_{x \rightarrow \infty} \lim_{y \rightarrow \infty} \left[2x \underset{0}{\overset{\infty}{\int }} \frac{N_{M,\mathrm{con}}(t) - N_{\mathrm{ver}}(t;\phi_{M})}{t} \left ( \frac{1}{t^{2}+x^{2}} - \frac{1}{t^{2}+y^{2}} \right) dt \right ].
$$

\vskip .10in
\noindent
A geometric interpretation of the constant $\mathfrak{g}_{1}$, in the case when the surface
has one cusp $\mathfrak{a}$ is derived on page 49 of \cite{Iw02}.  In that case,
$\mathfrak{g}_{1}^{-1}$ is the radius of the largest isometric circle
arising in the construction of the standard polygon for the group $\sigma_{\mathfrak{a}}^{-1} \Gamma \sigma_{\mathfrak{a}}$, where $\sigma_{\mathfrak{a}}$ denotes the scaling matrix of the cusp $\mathfrak{a}$.

\vskip .10in
As stated in the introduction, we do not know of a spectral or geometric interpretation of the constant
$\mathfrak{g}_{2}$, besides the trivial one which realizes $\mathfrak{g}_{2}$ as the second largest
denominator of the Dirichlet series portion of the scattering determinant.
Therefore, we view Corollary \ref{WeylNewCor} as giving rise to a new spectral invariant.
\end{remark}

\vskip .20in
\subsection{A comparison of counting functions}

\vskip .10in
In this section we will prove (\ref{countcomparison}) for $k$th derivatives, where $k\geq 1$.  In effect,
it is necessary to recall results from \cite{Hej83}, translate the
notation in \cite{Hej83} to the notation in the present paper,
then combine the result with (\ref{horSelbergderivcpt}) and parts
(c) of the Main Theorem and Theorem \ref{higherDerTh}.

\vskip .10in
Theorem 2.22 on page 456 of \cite{Hej83} states the
asymptotic relation

\vskip .10in
\begin{equation}\label{Hcount}
N_{\textrm{hor}}(T;H_{M})   =\frac{n_{1}}{2}\cdot \frac{T\log T}{2\pi }+\frac{T}{2}%
b_{6}+O(\log T)\,\,\,\,\,\textrm{as $T \rightarrow \infty$},
\end{equation}

\vskip .10in
\noindent
where the constant $b_{6}$ is given on the first line on page 448, by

\vskip .10in
$$
b_{6}=-\frac{%
n_{1}}{2\pi }-\frac{1}{\pi }\log \left\vert b_{2}\right\vert.
$$

\vskip .10in
\noindent
Note that in \cite{Hej83}, the author counts the zeros of $H_M$ in both
the upper and lower half-planes, whereas the counting function
$N_{\textrm{hor}}(T;H_{M})$ only considers those zeros in the
upper half-plane.  Recall that the zeros and poles of $H$ appear
symmetrically about the real axis.  As a result, the relation
(\ref{Hcount}) differs from Theorem 2.22, page 456 of
\cite{Hej83}, by a factor of two.  We now relate the constant
$b_{2}$ in (\ref{Hcount}) to the notation we employed above.

\vskip .10in  Equation (2.15) on page 445 of \cite{Hej83} gives
$$
\begin{array}{rl}
f_M (s) &\displaystyle =\mathfrak{g}_{1}^{2s-1}\varphi (s)=\mathfrak{g}_{1}^{2s-1}\pi ^{\frac{%
n_{1}}{2}}\left( \frac{\Gamma \left( s-\frac{1}{2}\right) }{\Gamma \left(
s\right) }\right) ^{n_{1}}\cdot \underset{n=1}{\overset{\infty }{\sum }}%
\frac{d(n)}{\mathfrak{g}_{n}^{2s}}\\[5mm]
&\displaystyle =\pi ^{\frac{n_{1}}{2}}\left( \frac{\Gamma \left( s-\frac{1}{2}\right) }{%
\Gamma \left( s\right) }\right) ^{n_{1}}\mathfrak{g}_{1}^{2s-1}\frac{d(1)}{%
\mathfrak{g}_{1}^{2s}}\left( 1+\overset{\infty }{\underset{n=2}{\sum }}\frac{%
a\left( n\right) }{r_{n}^{2s}}\right) \\[5mm] &\displaystyle
=b_{2}\left( \frac{\Gamma \left( s-\frac{1}{2}\right) }{\Gamma
\left(
s\right) }\right) ^{n_{1}}\cdot \underset{n=1}{\overset{\infty }{\sum }}%
\frac{\alpha _{n}}{Q_{n}^{2s}},
\end{array}%
$$

\vskip .10in
\noindent
where the last line is in the notation of page 446, line 5 of
\cite{Hej83} with $m=0$; see also equation (2.16) on page 445 of
\cite{Hej83}.  When comparing with our notation, as defined in
section 1.3 and section 2.3, we have that $Q_{n}=r_{n}$ and, more
importantly,

\vskip .10in
\begin{equation}\label{b2definition}
b_{2}=\pi ^{\frac{n_{1}}{2}}\mathfrak{g}_{1}^{-1}d(1).
\end{equation}

\vskip .10in
\noindent
By substituting (\ref{b2definition}) into (\ref{Hcount}), we are
able to rewrite Theorem 2.22 on page 456 of \cite{Hej83} as

\vskip .10in
\begin{equation}\label{horscattering}
N_{\textrm{hor}}(T;H_{M})  =\frac{n_{1}}{2}\cdot \frac{T\log T}{2\pi }-\frac{T}{2\pi }%
\left( \frac{n_{1}}{2}+\frac{n_{1}}{2}\log \pi +\log \left\vert
d(1)\right\vert -\log \mathfrak{g}_{1}\right) +O(\log
T)\,\,\,\,\,\textrm{as $T \rightarrow \infty$}.
\end{equation}

\vskip .10in
Comparing \eqref{horscattering} with part (c) of Main Theorem  we deduce that

\vskip .10in
$$
N_{\textrm{hor}}(T;(Z_{M}H_{M})') -N_{\textrm{hor}}(T;H_{M})
\displaystyle = \frac{T\log T}{2\pi } +\frac{T}{2\pi }C
+o(T)\,\,\,\,\,\textrm{as $T \rightarrow \infty$}.
$$

\vskip .10in
\noindent
with

\vskip .10in
$$
C = \frac{1}{2}\log A_M-\log \left\vert a_{M}\right\vert +\log
 \textrm{vol}(M) -1.
$$

\vskip .10in
\noindent
Assume that $e^{\ell_{M,0}} < \left(
\mathfrak{g}_{2}/\mathfrak{g}_{1}\right) ^{2}$, so then
$$
A_M =e^{\ell_{M,0}} \,\,\,\,\,\textrm{and}\,\,\,\,\, a_{M} =
\frac{m_{M,0}\ell_{M,0}}{1-e^{-\ell_{M,0}}},
$$

\vskip .10in
\noindent
from which we have that
$$
C = \log \left( \frac{2\textrm{vol}(M)\sinh(\ell_{M,0}/2)}{e \cdot
m_{M,0} \ell_{M,0}}\right).
$$

\vskip .10in
Let $\widetilde{M}$ be any co-compact hyperbolic Riemann
surface such that $\textrm{vol}(\widetilde{M}) = \textrm{vol}(M)$.
Assume that $M$ and $\widetilde{M}$ have systoles of equal length, and the same number of inconjugate
classes of systoles.  Then, using
(\ref{horSelbergderivcpt}), we arrive at the conclusion that

\vskip .10in
\begin{equation}\label{countcomparison2}
N_{\textrm{hor}}(T;(Z_{M}H_{M})') -N_{\textrm{hor}}(T;H_{M}) =
N_{\textrm{hor}}(T;Z'_{\widetilde{M}}) +o(T)
\,\,\,\,\,\textrm{as $T \rightarrow \infty$,}
\end{equation}

\vskip .10in
\noindent
as claimed in (\ref{countcomparison}).

\vskip .10in
Furthermore, when $e^{\ell_{M,0}} < \left(
\mathfrak{g}_{2}/\mathfrak{g}_{1}\right) ^{2}$, comparing \eqref{countcomparison2} with part c) of Theorem \ref{higherDerTh}, for $k\geq 2$ we arrive at

\vskip .10in
\begin{multline*}
N_{\textrm{hor}}(T;(Z_{M}H_{M})^{(k)}) -N_{\textrm{hor}}(T;H_{M})
\displaystyle = N_{\textrm{hor}}(T;Z'_{\widetilde{M}})+ \\
 + \frac{(k-1) T}{2\pi }[\log (T \mathrm{vol}(M)) -1] - \frac{T}{2 \pi} \log ((k-1) \ell_{M,0}) +o(T)
 \,\,\,\,\,\textrm{\rm as $T \rightarrow \infty$.}
\end{multline*}

\vskip .10in
\noindent
Then, from part c) of Theorem \ref{higherDerTh} applied to the zeta function $Z_{\widetilde{M}}$ we deduce
$$
N_{\textrm{hor}}(T;(Z_{M}H_{M})^{(k)}) -N_{\textrm{hor}}(T;H_{M})
\displaystyle = N_{\textrm{hor}}(T;(Z_{\widetilde{M}})^{(k)})
+o(T)\,\,\,\,\,\textrm{as $T \rightarrow \infty$}.
$$

\vskip .10in
\noindent
This proves (\ref{countcomparison}) for higher derivatives.

\vskip .10in
We find the comparison of counting functions, as
summarized in (\ref{countcomparison}) fascinating, especially
since the coefficients in the asymptotic expansions in
(\ref{Hcount}) and part (c) of the Main Theorem are somewhat
involved and dissimilar from other known asymptotic expansions.

\vskip .20in
\subsection{Concluding remarks}

We find the examples provided by studying $\Gamma_{0}^{+}(5)$ and
$\Gamma_{0}^{+}(6)$ particularly fascinating.  In a forthcoming article \cite{JST12}
with Holger Then, we further study various properties of the distribution
of eigenvalues on $\Gamma_{0}^{+}(N)$, for a squarefree, positive integer $N$.  In particular, we prove the groups
$\Gamma_{0}^{+}(5)$ and $\Gamma_{0}^{+}(6)$  have the same signature yet the coefficient
of $T$ in the Weyl's law expansion differs.  As stated above, we view this example as being
in support of the Phillips-Sarnak philosophy since the examples show that the
distribution of eigenvalues depends on more data than just the signature of the
group.

\vskip .10in
In \cite{CMS04} the authors defined $213$ genus zero subgroups of
which $171$ are associated to ``Moonshine''. It would be
interesting to compute the invariant $A_{M}$ for each of these
groups to see if further information regarding the groups,
possibly related to ``moonshine'', is uncovered.

\vskip .10in
Is it possible to explicitly determine an example of a surface where $e^{\ell_{M,0}} = (\mathfrak{g}_{2}/\mathfrak{g}_{1})^{2}$?
More generally, one could study the set of such surfaces, as a subset of moduli space.  Is the set of surfaces where
$e^{\ell_{M,0}} > (\mathfrak{g}_{2}/\mathfrak{g}_{1})^{2}$
a connected subset of moduli space, or are there several components?
Is there another characterization of surfaces where $e^{\ell_{M,0}} = (\mathfrak{g}_{2}/\mathfrak{g}_{1})^{2}$?
Many other basic questions can be easily posed, and we find these problems very interesting.

\vskip .10in
In general, one wonders if there is another interpretation for surfaces such that
$e^{\ell_{M,0}} = (\mathfrak{g}_{2}/\mathfrak{g}_{1})^{2}$.   J. Keating pointed
out that in many classical dynamical systems, one can have complex periodic orbits.  In the setting
of the Selberg trace formula, we were referred to \cite{CV88} and \cite{CV90}.  Perhaps the inequality
$e^{\ell_{M,0}} \neq (\mathfrak{g}_{2}/\mathfrak{g}_{1})^{2}$ is related to the existence of
complex orbits as in \cite{CV88} and \cite{CV90}.  It was also suggested to us that a connection
may exist with Wigner delay times associated to classical scattering problems, which is related
to Keating's remarks.   We thank Keating for allowing us to include his
speculative comments in this article, who wishes to emphasize that his suggestion
should not be viewed as a thought-out conjecture.  Nonetheless, the
authors believe it is important to include his intuitive comments.

\vskip .10in
In \cite{AJS09}, the authors determined the asymptotic behavior of Selberg's zeta function through degeneration
up to the critical line.  It would be interesting to study the asymptotic behavior of the zeros of the derivative
of Selberg's zeta function through degeneration, either in moduli space or through elliptic degeneration.

\vskip .10in
To come full circle, we return to the setting of the Riemann zeta function and speculate if one can attempt to mimic
results which follow from the Levinson-Montgomery article \cite{LM74}.  Specifically, we recall, as stated
in the Introduction, that Levinson used results from the distribution of zeros of $\zeta'_{\mathbb Q}$ to prove
that more than $1/3$ of the zeros of the Riemann zeta function lie on the critical line.  Can one follow a
similar investigation in the setting of the Selberg zeta function?  To do so, we note that a starting point
would be to establish an analogue of the approximate functional equation for the Selberg zeta function.  In the
case when $M$ is compact, the Riemann hypothesis is known for the Selberg zeta function, but if $M$ is not compact,
then one is facing the ideas behind the Phillips-Sarnak philosophy.  Results in this direction would be very
significant, and we plan to undertake the project in the near future.

\vskip .20in \noindent Jay Jorgenson \\ \noindent Department of
Mathematics\\ \noindent The City College of New York \\ \noindent
Convent Avenue at 138th Street \\ \noindent New York, NY 10031 USA
\\ \noindent e-mail: jjorgenson@mindspring.com

\vskip .20in \noindent Lejla Smajlovi\'c \\ \noindent Department
of Mathematics \\ \noindent University of Sarajevo \\ \noindent
Zmaja od Bosne 35, 71, 000 Sarajevo \\ \noindent Bosnia and
Herzegovina\\ \noindent e-mail: lejlas@pmf.unsa.ba

\end{document}